\documentclass[11pt,a4paper,dvipsnames]{article}

\usepackage[top=25mm,bottom=25mm,right=30mm,left=30mm]{geometry}
\usepackage[T1]{fontenc}
\usepackage[utf8]{inputenc}
\usepackage{mathptmx,lmodern,mathrsfs,amsthm,amsmath,graphicx,graphics, float,cases,xfrac,amssymb,bm,bbm,mathtools,gensymb,subcaption,tocloft,listings,pgffor,lipsum,microtype,titling,mathrsfs,proof,sectsty,fancyhdr,longtable,lscape,emptypage,parskip,booktabs,bm,stmaryrd,enumerate,autobreak,enumitem}
\usepackage[labelfont={bf,sf}]{caption}

\usepackage{tikz}
\usetikzlibrary{positioning}
\usetikzlibrary{plotmarks}
\usepackage{pgfplots}
\usepackage{pgfplotstable}
\pgfplotsset{compat=1.7}

\definecolor{color1}{RGB}{68,119,170}   
\definecolor{color2}{RGB}{102,204,238}  
\definecolor{color3}{RGB}{34,136,51}    
\definecolor{color4}{RGB}{17,119,51}    
\definecolor{color5}{RGB}{204,187,68}   
\definecolor{color6}{RGB}{221,204,119}  
\definecolor{color7}{RGB}{204,102,119}  
\definecolor{color8}{RGB}{136,34,85}    
\definecolor{color9}{RGB}{170,51,119}   
\definecolor{color10}{RGB}{102,102,102} 

\usepackage{xargs}
\usepackage{units}

\newcommand{\abs}[1]{\left\lvert#1\right\rvert}
\newcommand\norm[1]{\lVert#1\rVert}
\newcommand\inner[2]{\langle #1, #2 \rangle}
\newcommand\p[1]{\mathord{( #1 )}}
\newcommand\set[1]{\mathord{\left\{ #1 \right\}}}

\newcommand{\reals}{\mathbb{R}}

\newcommand{\N}{\mathbf{N}}
\newcommand{\naturals}{\mathbb{N}_0}

\newcommand{\sym}{\mathbb{S}}

\newcommand{\calH}{\mathcal{H}}

\newcommand{\prox}{{\mathrm{prox}}}

\DeclareMathOperator*{\minimize}{minimize}
\DeclareMathOperator*{\maximize}{maximize}

\DeclareMathOperator*{\argmin}{argmin}
\DeclareMathOperator*{\Argmin}{Argmin}

\DeclareMathOperator*{\zer}{zer}
\DeclareMathOperator*{\dom}{dom}
\DeclareMathOperator*{\Id}{Id}

\newcommand{\xmiddle}[1]{\;\middle#1\;}

\renewcommand{\N}{\mathbb{N}_{0}}
\newcommand{\V}{\mathcal V}

\newcommand{\LF}{L_F}

\newcommand{\QuadMinMaxParam}{\omega}

\newcommandx{\seq}[3][{2=k\in\naturals},{3={}}]{(#1)_{#2}^{#3}}

    \newif\ifshowold\showoldtrue
    \newif\ifshownew\shownewtrue

    \colorlet{newcolor}{black}
    \colorlet{oldcolor}{black!30}
    \newcommand{\disablecolorlinks}{\def\HyColor@UseColor##1{}}

    \newcommand{\new}[2][]{{%
        \ifshownew
            \disablecolorlinks
            \color{newcolor}{}#2%
        \fi
    }}

\makeatother

\allsectionsfont{\normalfont\sffamily\bfseries}

\makeatletter
\def\th@plain{
  \thm@headfont{\normalfont\sffamily\bfseries}
  \itshape 
}

\def\th@definition{
  \thm@headfont{\normalfont\sffamily\bfseries}
  \thm@notefont{\normalfont\sffamily\bfseries}
}

\newtheoremstyle{myStyle1}
  {0.3cm}
  {0.3cm}
  {\itshape}
  {}
  {\normalfont\sffamily\bfseries}
  {:}
  {.5em}
  {}

\newtheoremstyle{myStyle2}
  {0.3cm}
  {0.3cm}
  {}
  {}
  {\normalfont\sffamily\bfseries}
  {:}
  {.5em}
  {}

\makeatother

\providecommand{\keywords}[1]
{\textbf{\textsf{Keywords.}} #1}

\providecommand{\MSC}[1]
{\textbf{\textsf{Mathematics subject classification 2020.}} #1}

\makeatletter
\def\maketag@@@#1{\hbox{\m@th\normalfont\normalsize#1}}
\makeatother

\makeatletter
\newcommand{\pushright}[1]{\ifmeasuring@#1\else\omit\hfill$\displaystyle#1$\fi\ignorespaces}
\newcommand{\pushleft}[1]{\ifmeasuring@#1\else\omit$\displaystyle#1$\hfill\fi\ignorespaces}
\makeatother

\usepackage{hyperref}
\hypersetup{%
    hidelinks,
    hypertexnames = true,
    plainpages    = false,
    colorlinks = true,
    urlcolor   = MidnightBlue,
    linkcolor  = MidnightBlue,
    citecolor  = ForestGreen,
  }

\usepackage[capitalize,nameinlink]{cleveref}
\crefname{assumption}{assumption}{assumptions}
\crefname{enumi}{Assumption}{Assumptions}  
\creflabelformat{enumi}{#2\theassumption.(\roman{enumi})#3}  

\theoremstyle{myStyle1}
\newtheorem{theorem}{Theorem}[section]
\newtheorem{proposition}[theorem]{Proposition}
\newtheorem{corollary}[theorem]{Corollary}
\newtheorem{lemma}[theorem]{Lemma}
\newtheorem{definition}[theorem]{Definition}

\theoremstyle{myStyle2}
\newtheorem{remark}[theorem]{Remark}
\newtheorem{example}[theorem]{Example}

\usepackage{algorithm}
\usepackage{algpseudocode}
\newcommand{\INITIALIZE}{\item[\textbf{Initialize:}]}
\newcommand{\REQUIRE}{\item[\textbf{Require:}]}
\newcommand{\INPUT}{\item[\textbf{Input:}]}
\newcommand{\OUTPUT}{\item[\textbf{Output:}]}

\crefname{ALG@line}{Step}{Steps}
\Crefname{ALG@line}{Step}{Steps}
\renewcommand{\thealgorithm}{\arabic{algorithm}}

\makeatletter
\newcounter{algorithmicH}
\let\OLDalgorithmic\algorithmic
\def\fixalgref{%
  \phantomsection
  \addtocounter{ALG@line}{-1}%
  \refstepcounter{ALG@line}%
}
\let\OLDState\State  \renewcommand{\State}{\OLDState\fixalgref}
\let\OLDIf\If        \renewcommand{\If}[1]{\OLDIf{#1}\fixalgref}
\let\OLDFor\For      \renewcommand{\For}[1]{\OLDFor{#1}\fixalgref}
\let\OLDElse\Else    \renewcommand{\Else}{\OLDElse\fixalgref}
\let\OLDElsIf\ElsIf  \renewcommand{\ElsIf}[1]{\OLDElsIf{#1}\fixalgref}
\let\OLDWhile\While  \renewcommand{\While}[1]{\OLDWhile{#1}\fixalgref}
\renewcommand{\algorithmic}{%
  \setcounter{ALG@line}{0}%
  \stepcounter{algorithmicH}%
  \OLDalgorithmic%
}
\renewcommand{\theHALG@line}{\thealgorithm.\arabic{ALG@line}}

  \newcommand{\algnamefont}{\ttfamily} 
  \newcommand{\FLEX}{{\hyperref[alg:FLEX]{\algnamefont FLEX}}}
  \newcommand{\refFLEX}{\FLEX\ (\Cref{alg:FLEX})}  

  \newcommand{\IFLEX}{{\hyperref[alg:I-FLEX]{\algnamefont I-FLEX}}}
  \newcommand{\refIFLEX}{\IFLEX\ (\Cref{alg:I-FLEX})}

  \newcommand{\ProxFLEX}{{\hyperref[alg:Prox-FLEX]{\algnamefont Prox-FLEX}}}
  \newcommand{\refProxFLEX}{\ProxFLEX\ (\Cref{alg:Prox-FLEX})}

\makeatother

\newenvironment{customalgorithm}[1]{
    \floatstyle{plain}
    \restylefloat{algorithm}
    \floatname{algorithm}{#1}        
    \renewcommand{\thealgorithm}{}   
    \begin{algorithm}
}{
    \end{algorithm}
    \addtocounter{algorithm}{-1} 
    \floatstyle{ruled}
    \restylefloat{algorithm}
}

\definecolor{Red}{rgb}{1, 0, 0}
\definecolor{BlueIMT}{RGB}{0,42,72}
\definecolor{ForestGreen}{rgb}{0.13, 0.55, 0.13}
\definecolor{Gray}{rgb}{0.66, 0.66, 0.66}
\definecolor{Green}{rgb}{0.0, 0.5, 0.0}
\definecolor{MidnightBlue}{rgb}{0.098, 0.098, 0.439}
\definecolor{Orange}{rgb}{0.93, 0.53, 0.18}

\foreach \x/\y/\z in {gray/Gray/gray, blue/Blue/MidnightBlue, orange/Orange/orange, red/Red/red, black/Black/black, green/Green/ForestGreen} {%
    \edef\temp{%
      \noexpand\expandafter\noexpand\gdef\noexpand\csname\x\noexpand\endcsname{\noexpand\color{\z}}%
      \noexpand\expandafter\noexpand\gdef\noexpand\csname\y\noexpand\endcsname{\noexpand\bfseries\noexpand\csname\x\noexpand\endcsname}%
    }\temp
  }

\theoremstyle{myStyle1}
\newtheorem{assumption}{Assumption}

\renewcommand\theassumption{\Roman{assumption}}

\newlist{enumeratass}{enumerate}{1} 
\setlist[enumeratass]{label={(\roman*)}, ref=\theassumption.{(\roman*)}}
\crefname{enumeratassi}{Assumption}{Assumptions}

\newlist{enumeratlem}{enumerate}{1} 
\setlist[enumeratlem]{label={(\roman*)}, ref=\thelemma.{(\roman*)}}
\crefname{enumeratlemi}{Lemma}{Lemmas}

\newlist{thmenumerate}{enumerate}{1}
\setlist[thmenumerate]{label={(\roman*)}, ref=\thetheorem.{(\roman*)}}
\crefalias{thmenumeratei}{theorem}

\newlist{propenumerate}{enumerate}{1}
\setlist[propenumerate]{label={(\roman*)}, ref=\theproposition.{(\roman*)}}
\crefalias{propenumeratei}{proposition}

\Crefname{remark}{Remark}{Remarks}

\numberwithin{equation}{section}

  \setlist[itemize,1]{label=\textbullet,wide,labelwidth=!,labelindent=0pt}
  \SetEnumitemKey{widestL}{
    labelwidth=\eqboxwidth{listlabel@\EnumitemId},
    leftmargin=!,
    format=\mylistlabel,
  }
  \newcommand\mylistlabel[2][l]{\eqmakebox[listlabel@\EnumitemId][#1]{#2}}

  \setlist[enumerate,1]{
    label=\text{\it(\roman*)},
    ref=\text{(\roman*)},
  }

  \newcommand{\createthmlists}[1]{%
    \newlist{#1enumerate}{enumerate}{1}%
    \setlist[#1enumerate,1]{%
      label = \upshape{(\roman*)},%
      ref = \thetheorem(\roman*),%
    }%
    \crefalias{#1enumeratei}{#1}%
    \AtBeginEnvironment{#1}{%
      \expandafter\let\expandafter\enumerate\csname #1enumerate\endcsname
      \expandafter\let\expandafter\endenumerate\csname end#1enumerate\endcsname
    }{}{}%
  }%

  \createthmlists{assumption}
  \createthmlists{corollary}
  \createthmlists{example}
  \createthmlists{fact}
  \createthmlists{lemma}
  \createthmlists{remark}
  \createthmlists{theorem}

\title{\Large \sffamily\bfseries 
A Lyapunov analysis of Korpelevich's extragradient method with fast and flexible extensions
} 
\author{Manu Upadhyaya\(^{\star}\) \and Puya Latafat\(^{\dagger}\) \and Pontus Giselsson\(^{\star}\)}

\date{%
    \textit{ }\\
    \(^{\star}\)Department of Automatic Control\\%
    Lund University, Lund, Sweden\\%
    \{\href{mailto:manu.upadhyaya@control.lth.se}{manu.upadhyaya}, \href{mailto:pontus.giselsson@control.lth.se}{pontus.giselsson}\}@control.lth.se\\[2ex]%
    \(^{\dagger}\)DYSCO (Dynamical Systems, Control, and Optimization)\\%
    IMT School for Advanced Studies Lucca, Lucca, Italy\\%
    \href{mailto:puya.latafat@imtlucca.it}{puya.latafat@imtlucca.it}
}

\begin{document}
\maketitle
 
\begin{abstract}
    \noindent
\new{We develop a Lyapunov-based analysis of Korpelevich’s extragradient method and show that it achieves an \(o(1/k)\) last-iterate convergence rate of the constructed Lyapunov function. This Lyapunov function simultaneously upper bounds several standard measures of optimality, which allows our analysis to sharpen existing last-iterate convergence guarantees for these measures. Moreover, the same analysis enables the design of a class of flexible extensions of the extragradient method in which extragradient steps are adaptively blended with user-specified directions via a Lyapunov-guided line-search procedure. These extensions retain global convergence under practical assumptions and can attain superlinear rates when the directions are chosen appropriately. Numerical experiments confirm the simplicity and efficiency of the proposed framework.}
\end{abstract}

\keywords{Monotone inclusions, extragradient method, Lyapunov analysis, superlinear convergence}

\MSC{%
47J20, 
47J25, 
47J26, 
47H05, 
65K10, 
65K15, 
65J15 
}

\section{Introduction}\label{sec:introduction}
In this work, we consider the inclusion problem
\begin{align}\label{eq:the_inclusion_problem}
\text{find}\  z\in\calH\ \text{ such that }\ 0 \in F(z) + \partial g(z),
\end{align}
where \(F:\calH\to\calH\) is monotone and \(L_{F}\)-Lipschitz continuous for some \(L_{F}\in\reals_{++}\), \(g:\calH\to \reals \cup \{+\infty\}\) is a proper, convex, and lower semicontinuous function, and \((\calH,\inner{\cdot}{\cdot})\) is a real Hilbert space. Inclusion problems of the form~\eqref{eq:the_inclusion_problem} are known as \emph{hemivariational inequalities}~\cite{monteiro2011complexityvariantstsengs} or \emph{(mixed) variational inequalities}~\cite{malitsky2020goldenratioalgorithms,noor1990mixedvariationalinequalities}, and frequently arise in fundamental mathematical programming problems—either directly or through reformulation—including minimization, saddle-point, complementarity, Nash equilibrium, and fixed-point problems~\cite{bauschke2017convexanalysismonotone,facchinei2004finitedimensionalvariationalI}. The most common methods for solving~\eqref{eq:the_inclusion_problem} belong to the large class of extragradient-type methods~\cite{korpelevich1976extragradientmethodfinding,popov1980modificationarrowhurwicz,tseng2000modifiedforwardbackward}. For a recent review, see~\cite{trandinh2024revisitingextragradienttype,trandinh2025acceleratedextragradienttype}. Among these methods, the first and most widely recognized is Korpelevich's extragradient method~\cite{korpelevich1976extragradientmethodfinding}. Although originally proposed for the constrained case in which \(g\) is the indicator function of a nonempty, closed, and convex set, the method also applies to the more general setting in~\eqref{eq:the_inclusion_problem}. Specifically, given an initial point \(z^0\in \calH\) and a step-size parameter \(\gamma \in (0,1/L_{F})\), its iterations are given by
\begin{subequations}\label{alg:EG}
    \begin{align}\label{alg:EG:1}
        \bar{z}^{k}
        ={}&
        \prox_{\gamma g}\left(z^{k} - \gamma F(z^{k})\right), 
        \\ \label{alg:EG:2}
        z^{k+1} 
        ={}& \prox_{\gamma g}\left(z^{k} - \gamma F(\bar{z}^{k})\right)
    \end{align}
\end{subequations}
for each \(k\in\naturals\). A popular alternative is Tseng's forward-backward-forward method~\cite{tseng2000modifiedforwardbackward}, given by
\begin{subequations}\label{alg:Tseng}
    \begin{align}
        \bar{z}^{k} {} = {} \prox_{\gamma g}\left(z^{k} - \gamma F(z^{k})\right), \\
        z^{k+1} {} = {} \bar{z}^{k} + \gamma \p{ F\p{z^{k} } -F\p{\bar{z}^{k}} }
    \end{align}
\end{subequations}
for each \(k\in\naturals\), that requires one less evaluation of the proximal operator \(\prox_{\gamma g}\) per iteration. Classically, the convergence analyses of these methods rely on Fej\'er-type arguments~\cite{korpelevich1976extragradientmethodfinding, tseng2000modifiedforwardbackward}.

In this work, we propose an analysis centered around the Lyapunov function
\begin{align*}      
    \V\p{z^{k},\bar{z}^{k},z^{k+1}} 
    = 2\gamma^{-1}\inner{z^{k} - z^{k+1}}{F\p{z^{k}} - F\p{\bar{z}^{k}}}
    + \gamma^{-2}\norm{z^{k+1} - \bar{z}^{k}}^{2}
    + \gamma^{-2}\norm{z^{k} - z^{k+1}}^{2}.
\end{align*}
For the extragradient method, \(\V_k\) serves as a nonnegative optimality measure for the inclusion problem~\eqref{eq:the_inclusion_problem} as shown in \Cref{prop:V_performance_measure}. \new{Likewise, \(\V_k\) is a nonnegative optimality measure for Tseng's method, since the Lyapunov function reduces to \(\V_k = \gamma^{-2}\norm{z^{k} - \bar{z}^{k}}^{2}\) in this case.} In the particular case when \(g = 0\), both methods are identical, and the Lyapunov function reduces to \(\V_k=\|F(z^k)\|^2\).

Besides being an optimality measure, we show in \Cref{thm:EG:new_Lyapunov_w.o._sol} that \(\V_k\) satisfies a descent inequality for the extragradient method. Moreover, \new{for the extragradient method}, we show a Fejér-type inequality in which \(\V_k\) appears as the residual term (see \Cref{thm:EG:primal_descent_for_V}). \new{By combining this result with the descent property, we establish a \(o(1/k)\) last-iterate convergence rate for \(\V_k\) for the extragradient method, as shown in \Cref{cor:lastiter:T}.} \new{In \Cref{sec:comparison}, we show that \(\V_k\) upper bounds some common optimality measures used to judge the quality of approximate solutions for \eqref{eq:the_inclusion_problem}; hence, together with \Cref{cor:lastiter:T}, the same last-iterate convergence result automatically holds for each of those measures as well. Taken together, the results we obtain for $\V_k$ enable us to recover and, in some cases, extend recent last-iterate convergence-rate results for the extragradient method (cf. \cite[Theorem 3.3]{gorbunov2022extragradientmethodo1/k}, \cite[Theorem 3]{cai2022tightlastiterateconvergenceextragradient}, and \cite[Corollary 4.1]{trandinh2024revisitingextragradienttype}).}

\new{Interestingly, \Cref{thm:EG:new_Lyapunov_w.o._sol} is particular to Korpelevich's extragradient method.} We demonstrate through a simple counterexample in \Cref{ex:counterexample:tseng} that the descent inequality in terms of \(\V_k\) fails for Tseng's method. Moreover, even for the extragradient method, it is crucial to leverage the specific structure of~\eqref{eq:the_inclusion_problem}. Indeed, the claimed descent inequality fails, and even the convergence of the method does not hold if we replace \(\partial g\) with a maximally monotone operator \(T:\calH\to 2^{\calH}\) and correspondingly the proximal operators \(\prox_{\gamma g}\) in~\eqref{alg:EG} with the resolvent \(\p{\Id + \gamma T}^{-1}\). This broader setting is ruled out by a counterexample presented in \Cref{ex:counterexample:full_eg}. \new{However, if \(T\) is also  3-cyclically monotone (see \cite[Definition 22.13]{bauschke2017convexanalysismonotone}), then \Cref{thm:EG:new_Lyapunov_w.o._sol,thm:EG:primal_descent_for_V} (and therefore also \Cref{cor:lastiter:T}) remain valid; see \Cref{rmk:cyclically_monotone} for details.}

The second objective of this work is to develop flexible extragradient-type schemes that accommodate fast local directions while maintaining global convergence. In this regard, the seminal work \cite{solodov1998globallyconvergentinexact} 
\new{proposes a hybrid method for solving monotone equations, i.e., when \(g=0\) in \eqref{eq:the_inclusion_problem}. Their scheme achieves global convergence by blending an inexact regularized Newton step with the hyperplane projection framework from \cite{solodov1999hybridapproximateextragradient}. At each iteration, a search direction is computed based on an inexact regularized Newton step. A line search is then performed along this direction, not to decrease a merit function, but to identify a hyperplane separating the current iterate from the solution set. The algorithm then proceeds by projecting the iterate onto this hyperplane. While this approach incorporates a line search, the convergence analysis still fundamentally} relies on Fej\'er-type monotonicity arguments.
 In practice, however, the projection step can undermine the effectiveness of the Newtonian directions, resulting in slower convergence.


Another related work to ours is~\cite{themelis2019supermannsuperlinearlyconvergent}, which addresses the problem of finding fixed points of averaged operators. They propose a hybrid scheme accelerating many numerical algorithms under the Krasnosel'ski\u{\i}--Mann framework. Similar to~\cite{solodov1998globallyconvergentinexact}, their scheme incorporates a hyperplane projection step and achieves superlinear convergence under suitable assumptions. In addition, it allows for a general class of local directions, including quasi-Newton-type directions, providing greater flexibility in practice.

In contrast to the approaches mentioned earlier that use a line search to identify a separating hyperplane onto which the iterates are projected (see~\cite{solodov1998globallyconvergentinexact, themelis2019supermannsuperlinearlyconvergent}), our proposed schemes incorporate line search procedures grounded in our new Lyapunov analysis, directly aiming to reduce \(\V_k\).  
We introduce three flexible algorithms tailored to specific instances of problem~\eqref{eq:the_inclusion_problem}.~\refFLEX{} is introduced for finding zeros of \(F\),~\refIFLEX{} is applicable when \(F\) is injective, and~\refProxFLEX{} addresses problem~\eqref{eq:the_inclusion_problem} in its full generality. All three algorithms share the same guiding principle: at each iteration, one performs a convex combination of a standard extragradient step~\eqref{alg:EG} and a step based on a user-specified direction. The specific weighting for this convex combination is determined by the line-search procedure, ensuring sufficient descent of the optimality measure \(\V_k\). Similar to~\cite{themelis2019supermannsuperlinearlyconvergent}, our schemes accommodate a wide range of user-chosen directions, including quasi-Newton-type directions. A key feature enabling this approach is that \(\V_k\) depends solely on values computed at each iteration and does not involve a solution to~\eqref{eq:the_inclusion_problem}. This design ensures high flexibility while guaranteeing global (see \Cref{sec:algs}) and superlinear convergence (see \Cref{sec:superlinear}) when choosing suitable directions. 

Our preliminary numerical experiments indicate that using quasi-Newton directions in our proposed algorithms yields favorable performance. In particular, limited-memory type-I and type-II Anderson acceleration exhibit promising results (see \Cref{sec:numerical}). In related work,~\cite{zhang2020globallyconvergenttype} studies Anderson acceleration for finding fixed points of averaged operators, proposing a globalization strategy based on a stabilization and safeguarding mechanism—rather than a line search—that reverts to a nominal Krasnosel’ski\u{\i}–Mann step whenever the Anderson acceleration step fails to sufficiently reduce the forward residual. More recently,~\cite{qu2024extragradientanderson} introduced an extragradient-based scheme with memory-one Anderson acceleration, which reduces overhead and allows for simple, explicit updates of the directions. Furthermore,~\cite{asl2024jsymmetricquasi} presents a quasi-Newton method tailored to minimax problems. Our theory offers a direct globalization strategy for such directions, applicable in the uniformly monotone and injective settings (see \Cref{thm:IFLEX}), or whenever the resulting directions are summable (see \Cref{thm:FLEX:dk_summable}).

\subsection{Organization}\label{sec:organization}%
In \Cref{sec:Lyapunov}, we formally introduce the new Lyapunov analysis for Korpelevich’s extragradient method. Building on this framework, we present the three new algorithms in \Cref{sec:algs} and establish their global convergence under suitable assumptions. In \Cref{sec:global}, we provide detailed proofs of the results from the preceding section. Next, \Cref{sec:superlinear} focuses on superlinear convergence, including corresponding proofs for the proposed algorithms. Numerical experiments appear in \Cref{sec:numerical}, and we conclude in \Cref{sec:conclusions} with a summary of key findings and directions for future research. \new{Finally, \Cref{app:background} offers background material on Korpelevich’s extragradient method, \Cref{app:counterexamples} presents the counterexamples mentioned earlier, and \Cref{sec:comparison} contains a comparison to recent last-iterate convergence results for the extragradient method.}%

\subsection{Notation and preliminaries}\label{sec:notation}%
Let \(\naturals\) denote the set of nonnegative integers, \(\mathbb{N}\) the set of positive integers, \(\mathbb{Z}\) the set of integers, \(\llbracket n,m \rrbracket = \set{l \in \mathbb{Z} \xmiddle| n \leq l \leq m}\) the set of integers inclusively between the integers \(n\) and \(m\),  \(\reals\) the set of real numbers, \(\reals_+\) the set of nonnegative real numbers, \(\reals_{++}\) the set of positive real numbers, \(\reals^n\) the set of all \(n\)-tuples of elements of \(\reals\), \(\reals^{m\times n}\) the set of real-valued matrices of size \(m\times n\), if \(M\in\reals^{m\times n}\) then \([M]_{i,j}\) the \(i,j\)-th element of \(M\), \(\sym^{n}\) the set of symmetric real-valued matrices of size \(n\times n\), and \(\sym_+^n\subseteq \sym^{n}\) the set of positive semidefinite real-valued matrices of size \(n\times n\). Suppose that \( 1 \leq p < +\infty \), \(K\subseteq\naturals\), and \(\mathcal{U} \subseteq \mathcal{W} \), where \(\mathcal{W}\) is a normed space. Then we define the space \( \ell^{p}\p{K;\,\mathcal{U}} = \{(u^{k})_{k\in K} \in \mathcal{U}^{K} \mid \sum_{k\in K} \norm{u^{k}}^{p} < +\infty \} \).

Throughout this paper, \((\calH,\inner{\cdot}{\cdot})\) will denote a real Hilbert space and \(\norm{\cdot}\) the canonical norm, which will be clear from the context. Let \(F:\calH\to\calH\) be an operator. Suppose that \(L_{F}\geq0\). The operator \(F\) is said to be \emph{\(L_{F}\)-Lipschitz continuous} if \(\norm{F\p{x} - F\p{y}} \leq L_{F} \norm{x-y}\) for each \(x,y\in\calH\). The operator \(F\) is said to be \emph{monotone} if \(0\leq \inner{F\p{x} - F\p{y}}{x-y}\) for each \(x,y\in\calH\). Suppose that \(\mu_{F}\geq0\). The operator \(F\) is said to be \emph{\(\mu_{F}\)-strongly monotone} if \(\mu_{F}\norm{x-y}^{2}\leq \inner{F\p{x} - F\p{y}}{x-y}\) for each \(x,y\in\calH\). Moreover, \(F\) is said to be \emph{uniformly monotone with modulus} \(\phi:\reals_{+}\to \reals_{+}\) if \(\phi\) is increasing, vanishes only at \(0\), and \(\phi\p{\norm{x-y}}\leq \inner{F\p{x} - F\p{y}}{x-y}\) for each \(x,y\in\calH\). For a general set-valued operator \(T:\calH \to 2^{\calH}\), the set of \emph{zeros} is denoted by \(\zer\p{T}=\set{x\in\calH \xmiddle| 0 \in T\p{x} }\). The Cauchy–Schwarz inequality states that \(\abs{\inner{x}{y}}\leq\norm{x}\norm{y}\) for each \(x,y\in\calH\) and Young's inequality that \(2\inner{x}{y} \leq \alpha\norm{x}^{2} + \alpha^{-1}\norm{y}^{2}\) for each \(x,y\in\calH\) and \(\alpha>0\).

Given a function \(g:\calH\rightarrow\reals\cup\{+\infty\}\), the \emph{effective domain} of \(g\) is the set \(\dom g = \{x\in\calH \mid g\p{x}<+\infty\}\). The function \(g\) is said to be \emph{proper} if \(\dom g \neq \emptyset\). The \emph{subdifferential} of a proper function \(g\) is the set-valued operator \(\partial g:\calH \rightarrow 2^{\calH}\) defined as the mapping \(x \mapsto \{u \in \calH \mid \forall y \in \calH,\, g\p{y} \geq g\p{x} + \inner{u}{y-x} \}\). The function \(g\) is said to be \emph{convex} if \(g\p{\p{1 - \lambda} x + \lambda y} \leq \p{1 - \lambda} g\p{x} + \lambda g\p{y}\) for each \(x, y \in \calH\) and \(0 \leq \lambda \leq 1\). The function \(g\) is said to be \emph{lower semicontinuous} if \(\liminf_{y\to x} g\p{y} \geq g\p{x}\) for each \(x\in\calH\). If \(C\subseteq\calH\), the \emph{indicator function of \(C\)}, denoted \(\delta_{C}:\calH \to \reals\cup\set{+\infty}\), is defined as \(\delta_{C}(x) = 0\) if \(x\in C\) and \(\delta_{C}(x) = +\infty\) if \(x\in \calH\setminus C\).

Let \(g:\calH\to \reals \cup \{+\infty\}\) be proper, convex and lower semicontinuous, and let \(\gamma>0\). Then the \emph{proximal operator} \(\prox_{\gamma g} : \calH \to \calH \) is defined as the single-valued operator given by
\begin{align*}
    \prox_{\gamma g}\p{x} = \argmin_{z\in\calH}\left(g(z) + \tfrac{1}{2\gamma}\norm{x-z}^2\right) 
\end{align*}
for each \(x\in\calH\)~\cite[Proposition 12.15]{bauschke2017convexanalysismonotone}. If \(x,p\in\calH\), then \(p = \prox_{\gamma g}\p{x}\) \(\Leftrightarrow\) \(\gamma^{-1}\p{x-p} \in \partial g\p{p}\) \(\Leftrightarrow\) \(0 \leq g\p{y} - g\p{p} - \inner{\gamma^{-1}\p{x-p}}{y-p}\) for each \(y\in\calH\)~\cite[Proposition 16.44, Proposition 16.6]{bauschke2017convexanalysismonotone}. 

%

\section{A new Lyapunov analysis}\label{sec:Lyapunov}
Classical convergence analyses of Korpelevich’s extragradient method~\eqref{alg:EG} typically rely on Fejér-type arguments, as discussed in \Cref{app:background}. In this section, we introduce a complementary Lyapunov inequality that not only leads to a last-iterate result but also forms the basis of the new algorithms presented in \Cref{sec:algs}. Throughout this work, we investigate~\eqref{eq:the_inclusion_problem} under the following assumption.

\begin{assumption}\label{ass:main} 
    The following hold in problem~\eqref{eq:the_inclusion_problem}.
    \begin{enumeratass} 
        \item \label{ass:F} \(F:\calH\to\calH\) is monotone and \(L_{F}\)-Lipschitz continuous for some \(L_{F}\in\reals_{++}\).
        \item \label{ass:g}\(g:\calH\to \reals \cup \{+\infty\}\) is proper, convex and lower semicontinuous. 
    \end{enumeratass}
\end{assumption}

Our analysis is centered around the Lyapunov function \(\mathcal{V}:\calH^{3}\to\reals\) given by
\begin{equation}\label{eq:V}
    \mathcal{V}(z, \bar z, z^{+}) = 2\gamma^{-1}\inner{z - z^{+}}{F\p{z} - F\p{\bar{z}}} + \gamma^{-2}\norm{z^{+} - \bar{z}}^{2} + \gamma^{-2}\norm{z - z^{+}}^{2} 
\end{equation}
for each \(\p{z, \bar z, z^{+}} \in \calH^{3}\). \Cref{prop:V_performance_measure} establishes that \(\mathcal{V}\) is generally a valid optimality measure for the inclusion problem in~\eqref{eq:the_inclusion_problem}. For notational convenience, we define the algorithmic operators \(T_{1}^{\gamma}, T_{2}^{\gamma} : \calH \to \calH\) by 
\begin{equation}\label{eq:T1T2}
    T_{1}^{\gamma} = \prox_{\gamma g} \circ (\Id - \gamma F) 
    \quad \text{and} \quad
    T_{2}^{\gamma} = \prox_{\gamma g} \circ (\Id - \gamma F \circ T_{1}^{\gamma}),
\end{equation}
where \(\gamma \in \reals_{++}\) is the step-size parameter. With this notation, the iterates of~\eqref{alg:EG} can be written compactly as \(\bar{z}^k = T_{1}^{\gamma}(z^k)\) and \(z^{k+1} = T_2^{\gamma}(z^k)\).

\begin{proposition}\label{prop:V_performance_measure}
    Suppose that \Cref{ass:main} holds. Let  \(\gamma \in (0, 1/L_F)\), \(z\in \calH\), \(\bar z = T_{1}^{\gamma}(z)\) and \(z^+ = T_{2}^{\gamma}(z)\) where \(T_{1}^{\gamma}\) and \(T_{2}^{\gamma}\) are the algorithmic operators defined in~\eqref{eq:T1T2}, and \(\mathcal{V}\) the Lyapunov function defined in~\eqref{eq:V}. Then the following hold.
    \begin{propenumerate}
        \item \(\mathcal{V}\p{z, \bar z, z^+}  \geq  \p{1- \gamma \LF}\gamma^{-2} \p{\norm{{z}^{+}-\bar{z}}^{2} + \norm{{z}^{+}-{z}}^{2} } \geq 0\).
        \label{prop:V_performance_measure:lower_bound}
        \item \(\mathcal{V}\p{z, \bar z, z^+} = 0\) if and only if \(z = \bar z = z^+ \in \zer\p{F + \partial g }\).\label{prop:V_performance_measure:zero_iff}
    \end{propenumerate}
    \begin{proof}\textit{ }
    \begin{itemize}
            \item[\ref{prop:V_performance_measure:lower_bound}:] 
            The inner product in the definition of \(\V\) can be written as 
            {\begin{allowdisplaybreaks}
            \begin{align*}
                \inner{ {z}-{z}^{+}}{F\p{z}-F\p{\bar{z}}}
                ={}&
                 \inner{ {z}-{z}^{+}}{F\p{z}-F\p{z^{+}}}
                 + 
                 \inner{ {z}-{z}^{+}}{F\p{z^{+}}-F\p{\bar{z}}}
                 \\
                 \geq{}& 
                 \inner{ {z}-{z}^{+}}{F\p{z^{+}}-F\p{\bar{z}}}
                 \\
                 \geq{}&
                 -\norm{{z}-{z}^{+}}\norm{F\p{z^{+}}-F\p{\bar{z}}}
                 \\
                 \geq{}& 
                 - \LF\norm{{z}-{z}^{+}}\norm{{z}^{+}-\bar{z}} \\
                 \geq{}&
                 -\tfrac{\LF}{2} \p{\norm{{z}-{z}^{+}}^{2} + \norm{{z}^{+}-\bar{z}}^{2} },
            \end{align*}
            \end{allowdisplaybreaks}}%
            where monotonicity of \(F\) is used in the first inequality, the Cauchy--Schwarz inequality is used in the second inequality, Lipschitz continuity of \(F\) is used in the third inequality, and Young's inequality for products is used in the fourth inequality. The lower bound of \(\mathcal{V}\) follows from using this inequality in~\eqref{eq:V} and the assumption \(\gamma\LF \in (0,1)\).
            \item[\ref{prop:V_performance_measure:zero_iff}:] Suppose that \(\mathcal{V}\p{z, \bar z, z^+} = 0\). Then \Cref{prop:V_performance_measure:lower_bound} and \Cref{prop:EG:fixed_point:fixed_points_are_solutions} imply that \(z = \bar z = z^+ \in \zer\p{F + \partial g }\). Conversely, suppose that \(z = \bar z = z^+ \in \zer\p{F + \partial g }\). Then it is clear from~\eqref{eq:V} that \(\mathcal{V}\p{z, \bar z, z^+} = 0\). \qedhere
        \end{itemize}
    \end{proof}
\end{proposition}

The following result shows that \(\mathcal{V}\) is, in fact, a suitable Lyapunov function for the extragradient method~\eqref{alg:EG}, i.e., it fulfills a descent inequality. Moreover, the descent inequality neither contains a solution of~\eqref{eq:the_inclusion_problem} nor assumes the existence of a solution. 

\begin{theorem}\label{thm:EG:new_Lyapunov_w.o._sol}
    Suppose that \Cref{ass:main} holds and the sequence \(\p{\p{z^{k},\bar{z}^{k}}}_{k\in\naturals}\) is generated by~\eqref{alg:EG} with initial point \(z^0\in \calH\) and step-size parameter \(\gamma \in \reals_{++}\). Then
    \begin{align}\label{eq:EG:new_Lyapunov_w.o._sol}
        \mathcal{V}_{k+1} \leq \mathcal{V}_{k} - \p{1-\gamma^{2}L_{F}^{2}} \gamma^{-2}\norm{z^{k+1}-\bar{z}^{k}}^{2}
    \end{align}
    for each \(k\in \naturals\), where \(\V_k =\mathcal{V}(z^k, \bar z^k, z^{k+1})\) for the Lyapunov function \(\mathcal V\) defined in~\eqref{eq:V}.

    \begin{proof}
        Note that the first and second proximal steps in~\eqref{alg:EG} can equivalently be written via their subgradient characterization as
        \begin{equation}\label{eq:subgrad:char}
            \gamma^{-1}\p{z^{k}-\bar{z}^{k}}-F(z^{k}) \in \partial g(\bar z^k) \quad \text{ and } \quad 
         \gamma^{-1}\p{z^{k}-z^{k+1}}-F(\bar{z}^{k}) \in \partial g(z^{k+1}),
        \end{equation}
        respectively. Using the subgradient inequality at the points \(z^{k+1}\), \(z^{k+2}\) and \(\bar{z}^{k+1}\), with the particular subgradients given in~\eqref{eq:subgrad:char}, it follows that
        \begin{subequations}
            \begin{align}\label{eq:EG:prox_implicit:k1}
            0 &\geq g\p{z^{k+1}} - g\p{z} + \inner{\gamma^{-1}\p{z^{k} - z^{k+1}} - F\p{\bar{z}^{k}}}{z-z^{k+1}},
            \\ \label{eq:EG:prox_implicit:k2}
            0 &\geq g\p{z^{k+2}} - g\p{z} + \inner{\gamma^{-1}\p{z^{k+1} - z^{k+2}} - F\p{\bar{z}^{k+1}}}{z-z^{k+2}},
            \\ \label{eq:EG:prox_implicit:kbar1}
            0 &\geq g\p{\bar{z}^{k+1}} - g\p{z} + \inner{\gamma^{-1}\p{z^{k+1} - \bar{z}^{k+1}} - F\p{z^{k+1}}}{z-\bar{z}^{k+1}}.
        \end{align}
        \end{subequations}
        holds for any \(z \in \calH\), respectively. Picking \(z = \bar{z}^{k+1}\) in~\eqref{eq:EG:prox_implicit:k1}, \(z = z^{k+1}\) in~\eqref{eq:EG:prox_implicit:k2}, \(z = z^{k+2}\) in~\eqref{eq:EG:prox_implicit:kbar1}, summing the resulting inequalities, and multiplying by \(2\gamma^{-1}\) gives 
        {\begin{allowdisplaybreaks}
        \begin{align} \notag
            0 
            \geq{}& 
            2\gamma^{-2}\inner{ z^{k} - z^{k+1} }{\bar{z}^{k+1}-z^{k+1}}   - 2 \gamma^{-1} \inner{F\p{\bar{z}^{k}} }{\bar{z}^{k+1}-z^{k+1}}  
            \nonumber
            \\
            &{}+ 
            2\gamma^{-2}\norm{ z^{k+1} - z^{k+2} }^{2}  - 2\gamma^{-1}\inner{ F\p{\bar{z}^{k+1}} }{z^{k+1}-z^{k+2}} 
            \nonumber
            \\
            &{}+ 
            2\gamma^{-2}\inner{ z^{k+1} - \bar{z}^{k+1}}{z^{k+2}-\bar{z}^{k+1}} - 2\gamma^{-1} \inner{ F\p{z^{k+1}} }{z^{k+2}-\bar{z}^{k+1}}.
                \label{eq:ApB}
        \end{align}
        \end{allowdisplaybreaks}}%
        The first two inner products in~\eqref{eq:ApB} can be simplified as  
        {\begin{allowdisplaybreaks}
        \begin{align*}
            A_k & =  2\gamma^{-2}\inner{ z^{k} - z^{k+1} }{\bar{z}^{k+1}-z^{k+1}}   - 2 \gamma^{-1} \inner{F\p{\bar{z}^{k}} }{\bar{z}^{k+1}-z^{k+1}} \\
            &= \gamma^{-2}\norm{z^{k} - z^{k+1}}^2 + \gamma^{-2}\norm{ \bar{z}^{k+1}-z^{k+1} }^2 - \gamma^{-2}\norm{z^{k} - \bar{z}^{k+1} }^2  - 2 \gamma^{-1} \inner{F\p{\bar{z}^{k}} }{\bar{z}^{k+1}-z^{k+1}},
        \end{align*}
        \end{allowdisplaybreaks}}%
        where the identity \(2\inner{x}{y} = \norm{x}^2 + \norm{y}^2 - \norm{x-y}^2\) for each \(x,y\in \calH\) is used in the second equality, while the remaining four terms in~\eqref{eq:ApB} can be simplified as 
        {\begin{allowdisplaybreaks}
        \begin{align*}
            B_k &  =  
            \begin{aligned}[t]
                & 2\gamma^{-2}\norm{ z^{k+1} - z^{k+2} }^{2}  - 2\gamma^{-1}\inner{ F\p{\bar{z}^{k+1}} }{z^{k+1}-z^{k+2}} \\
                & + 2\gamma^{-2}\inner{ z^{k+1} - \bar{z}^{k+1}}{z^{k+2}-\bar{z}^{k+1}} -  2\gamma^{-1} \inner{ F\p{z^{k+1}} }{z^{k+2}-\bar{z}^{k+1}} 
            \end{aligned}\\
            & = 
            \begin{aligned}[t]
                & 2\gamma^{-1}\inner{z^{k+1} - z^{k+2}}{F\p{z^{k+1}} - F\p{\bar{z}^{k+1}}} + \gamma^{-2}\norm{z^{k+2} - \bar{z}^{k+1}}^{2} + \gamma^{-2}\norm{z^{k+1} - z^{k+2}}^{2}  \\
                & + \gamma^{-2}\norm{z^{k+1} - \bar{z}^{k+1}}^2 - 2\gamma^{-1} \inner{ F\p{z^{k+1}} }{z^{k+1}-\bar{z}^{k+1}} 
            \end{aligned}\\
            &= \mathcal{V}_{k+1} + \gamma^{-2}\norm{z^{k+1} - \bar{z}^{k+1}}^2 - 2\gamma^{-1} \inner{ F\p{z^{k+1}} }{z^{k+1}-\bar{z}^{k+1}},
        \end{align*}
        \end{allowdisplaybreaks}}%
        where the identity \(2\inner{x}{y} = \norm{x}^2 + \norm{y}^2 - \norm{x-y}^2\) for each \(x,y\in \calH\) is used in the second equality, and the explicit expression for \(\V_{k+1} =\mathcal{V}(z^{k+1}, \bar z^{k+1}, z^{k+2})\) is used in the third equality. Therefore, the inequality \(A_k + B_k \leq 0\) in~\eqref{eq:ApB} can be rearranged as  
        \begin{equation}\label{eq:V+:ub}
            \mathcal{V}_{k+1}  
            \leq{}
            \begin{aligned}[t]
                &\gamma^{-2}\norm{z^{k} - \bar{z}^{k+1}}^2  - \gamma^{-2}\norm{z^{k} - z^{k+1}}^2 - 2\gamma^{-2}\norm{ \bar{z}^{k+1}-z^{k+1} }^2 \\
                & + 2\gamma^{-1} \smash{\underbracket{\inner{F\p{\bar{z}^{k}} - F\p{z^{k+1}} }{\bar z^{k+1}-{z}^{k+1}}%
                }_{=C_k}}.
            \end{aligned}
        \end{equation}
        We can upper bound the term \(C_k\) as
            \begin{align}
            C_k
            ={}&
            \inner{F\p{{z}^{k}} - F\p{z^{k+1}} }{z^{k+1} - z^k} 
            + 
            \inner{F\p{\bar{z}^{k}} - F\p{z^{k}} }{z^{k+1} - z^k} 
            \nonumber \\
            &{}
            + \inner{F\p{\bar{z}^{k}} - F\p{z^{k+1}} }{\bar{z}^{k+1} + z^k - 2z^{k+1}}
            \nonumber
            \\
            \leq{}&
            \inner{F\p{\bar{z}^{k}} - F\p{z^{k}} }{z^{k+1} - z^k}
            \nonumber \\
            &{}
            + \tfrac{\gamma}{2}\norm{F\p{\bar{z}^{k}} - F\p{z^{k+1}}}^{2} + \tfrac1{2\gamma}\norm{\bar{z}^{k+1}-2z^{k+1} + z^{k}}^{2}
            \nonumber 
            \\
            \leq{}&
            \inner{F\p{\bar{z}^{k}} - F\p{z^{k}} }{z^{k+1} - z^k}
            {}
            + \tfrac{\gamma L_F^2}{2}\norm{{\bar{z}^{k}} - {z^{k+1}}}^{2} 
            \nonumber \\
            &{}
            + \tfrac1{2\gamma}
            (%
            2\norm{z^{k} - z^{k+1}}^2 + 2\norm{ \bar{z}^{k+1}-z^{k+1} }^2
            -\norm{z^{k} - \bar{z}^{k+1}}^2 
            )
            \label{eq:innerprod:vk+:ub}
        \end{align}
        where monotonicity of \(F\) and Young's inequality is used in the first inequality, and Lipschitz continuity of \(F\) along with the identity \(\norm{x+y}^{2} + \norm{x-y}^{2} = 2\norm{x}^{2} + 2\norm{y}^{2}\) for each \(x,y\in\calH\) is used in the last inequality. 

        Combining~\eqref{eq:V+:ub} and~\eqref{eq:innerprod:vk+:ub}, and using \(\mathcal{V}_{k}= 2\gamma^{-1}\inner{z^{k} - z^{k+1}}{F\p{z^{k}} - F\p{\bar{z}^{k}}} + \gamma^{-2}\norm{z^{k+1} - \bar{z}^{k}}^{2}  + \gamma^{-2}\norm{z^{k} - z^{k+1}}^{2}\) gives 
        \begin{align*}
            \V_{k+1} - \V_k 
            \leq{}&
            2\gamma^{-1}\inner{z^{k+1} - z^k}{F\p{\bar{z}^{k}} - F\p{z^{k}} }
            {}
            + { L_F^2}\norm{{\bar{z}^{k}} - {z^{k+1}}}^{2} 
            \\
            &{}
            +\gamma^{-2}\norm{z^{k} - z^{k+1}}^2 
            - \V_k
            \\
            = {}&
            - (1- \gamma^2 L_F^2)\gamma^{-2}\norm{{\bar{z}^{k}} - {z^{k+1}}}^{2},
        \end{align*}
        as claimed. 
    \end{proof}
\end{theorem}

Next, we present \Cref{cor:EG:new_Lyapunov_w.o._sol:g_zero}, which follows immediately from \Cref{thm:EG:new_Lyapunov_w.o._sol} by letting \(g=0\). \new{Observe that \Cref{cor:EG:new_Lyapunov_w.o._sol:g_zero} recovers known results, e.g., see~\cite[Lemma 3.2 and Theorem 3.3]{gorbunov2022extragradientmethodo1/k},~\cite[Theorem 1]{cai2022tightlastiterateconvergenceextragradient}, and~\cite[Remark 2.1]{bot2024extragradientmethod}.}

\begin{corollary}\label{cor:EG:new_Lyapunov_w.o._sol:g_zero}
    Suppose that \Cref{ass:F} holds and the sequence \(\p{\p{z^{k},\bar{z}^{k}}}_{k\in\naturals}\) is generated by 
    \begin{subequations}
            \begin{align*}
                \bar{z}^{k} &=z^{k} - \gamma F\p{z^{k}}, \\
                z^{k+1}     &=z^{k} - \gamma F\p{\bar{z}^{k}}
            \end{align*}
    \end{subequations}
    for each \(k\in\naturals\), with initial point \(z^0\in \calH\) and step-size parameter \(\gamma \in \reals_{++}\). Then, 
    \begin{align}\label{eq:EG:new_Lyapunov_w.o._sol:g_zero}
        \norm{F\p{z^{k+1}}}^{2} \leq \norm{F\p{z^{k}}}^{2} - \p{1-\gamma^{2}L_{F}^{2}} \norm{F\p{z^{k}}-F\p{\bar{z}^{k}}}^{2}
    \end{align}
    for each \(k\in \naturals\). \new{Moreover, if \(\gamma\in\p{0,1/L_{F}}\) and \(\zer\p{F} \neq \emptyset\), then
    \begin{align}\label{eq:EG:g_zero:little-o}
        \norm{F\p{z^{k}}}^{2} \in o\p{1/k} \text{ as }  k \to \infty,
    \end{align}
    and for any \(k\in\naturals\) and \(z^{\star} \in \zer\p{F}\) it holds that
    \begin{align}\label{eq:EG:g_zero:last_iterate_rate}
        \norm{F\p{z^{k}}}^{2} \leq \frac{\norm{z^{0}-z^{\star}}^{2} }{\gamma^{2}\p{1-\gamma^{2} L_{F}^{2}}\p{k+1}}.
    \end{align}}
    
    \begin{proof}
        Letting \(g=0\) in \Cref{thm:EG:new_Lyapunov_w.o._sol} gives~\eqref{eq:EG:new_Lyapunov_w.o._sol:g_zero}. Using \(g=0\),~\eqref{eq:EG:Lyapunov_with_solution} in \Cref{prop:EG:primal_descent} gives
        \begin{equation}\label{eq:EG:Lyapunov_with_solution:g_zero_and_relaxed}
            \norm{z^{i+1}-z^{\star}}^{2} \leq \norm{z^{i}-z^{\star}}^{2} 
            - \gamma^{2}\p{1-\gamma^{2} L_{F}^{2}}\norm{F\p{z^{i}}}^{2}.        
        \end{equation}
        for each \(i\in\naturals\). Inductively summing~\eqref{eq:EG:Lyapunov_with_solution:g_zero_and_relaxed} from \(i=0\) to \(i=k\), rearranging, and dividing by \(\gamma^{2}\p{1-\gamma^{2} L_{F}^{2}}\) gives that
        {\begin{allowdisplaybreaks}
        \begin{align*}
            \sum_{i=0}^{k} \norm{F\p{z^{i}}}^{2} 
            &\leq \frac{\sum_{i=0}^{k}\p{\norm{z^{i}-z^{\star}}^{2} - \norm{z^{i+1}-z^{\star}}^{2}}}{\gamma^{2}\p{1-\gamma^{2} L_{F}^{2}}} \\
            &\leq \frac{\norm{z^{0}-z^{\star}}^{2} }{\gamma^{2}\p{1-\gamma^{2} L_{F}^{2}}} 
        \end{align*}
        \end{allowdisplaybreaks}}%
        for each \(k\in\naturals\). Now \eqref{eq:EG:g_zero:little-o} and \eqref{eq:EG:g_zero:last_iterate_rate} follow from the monotonicity of \(\p{\norm{F\p{z^{i}}}}_{i\in\naturals}\), i.e., \eqref{eq:EG:new_Lyapunov_w.o._sol:g_zero}. 
    \end{proof}
\end{corollary}

The following result shows that \(\V_k\), when scaled by a nonnegative constant, equals the residual of a Fej\'er-type inequality. As a direct consequence, this gives a \(o(1/k)\) last-iterate convergence result in terms of \(\V_k\) as presented in \Cref{cor:lastiter:T}. 

\begin{theorem}\label{thm:EG:primal_descent_for_V}
    Suppose that \Cref{ass:main} holds, the sequence \(\p{\p{z^{k},\bar{z}^{k}}}_{k\in\naturals}\) is generated by~\eqref{alg:EG} with initial point \(z^0\in \calH\) and step-size parameter \(\gamma\in(0,1/L_{F}]\), and the sequence \(\p{\mathcal{V}_{k}}_{k\in\naturals}\) is given by \(\V_k =\mathcal{V}(z^k, \bar z^k, z^{k+1})\) for each \(k\in \naturals\) and the Lyapunov function \(\V\) defined in~\eqref{eq:V}. Then, for any \(k\in\naturals\) and \(z^{\star}\in\zer\p{F + \partial g}\) it holds that
    \begin{equation}\label{eq:EG:primal_descent_for_V}
        \norm{z^{k+1}-z^{\star}}^{2} \leq \norm{z^{k}-z^{\star}}^{2} 
        - \alpha\p{\gamma, L_F} \V_k,
    \end{equation}
    where
    \begin{align}\label{eq:EG:primal_descent_for_V:alpha}
        \alpha\p{\gamma, L_F} = \tfrac{\gamma^{2}}{2}\p{\sqrt{5-4\gamma^2 L_{F}^2 } - 1} \geq 0.
    \end{align}
    \begin{proof}
         Note that~\eqref{eq:subgrad:char} and \(-F\p{z^{\star}} \in \partial g\p{z^{\star}}\) can equivalently be characterized by
        \begin{subequations}
            \begin{align}\label{eq:primal_descent_for_V:subdiff:1}
                0&\leq g\p{z} -  g\p{\bar{z}^{k}} - \inner{ \gamma^{-1} \p{ z^{k} - \bar{z}^{k} } - F\p{z^{k}}}{z - \bar{z}^{k} }, \\\label{eq:primal_descent_for_V:subdiff:2}
                0&\leq  g(z) - g\p{z^{k+1}} - \inner{ \gamma^{-1} \p{ z^{k} - z^{k+1} } - F\p{\bar{z}^{k}}}{z - z^{k+1} },\\\label{eq:primal_descent_for_V:subdiff:3}
                0&\leq g\p{z} - g(z^{\star}) - \inner{ - F\p{z^{\star}}}{z - z^{\star} }
            \end{align}
        \end{subequations}
        for each \(z\in\calH\). Picking \(z = z^{k+1}\) in~\eqref{eq:primal_descent_for_V:subdiff:1}, \(z = z^{\star}\) in~\eqref{eq:primal_descent_for_V:subdiff:2}, \(z = \bar{z}^{k}\) in~\eqref{eq:primal_descent_for_V:subdiff:3}, summing the resulting inequalities, and multiplying by \(2\gamma\) gives  
        {\begin{allowdisplaybreaks}
        \begin{align}
            0 &\leq 
            \begin{aligned}[t] 
                & - 2\inner{ z^{k} - \bar{z}^{k}}{z^{k+1} - \bar{z}^{k} } + 2\gamma\inner{  F\p{z^{k}}}{z^{k+1} - \bar{z}^{k} } \\
                & - 2\inner{ z^{k} - z^{k+1}}{z^{\star} - z^{k+1} }  +  2\gamma\inner{ F\p{\bar{z}^{k}}}{z^{\star} - z^{k+1} } \\ 
                & - 2\gamma\inner{F\p{\bar{z}^{k}} - F\p{z^{\star}}}{\bar{z}^{k} - z^{\star} } + 2\gamma\inner{F\p{\bar{z}^{k}}}{\bar{z}^{k} - z^{\star} } 
            \end{aligned}\notag\\
            &\leq 
            \begin{aligned}[t] 
                & \norm{z^{k} -z^{k+1}}^{2} - \norm{z^{k} - \bar{z}^{k}}^{2} - \norm{z^{k+1} - \bar{z}^{k}}^{2} + 2\gamma\inner{  F\p{z^{k}}}{z^{k+1} - \bar{z}^{k} } \\
                & + \norm{z^{k} - z^{\star}}^{2} - \norm{z^{k} - z^{k+1}}^{2} - \norm{z^{\star} - z^{k+1}}^{2}  +  2\gamma\inner{ F\p{\bar{z}^{k}}}{z^{\star} - z^{k+1} } \\ 
                &+ 2\gamma\inner{F\p{\bar{z}^{k}}}{\bar{z}^{k} - z^{\star} } 
            \end{aligned}\notag\\
            &= 
            \begin{aligned}[t] 
                & \norm{z^{k} - z^{\star}}^{2} - \norm{z^{\star} - z^{k+1}}^{2} - \norm{z^{k} - \bar{z}^{k}}^{2} - \norm{z^{k+1} - \bar{z}^{k}}^{2} \\ 
                &- 2\gamma\inner{  F\p{z^{k}} - F\p{\bar{z}^{k}}}{ \bar{z}^{k} - z^{k+1} },
            \end{aligned} \label{eq:V:primal_ineq:sub_ineq_1}
        \end{align}
        \end{allowdisplaybreaks}}%
        where the identity \(-2\inner{x}{y} = \norm{x-y}^2 - \norm{x}^2 - \norm{y}^2\) for each \(x,y\in \calH\) and monotonicity of \(F\) is used in the second inequality. Picking \(z = z^{k+1}\) in~\eqref{eq:primal_descent_for_V:subdiff:1}, \(z = \bar{z}^{k}\) in~\eqref{eq:primal_descent_for_V:subdiff:2},  and summing the resulting inequalities gives
        {\begin{allowdisplaybreaks}
        \begin{align}
            0 &\leq 
            \begin{aligned}[t]
                & g\p{z^{k+1}} - g\p{\bar{z}^{k}} - \inner{ \gamma^{-1} \p{ z^{k} - \bar{z}^{k} } - F\p{z^{k}}}{z^{k+1} - \bar{z}^{k} } \\
                & + g\p{\bar{z}^{k}} - g\p{z^{k+1}} - \inner{\gamma^{-1} \p{ z^{k} - z^{k+1} } - F\p{\bar{z}^{k}}}{\bar{z}^{k}-z^{k+1}}
            \end{aligned} \notag \\
            &=  - \inner{ F\p{z^{k}} - F\p{\bar{z}^{k}}  }{ \bar{z}^{k} - z^{k+1} }  -  \gamma^{-1}\norm{ \bar{z}^{k} - z^{k+1}  }^{2}. \label{eq:V:primal_ineq:sub_ineq_2}
        \end{align}
        \end{allowdisplaybreaks}}%
        For notational simplicity, we let \(\alpha = \alpha\p{\gamma, L_{F}}\) for \(\alpha\p{\gamma, L_{F}}\) as in~\eqref{eq:EG:primal_descent_for_V:alpha}, where simple algebra shows that \(\alpha \geq 0\) if and only if \(\gamma L_{F} \leq 1\). Multiplying~\eqref{eq:V:primal_ineq:sub_ineq_2} with \(2 \alpha\gamma^{-1}\), and adding the result to~\eqref{eq:V:primal_ineq:sub_ineq_1} gives 
        {\begin{allowdisplaybreaks}
        \begin{align*}
            0 &\leq 
            \begin{aligned}[t] 
                &\norm{z^{k} - z^{\star} }^{2} - \norm{z^{k+1} - z^{\star}}^{2} - \norm{z^{k} - \bar{z}^{k}}^{2} - (1+2\alpha\gamma^{-2}) \norm{z^{k+1} - \bar{z}^{k}}^{2} \\
                &  - 2\gamma(1+\alpha\gamma^{-2}) \inner{  F\p{z^{k}} - F\p{\bar{z}^{k}}}{\bar{z}^{k} - z^{k+1} } 
                \end{aligned} \\
            &= \norm{z^{k} - z^{\star}}^{2} - \norm{z^{k+1} - z^{\star}}^{2}  -  \alpha\V_k + A_k,
        \end{align*}
        \end{allowdisplaybreaks}}%
        where
        \begin{align}
            A_k 
            ={}&
             - \norm{z^{k} - \bar{z}^{k}}^{2}  - \p{1+ 2\alpha\gamma^{-2}} \norm{z^{k+1} - \bar{z}^{k}}^{2} - 2\gamma\p{1+ \alpha\gamma^{-2}} \inner{  F\p{z^{k}} - F\p{\bar{z}^{k}}}{\bar{z}^{k} - z^{k+1} } 
             \nonumber\\  
            &{} + \alpha\V_k 
            \nonumber\\
            ={}& - \norm{z^{k} - \bar{z}^{k}}^{2}  - \p{1+ \alpha\gamma^{-2}} \norm{z^{k+1} - \bar{z}^{k}}^{2} 
                +
                 \alpha\gamma^{-2}\norm{z^{k} - z^{k+1}}^{2} 
            \nonumber\\  
            &{}
            +
            2\gamma
            \smash{\underbracket{\inner{%
            \alpha\gamma^{-2}\p{z^{k} - z^{k+1}} 
            - 
            (1+\alpha\gamma^{-2})\p{\bar z^k - z^{k+1}}
            }{F\p{z^{k}} - F\p{\bar{z}^{k}}}}_{= B_k}}, 
            \label{eq:Ak:thm:tight}
        \end{align}
        where we substituted \(\V_k\).
        
        To complete the proof, it is enough to show that \(A_k \leq 0\). The last inner product in~\eqref{eq:Ak:thm:tight} can be upper bounded using Young's inequality as
        \begin{align} \notag
            B_k \leq {}&
            \tfrac{\gamma}{2}
            \norm{F\p{z^{k}} - F\p{\bar{z}^{k}}}^2 
            + \tfrac{1}{2\gamma}
            \norm{\alpha\gamma^{-2}\p{z^{k} - z^{k+1}} - (1+\alpha\gamma^{-2})\p{\bar z^k - z^{k+1}}}^2 \\\notag
            \leq{}&
            \tfrac{\gamma L_F^2}{2}
            \norm{z^{k} - \bar{z}^{k}}^2 \\ \label{eq:Bk:thm:tight}
            &{} + 
            \tfrac{1}{2\gamma} 
            \p{            
            \p{1+\alpha\gamma^{-2}}\norm{\bar{z}^k - z^{k+1}}^2 
            + \alpha\gamma^{-2}\p{1+\alpha\gamma^{-2}} \norm{z^{k} - \bar{z}^k}^2
            - \alpha\gamma^{-2}\norm{z^{k} - z^{k+1}}^2  
            },
        \end{align}
        where Lipschitz continuity of \(F\) and the identity 
        \(\norm{\beta x - \p{1+\beta}y}^2 = \p{1+\beta}\norm{y}^2 + \beta(1+\beta)\norm{x-y}^2 - \beta\norm{x}^2\)
        for each \(x,y\in\calH\) and each \(\beta\in\reals\) {\cite[Corollary 2.15]{bauschke2017convexanalysismonotone}} are used in the second inequality. Substituting~\eqref{eq:Bk:thm:tight} in~\eqref{eq:Ak:thm:tight} gives 
        \begin{equation}
            A_k \leq{} - \p{1- \gamma^2 L_F^2 - \alpha\gamma^{-2}\p{1+\alpha\gamma^{-2}}} \norm{z^{k} - \bar{z}^{k}}^{2}  = 0,
       \end{equation}
       where the last equality follows from simple algebra after substituting \(\alpha = \alpha\p{\gamma, L_F}\) as in~\eqref{eq:EG:primal_descent_for_V:alpha}. 
    \end{proof}
\end{theorem}

\begin{corollary}\label{cor:lastiter:T}
    Suppose that \Cref{ass:main} and \(\zer\p{F + \partial g}\neq \emptyset\) hold, and the sequence \(\p{\p{z^{k},\bar{z}^{k}}}_{k\in\naturals}\) is generated by~\eqref{alg:EG} with initial point \(z^0\in \calH\) and step-size parameter \(\gamma\in(0,1/L_{F})\). \new{Then,
    \begin{align}\label{eq:Vk:last_iterate:little-o}
        \V_{k} \in o\p{1/k} \text{ as } k \to \infty,
    \end{align}
    and for any \(k\in\naturals\) and \(z^{\star}\in\zer\p{F + \partial g}\) it holds that  
    \begin{align}\label{eq:Vk:last_iterate}
        \V_{k} \leq \frac{\norm{z^{0} - z^{\star} }^{2} }{ \alpha\p{ \gamma, L_{F}} \p{k+1} },
    \end{align}
    where \(\alpha\p{ \gamma, L_{F}}>0\) is defined in~\eqref{eq:EG:primal_descent_for_V:alpha} and \(\V_k =\mathcal{V}(z^k, \bar z^k, z^{k+1})\) for the Lyapunov function \(\V\) defined in~\eqref{eq:V}.}
    \begin{proof}
        \new{The last-iterate convergence results in \eqref{eq:Vk:last_iterate:little-o} and \eqref{eq:Vk:last_iterate} follows by inductively summing~\eqref{eq:EG:primal_descent_for_V}, rearranging, dividing by \(\alpha\p{ \gamma, L_{F}}\), and using monotonicity of \(\p{\V_k}_{k\in\naturals}\) as shown in~\eqref{eq:EG:new_Lyapunov_w.o._sol}.} 
    \end{proof}
\end{corollary}

\new{
\begin{remark}\label{rmk:cyclically_monotone}
    A close review of the proofs of \Cref{thm:EG:new_Lyapunov_w.o._sol,thm:EG:primal_descent_for_V} (and therefore also \Cref{cor:lastiter:T}) shows that the arguments remain valid when \(\partial g\) in \eqref{eq:the_inclusion_problem} is replaced with a maximally monotone and 3-cyclically monotone operator \(T:\calH\to\calH\) (see \cite[Definition 22.13]{bauschke2017convexanalysismonotone}) and the proximal operators \(\prox_{\gamma g}\) in \eqref{alg:EG} with the resolvent \(\p{\Id + \gamma T}^{-1}\); the only change is to use cyclic-monotonicity inequalities in place of the subgradient inequalities. We nonetheless choose to present the slightly more restrictive subdifferential‐based formulation, as it avoids the added abstraction of cyclic monotonicity and is likely more immediately accessible to a broader audience.
\end{remark}
}

\section{Algorithms for monotone inclusions}\label{sec:algs}
In light of the descent property established in \Cref{thm:EG:new_Lyapunov_w.o._sol}, we propose line-search extensions of the extragradient method that combine the nominal steps of~\eqref{alg:EG} with user-specified directions. This section focuses on identifying appropriate conditions to guarantee global convergence, with detailed proofs deferred to \Cref{sec:global}. We deliberately leave the choice of directions open at this stage and postpone the superlinear convergence analysis to \Cref{sec:superlinear}. This abstraction offers flexibility in choosing methods---such as (inexact) (quasi-)Newton approaches, Anderson acceleration, or other suitable algorithms---for computing the directions.

In the first subsection, we consider the classical extragradient setting where \(g = 0\) and introduce a line search based on \(\|F(z^k)\|^2\) and its descent inequality in~\eqref{eq:EG:new_Lyapunov_w.o._sol:g_zero}. We then extend our approach in \Cref{sec:ProxFLEX} to the more general setting of~\eqref{eq:the_inclusion_problem}. Separating the analysis in this way reflects the stronger convergence results available when \(g = 0\), as well as the fact that, in this case, the line search is more computationally efficient.

\subsection{Fast line-search extragradient}\label{sec:FLEX}
In this subsection, we focus on the case \(g = 0\) in~\eqref{eq:the_inclusion_problem}, where the Lyapunov inequality~\eqref{eq:EG:new_Lyapunov_w.o._sol} simplifies to~\eqref{eq:EG:new_Lyapunov_w.o._sol:g_zero}. The first algorithm introduced here is~\refFLEX, which can be viewed as a hybrid scheme in the same spirit as \cite[Algorithm 5.16]{izmailov2014newtontypemethods}. At each iteration, it computes a suitable direction \(d^k\) (see \Cref{sec:superlinear}) and performs the updates \(z^{k+1} = z^k + d^k\) whenever the contraction condition in \Cref{State:FLEX:contraction} holds. Otherwise, it conducts a line search based on the descent inequality~\eqref{eq:EG:new_Lyapunov_w.o._sol:g_zero}, serving as a \emph{performance safeguard}.

Before we present the convergence results for~\FLEX{}, we offer some observations on the line-search procedure.

\begin{remark}
    \label{rem:FLEX:LS} 
   The line-search interpolation strategy in~\FLEX{} is designed to ensure global convergence while infusing local update directions in the algorithm. It differs from standard line-search procedures in some respects. 
    \begin{enumerate}
        \item\label{rem:FLEX:LS:bad_case}%
        After a finite number of backtracks, the method defaults to \(\tau_k = 0\), at which point~\eqref{eq:FLEX:descent} is satisfied due to~\eqref{eq:EG:new_Lyapunov_w.o._sol:g_zero} in \Cref{cor:EG:new_Lyapunov_w.o._sol:g_zero}. 
        Taking the nominal step after a finite number of trials is not just a practical consideration but is also theoretically grounded. Without additional assumptions, it is possible that \(\|F(z^k) - F(\bar{z}^k)\| = 0\) even when no solution has been found, and~\eqref{eq:FLEX:descent} is not satisfied by any \(\tau_k > 0\) for some ill-chosen user-specified direction \(d^k\). Therefore, additional assumptions are required for such edge cases if an infinite backtracking strategy with known finite termination is to be employed. This is further explored in \Cref{sec:IFLEX}.
        
        \item 
        Enforcing a descent inequality as in~\eqref{eq:FLEX:descent} of \Cref{State:FLEX:LS} can be viewed as a performance safeguarding. 
        As it is shown in \Cref{thm:FLEX:dk_summable} below, the convergence of~\FLEX{} can be guaranteed provided that \(\seq{d^k} \in \ell^{1}(\naturals;\,\calH)\). Therefore, the descent inequality within the line-search procedure ensures that the directions contribute effectively to the convergence, preventing arbitrarily poor performance.  
    \end{enumerate}  
\end{remark}
The next theorem establishes global convergence of~\FLEX{} under the assumption that the directions are summable, a setting that will be revisited in \Cref{sec:superlinear}, see also \Cref{thm:suplin:summable}.
Alternatively, if \(F\) is uniformly monotone, the summability assumption is dropped, as shown in \Cref{thm:FLEX:F_unifrom}. Moreover, when \(F\) is \(\mu_F\)-strongly monotone, as in \Cref{thm:FLEX:F_strongly}, a linear convergence rate is achieved.

\begin{algorithm}[t]
    \caption{\FLEX~(Fast Line-search EXtragradient)}
    \label{alg:FLEX}
    \begin{algorithmic}[1]
        \INITIALIZE  \(z^{0}\in\calH\), \(\gamma\in\p{0,{1}/{L_{F}}}\), \(\p{\rho,\sigma,\beta} \in \p{0,1}^{3}\), \(M\in\mathbb{N}_{0}\)
        \For{\(k = 0, 1, 2, \ldots\)}
            \State \label{State:FLEX:dk} Compute a direction \(d^{k}\in\calH\) at \(z^{k}\)
            \If{\(\norm{F\p{z^{k} + d^{k} }} \leq \rho\norm{F\p{z^{k}}}\)} \label{State:FLEX:contraction}
                \State \label{State:FLEX:full} \(z^{k+1} = z^{k} + d^{k}\)
            \Else
                \State \label{State:FLEX:EG} \(\bar{z}^{k} =  z^{k} - \gamma F\p{z^{k}}\); \(w^{k} = z^{k} - \gamma F\p{\bar{z}^{k}}\)
                \State \label{State:FLEX:LS} 
                \begin{tabular}[t]{@{}r@{}l@{}}
                     Set \(z^{k+1} = \p{1-\tau_{k}} w^{k} + \tau_{k} (z^{k} + d^{k})\) where \(\tau_{k}\) is the largest number in
                    \\
                     \(\set{ \beta^{i} \mid i \in \llbracket 1 , M \rrbracket}\cup\set{0}\) such that \hfill\hfill
                    \end{tabular}
                \begin{align}\label{eq:FLEX:descent}
                    \norm{F\p{z^{k+1}}}^{2} \leq \norm{F\p{z^{k}}}^{2} - \sigma\p{1-\gamma^{2}L_{F}^{2}} \norm{F\p{z^{k}}-F\p{\bar{z}^{k}}}^{2}
                \end{align}    
            \EndIf
        \EndFor
    \end{algorithmic}
\end{algorithm}

\begin{theorem}\label{thm:FLEX}
    Suppose that \Cref{ass:F} holds,  \(\zer \p{F}\neq \emptyset\), and the sequence \(\p{z^{k}}_{k\in\naturals}\) is generated by~\refFLEX. Then, the following hold.
    \begin{thmenumerate}

        \item \label{thm:FLEX:dk_summable} If \(\seq{d^k}\in\ell^{1}(\naturals;\,\calH)\), then \(\seq{z^k}\) converges weakly to some point in \(\zer \p{F}\). 

        \item \label{thm:FLEX:F_unifrom} If \(F\) is uniformly monotone, then \(\seq{z^k}\) converges weakly to some point in \(\zer \p{F}\).
    
        \item \label{thm:FLEX:F_strongly} If there exists \(0<\mu_F\leq L_F\) such that \(\mu_F \norm{x - y}\leq\norm{F\p{x}- F\p{y}}\) for each \(x,y\in \calH\), then \(\seq{z^k}\) converges strongly to some point in \(\zer \p{F}\) and
        \begin{align}\label{eq:FLEX:F_strongly:linear}
                \norm{F(z^{k+1})}^2 \leq 
                \smash{\underbracket{
                \max\p{ \rho^2, 1- \sigma\gamma^{2}\mu_{F}^{2}\p{1-\gamma^{2}L_{F}^{2}} }
                }_{\in(0,1)}}
                \norm{F(z^{k})}^2
            \end{align} 
        for each \(k\in\naturals\).
    \end{thmenumerate}
\end{theorem}

\subsubsection{Variant of \texttt{FLEX} under injectivity}\label{sec:IFLEX}
As highlighted in \Cref{rem:FLEX:LS:bad_case}, since \(\|F(z^k) - F(\bar{z}^k)\| = 0\) can occur without reaching a solution, special considerations are necessary. In~\FLEX, this is addressed by employing an explicit finite termination in the line-search procedure and assuming that the directions \(d^{k}\) are summable. However, when the operator \(F\) is injective, it is possible to exploit the Lyapunov inequality in~\eqref{eq:EG:new_Lyapunov_w.o._sol:g_zero} directly to establish convergence results without additional assumptions. To this end, we introduce~\IFLEX{}, which incorporates a more traditional line-search procedure similar to that used in~\cite[Algorithm \texttt{PANOC}]{stella2017simpleefficientalgorithm}. However, the \texttt{PANOC} algorithm is developed for minimization problems and utilizes a fundamentally different Lyapunov function. Importantly,~\IFLEX{} uses an infinite backtracking strategy with guaranteed finite termination since injectivity ensures that \(\|F(z^k) - F(\bar{z}^k)\| = 0\) only when a solution has been found. Moreover,~\IFLEX{} has two fewer parameters than~\FLEX, simplifying its implementation. \new{Note that the computation of \(w^{k}\) in \Cref{State:I-FLEX:EG} in~\IFLEX{} can be deferred to the case in which \(\tau_{k}=1\) in \eqref{eq:I-FLEX:descent} fails, saving some computations.}

\begin{algorithm}[t]
    \caption{\IFLEX{} (Injective-\texttt{FLEX})}
    \label{alg:I-FLEX}
    \begin{algorithmic}[1]
        \INITIALIZE  \(z^{0}\in\calH\), \(\gamma\in\p{0,{1}/{L_{F}}}\), \(\p{\sigma,\beta} \in \p{0,1}^{2}\)
        \For{\(k = 0, 1, 2, \ldots\)}
        \State \label{State:I-FLEX:dk} Compute a direction \(d^{k}\in\calH\) at \(z^{k}\)
        \State \label{State:I-FLEX:EG} \(\bar{z}^{k} =  z^{k} - \gamma F\p{z^{k}}\); \(w^{k} = z^{k} - \gamma F\p{\bar{z}^{k}}\)
        \State \label{State:I-FLEX:LS} 
            \begin{tabular}[t]{@{}r@{}l@{}}
                     Set \(z^{k+1} = \p{1-\tau_{k}} w^{k} + \tau_{k} (z^{k} + d^{k})\) where \(\tau_{k}\) is the largest number in 
                    \\
                     \(\set{ \beta^{i} \mid i \in \naturals}\) such that \hfill\hfill
            \end{tabular}
            \begin{align}\label{eq:I-FLEX:descent}
                \norm{F\p{z^{k+1}}}^{2} \leq \norm{F\p{z^{k}}}^{2} - \sigma\p{1-\gamma^{2}L_{F}^{2}} \norm{F\p{z^{k}}-F\p{\bar{z}^{k}}}^{2}
            \end{align}    
        \EndFor
    \end{algorithmic}
\end{algorithm}

\begin{proposition}\label{prop:IFLEX:well_defined}
    Suppose that \Cref{ass:F} holds and \(F\) is injective. Then, independent of the choice of the direction \(d^{k}\) in \Cref{State:I-FLEX:dk} of~\refIFLEX{}, either there exists an iteration \(k\in\naturals\) such that \(z^{k}\in\zer\p{F}\) or the line search in \Cref{State:I-FLEX:LS} is well-defined for each iteration \(k\in\naturals\).
    \begin{proof}
        Follows from \(\sigma \in (0,1)\),~\eqref{eq:EG:new_Lyapunov_w.o._sol:g_zero}, continuity of \(F\), and that \(\norm{F\p{z^{k}}-F\p{\bar{z}^{k}}}\neq 0\) if and only if \(z^{k} \notin \zer\p{F}\).  
    \end{proof}
\end{proposition}

\begin{theorem}\label{thm:IFLEX}
    Suppose that \Cref{ass:F} holds,  \(\zer \p{F}\neq \emptyset\), and the sequence \(\p{z^{k}}_{k\in\naturals}\) is generated by~\refIFLEX. 
    \begin{thmenumerate}

        \item \label{thm:I-FLEX:injective} If \(F\) is injective and weakly continuous, then each weak sequential cluster point of  \(\seq{z^k}\) is in \(\zer \p{F}\). 

        \item \label{thm:I-FLEX:dk_bounded} If \(F\) is injective and weakly continuous, and \(\seq{d^k}\in\ell^{1}(\naturals;\,\calH)\), then \(\seq{z^k}\) converges weakly to some point in \(\zer \p{F}\). 

        \item \label{thm:I-FLEX:F_unifrom} If \(F\) is uniformly monotone, then \(\seq{z^k}\) converges weakly to some point in \(\zer \p{F}\).
        
        \item \label{thm:I-FLEX:F_strongly} If there exists \(0<\mu_F\leq L_F\) such that \(\mu_F \norm{x - y}\leq\norm{F\p{x}- F\p{y}}\) for each \(x,y\in \calH\), then \(\seq{z^k}\) converges strongly to some point in \(\zer \p{F}\) and
        \begin{align}\label{eq:IFLEX:F_strongly:linear}
            \norm{F(z^{k+1})}^2 \leq 
            \smash{\underbracket{
            \p{1-\sigma\gamma^{2}\mu_{F}^{2}\p{1-\gamma^{2}L_{F}^{2}}} 
            }_{\in(0,1)}}
            \norm{F(z^{k})}^2
        \end{align}
        for each \(k\in\naturals\).
    \end{thmenumerate}
\end{theorem}


\subsection{Proximal fast line-search extragradient}
\label{sec:ProxFLEX}
A direct generalization of~\refFLEX{} in \Cref{sec:FLEX} is provided in~\refProxFLEX{} for the case when \( g \) in~\eqref{eq:the_inclusion_problem} is nonzero. Here, the Lyapunov inequality~\eqref{eq:EG:new_Lyapunov_w.o._sol} from \Cref{thm:EG:new_Lyapunov_w.o._sol} is used to modify the standard extragradient method in~\eqref{alg:EG}; otherwise, the underlying approach remains the same. However, there is one important difference between~\FLEX{} and~\ProxFLEX{} in terms of computations required per line-search trial. The condition~\eqref{eq:FLEX:descent} in~\FLEX{} requires only one additional \(F\) evaluation per trial while condition~\eqref{eq:Prox-FLEX:descent} in~\ProxFLEX{} requires two additional \(F\) evaluations and two additional \(\prox_{\gamma g}\) evaluations per trial. Next, we present a convergence result of~\ProxFLEX.

\begin{theorem}\label{thm:ProxFLEX}
    Suppose that \Cref{ass:main} holds, \(\zer \p{F + \partial g}\neq \emptyset\), the sequence \(\p{z^{k}}_{k\in\naturals}\) is generated by~\refProxFLEX, and \(\seq{d^k}\in\ell^{1}(\naturals;\,\calH)\). Then \(\seq{z^k}\) converges weakly to some point in \(\zer \p{F+ \partial g}\).
\end{theorem}

\begin{algorithm}[t]
    \caption{\ProxFLEX{} (Proximal-\texttt{FLEX})} 
    \label{alg:Prox-FLEX}
    \begin{algorithmic}[1]
        \INITIALIZE  \(z^{0}\in\calH\), \(\gamma\in\p{0,1/L_{F}}\), \(\p{\rho,\sigma,\beta} \in \p{0,1}^{3}\), \(M\in\mathbb{N}_{0}\)
        \REQUIRE Lyapunov function \(\V\) as in~\eqref{eq:V} and algorithmic operators \(T_1^{\gamma}, T_2^{\gamma}\) as in~\eqref{eq:T1T2}
        \For{\(k = 0, 1, 2, \ldots\)}
            \State\label{State:Prox-FLEX:dk}Compute a direction \(d^{k}\in\calH\) at \(z^{k}\)
            \State \(\bar{z}^{k} = T_1^\gamma(z^k)\) \label{State:Prox-FLEX:EG:1} \hfill {\footnotesize\gray  \( = \prox_{\gamma g} \big(z^{k} - \gamma F\p{z^{k}}\big) \)}
            \State \(w^{k} = T_2^\gamma(z^k)\)
            \hfill {\footnotesize\gray  \(= \prox_{\gamma g} \big(z^{k} - \gamma F\p{\bar{z}^{k}}\big)\)}
            \If{\(\mathcal{V}(z^{k} + d^{k}, T_1^\gamma(z^{k} + d^{k}), T_2^\gamma(z^{k} + d^{k})) \leq \rho^2 \mathcal{V}(z^k, \bar z^k, w^{k}) \)} \label{State:Prox-FLEX:contraction}
                \State \label{State:Prox-FLEX:full} \(z^{k+1} = z^{k} + d^{k}\)
            \Else
                \State \label{State:Prox-FLEX:LS} 
                \begin{tabular}[t]{@{}r@{}l@{}}
                Set \(z^{k+1} = \p{1-\tau_{k}} w^{k} + \tau_{k} (z^{k} + d^{k})\) where \(\tau_{k}\) is the largest number in 
                \\
                \(\set{ \beta^{i} \mid i \in \llbracket 1 , M \rrbracket}\cup\set{0}\) such that \hfill\hfill
                \end{tabular}
                \begin{align}\label{eq:Prox-FLEX:descent}
                     ~~~~~~~~~\mathcal{V}(z^{k+1}, T_1^\gamma(z^{k+1}), T_2^\gamma(z^{k+1})) \leq \mathcal{V}(z^k, \bar z^k, w^{k}) - \sigma\p{1-\gamma^{2}L_{F}^{2}}\gamma^{-2} \norm{w^{k}-\bar{z}^{k}}^{2}
                \end{align}
            \EndIf
        \EndFor
    \end{algorithmic}
\end{algorithm}

\begin{remark}
    \begin{enumerate}
        \item \label{rem:strength:proxflex} If \(F\) is \(\mu_{F}\)-strongly monotone, then the Lyapunov inequality~\eqref{eq:EG:new_Lyapunov_w.o._sol} in \Cref{thm:EG:new_Lyapunov_w.o._sol} can be strengthened to include the additional term \(-2\gamma^{-1}\mu_{F}\norm{z^{k+1} - z^{k} }^2\) in the right-hand side; this follows from using strong monotonicity of \(F\) instead of monotonicity of \(F\) in the first inequality in~\eqref{eq:innerprod:vk+:ub}. This observation suggests that the line-search condition~\eqref{eq:Prox-FLEX:descent} in~\ProxFLEX{} can be replaced by
        \begin{align}\notag
            \V(z^{k+1}, \bar{z}^{k+1}, w^{k+1}) 
            \leq{}& \V(z^k, \bar z^k, w^{k}) - \sigma\p{1-\gamma^{2}L_{F}^{2}}\gamma^{-2} \norm{w^{k}-\bar{z}^{k}}^{2} \\\label{eq:Prox-FLEX:descent:strongly}
            &{}-2\gamma^{-1}\sigma\mu_{F}\norm{w^{k} - z^{k} }^2.
        \end{align}
        Note that using Young's inequality, we get 
        \begin{align}\label{eq:V:upper_bound}
            \V(z^k, \bar z^k, w^{k}) \leq \p{2\gamma^{-1} + \gamma^{-1} L_{F} + \gamma^{-2}}\norm{w^{k} - z^{k}}^{2} + \p{\gamma^{-1} L_{F} + \gamma^{-2}}\norm{w^{k} - \bar{z}^{k}}^{2}.
        \end{align}
        Combining~\eqref{eq:Prox-FLEX:descent:strongly},~\eqref{eq:V:upper_bound} and \Cref{State:Prox-FLEX:contraction} in~\ProxFLEX{} gives 
        \begin{align*}
            \V(z^{k+1}, \bar{z}^{k+1}, w^{k+1}) \leq{}& 
            \underbracket{
             \max\left(\rho^2 ,1 - \frac{\sigma \min\p{(1-\gamma^{2}L_{F}^{2}) , 2\gamma\mu_{F}}}{2\gamma + \gamma L_{F} + 1} \right)
             }_{\in(0,1)}
             \V(z^k, \bar z^k, w^{k})
        \end{align*}
        for each \(k\in\naturals\), i.e. \(\seq{\V(z^k, \bar z^k, w^{k})}\) converges at least \(Q\)-linearly to zero. However, the resulting line-search condition is not always actionable since \(\mu_{F}\) may not be known in many practical problems. Therefore, we have chosen not to consider the strongly monotone case further. 
        
        \item Similar to~\IFLEX{},~\ProxFLEX{} can be modified to perform infinite backtracking on~\eqref{eq:Prox-FLEX:descent} with guaranteed finite termination, even without the strengthened line-search condition described above in \Cref{rem:strength:proxflex}. 
        This modification requires \(\norm{w^k-\bar{z}^k}\) to be an optimality measure, which holds when both \(F\) and \(\prox_{\gamma g}\) are injective. However, since \(\prox_{\gamma g}\) is rarely injective in practical applications, we omit this modification from our analysis.
    \end{enumerate}
\end{remark}


\section{Global convergence}\label{sec:global}
This section provides detailed proofs of the results presented in \Cref{sec:algs}. We start by providing two useful lemmas. The first lemma establishes that the iterates generated by~\FLEX,~\IFLEX, and~\ProxFLEX{} are quasi-Fej\'er monotone with respect to the solution set, which is an important tool in establishing global convergence. The second lemma contains some auxiliary results. 

\begin{lemma}\label{lem:zk_quasi_fejer}
    Suppose that \Cref{ass:main} holds, \(z^{\star}\in\zer \p{F + \partial g}\), \(\seq{d^k}\in\ell^{1}(\naturals;\,\calH)\), \(T_1^{\gamma}\) and \(T_2^{\gamma}\) are the algorithmic operators defined in~\eqref{eq:T1T2}, and \(\V\) is the Lyapunov function given in~\eqref{eq:V}. Let \(\p{z^{k}}_{k\in\naturals}\in\calH^{\naturals}\) such that \(z^{k+1} = (1-\tau_k)w^{k} + \tau_{k}(z^{k} + d^{k})\), where \(\tau_k \in [0,1]\) and \(w^{k} = T_2^\gamma(z^k)\) for each \(k\in\naturals\). Then there exists a sequence \(\seq{\varepsilon_k}\in\ell^{1}\p{\naturals;\,\reals_{+}}\) such that
    \begin{align}\label{eq:zk_quasi_fejer}
        \norm{z^{k+1}-z^\star}^2 \leq \norm{z^k-z^\star}^2 + \varepsilon_k - (1- \tau_k)\alpha\p{\gamma, L_F}\V_k,
    \end{align}
    where \(\alpha\p{\gamma, L_F}\) is defined in~\eqref{eq:EG:primal_descent_for_V:alpha}, \(\V_{k} = \V(z^k, \bar z^k, w^{k})\), and \(\bar z^k = T_1^{\gamma}(z^k)\) for each \(k\in\naturals\).
    \begin{proof}
        Note that \Cref{prop:V_performance_measure:lower_bound} gives that \(\V_{k}\geq 0\) for each \(k\in\naturals\). Using the identity 
        \begin{align*}
            \norm{\tau x+\p{1-\tau}y}^2 = \tau\norm{x}^2 + \p{1-\tau}\norm{y}^2 - \tau\p{1-\tau}\norm{x-y}^2
        \end{align*}
        for each \(x,y\in\calH\) and \(\tau\in\reals\) {\cite[Corollary 2.15]{bauschke2017convexanalysismonotone}}, and \(z^{k+1} - z^{\star} = \tau_{k} (z^{k} + d^{k} - z^{\star}) + (1-\tau_{k}) (w^{k} - z^{\star}) \) for each \(k\in\naturals\), we get that
        {\begin{allowdisplaybreaks}
        \begin{align} \nonumber
            & \norm{z^{k+1} - z^\star}^2 \\ \nonumber
            & \quad =
            \tau_k\norm{z^{k}+ d^k - z^\star}^2 
            + (1- \tau_k)\norm{w^k-z^\star}^2 
            - \tau_k(1- \tau_k)\norm{z^{k} + d^{k} - w^k}^2 \\ \nonumber
            & \quad \leq 
            \tau_k\norm{z^{k} + d^k- z^\star}^2 
            + (1- \tau_k)\norm{z^k-z^\star}^2  
            - (1- \tau_k)\alpha\p{\gamma, L_F} \V_k \\ \nonumber
            & \quad \leq 
            \tau_k\p{\norm{z^{k}- z^\star}+\norm{d^k}}^2 + (1- \tau_k)\norm{z^k-z^\star}^2  - (1- \tau_k)\alpha\p{\gamma, L_F} \V_k  \\ \label{eq:zk_quasi_fejer_}
            & \quad \leq
            \norm{z^k-z^\star}^2 
            + 2\norm{z^{k}- z^\star}\norm{d^k}
            + \norm{d^k}^2 - (1- \tau_k)\alpha\p{\gamma, L_F} \V_k \\ \label{eq:zk_bounded_}
            & \quad  \leq \p{\norm{z^k-z^\star} + \norm{d^k}}^2
        \end{align}
        \end{allowdisplaybreaks}}%
        for each \(k\in\naturals\), where~\eqref{eq:EG:primal_descent_for_V}  and \(\tau_k(1-\tau_k)\geq 0\) is used in the first inequality, the triangle inequality is used in the second inequality, \(\tau_{k} \leq 1\) is used in the third inequality, and \( (1- \tau_k)\alpha\p{\gamma, L_F} \geq 0\) is used in the last inequality. Taking the square root of~\eqref{eq:zk_bounded_} and inductively applying the resulting inequality gives that 
        \begin{align}\label{eq:zk_bounded}
            \norm{z^k-z^\star} \leq \norm{z^0-z^\star} + \textstyle\sum_{i=0}^{k-1} \norm{d^{i}} \leq \underbracket{\norm{z^0-z^\star} + \textstyle\sum_{i=0}^{\infty} \norm{d^{i}}}_{=E} < \infty
        \end{align}
        for each \(k\in\naturals\), where the empty sum is interpreted as zero and \(E\) is finite since \(\seq{d^k}\in\ell^{1}(\naturals;\,\calH)\). Therefore,~\eqref{eq:zk_quasi_fejer_} and~\eqref{eq:zk_bounded} imply that 
        \begin{align*}
            \norm{z^{k+1}-z^\star}^2 {}\leq{}& \norm{z^k-z^\star}^2  + \underbracket{2E \norm{d^k} +\norm{d^k}^2}_{=\varepsilon_k} - (1- \tau_k)\alpha\p{\gamma, L_F} \V_k
        \end{align*}
        for each \(k\in\naturals\), where summability of \(\seq{\varepsilon_k}\) follows from \(\seq{d^k}\in\ell^{1}(\naturals;\,\calH)\). 
    \end{proof}
\end{lemma}

\begin{lemma}\label{lem:(Prox)(I)FLEX}
    Suppose that \Cref{ass:main} holds, the sequences \(\p{z^{k}}_{k\in\naturals}\), \(\p{\bar{z}^{k}}_{k\in\naturals}\) and \(\p{w^{k}}_{k\in\naturals}\) are generated by either~\refFLEX{},~\refIFLEX, or~\refProxFLEX, and \(\V\) is the Lyapunov function given in~\eqref{eq:V}. Then, the following hold.
    \begin{enumeratlem}
    
        \item \label{lem:(Prox)(I)FLEX:norm_F_or_V_convergent} 
        \(\seq{\V\p{z^{k},\bar{z}^{k},w^{k}}}\) is convergent, which for~\FLEX{} and~\IFLEX{} reduces to \(\seq{F\p{z^{k}}}\) being convergent.
        
        \item \label{lem:(Prox)(I)FLEX:residual_summable} \(\seq{\norm{w^{k}-\bar{z}^{k}}^{2}}\in\ell^{1}\p{\naturals;\,\reals_{+}}\), which for~\FLEX{} and~\IFLEX{} can be written as \(\seq{\norm{F\p{z^{k}}-F\p{\bar{z}^{k}}}^{2}}\in\ell^{1}\p{\naturals;\,\reals_{+}}\).

        \item \label{lem:(Prox)(I)FLEX:uniformly} If \(F\) is uniformly monotone, then \(\seq{\norm{F( z^{k})}}\) converges to zero for~\FLEX{} and~\IFLEX.
    \end{enumeratlem}
    \begin{proof}

        First, we establish for~\ProxFLEX{} that 
        \begin{align}\notag
            \V\p{z^{k+1},\bar{z}^{k+1},w^{k+1}} \leq{}& \V\p{z^{k},\bar{z}^{k},w^{k}} \\ \label{eq:ProxFLEX:has_dual_descent}
            &{}- \min(\p{1- \gamma \LF}\p{1-\rho^{2}},\sigma\p{1-\gamma^{2}L_{F}^{2}}) \gamma^{-2}\norm{w^{k}-\bar{z}^{k}}^{2} 
        \end{align}
        for each \(k\in\naturals\). Note that \Cref{State:Prox-FLEX:contraction} in~\ProxFLEX{} implies that
        {\begin{allowdisplaybreaks}
        \begin{align*}
            \V\p{z^{k+1},\bar{z}^{k+1},w^{k+1}} 
            \leq{} & \rho^{2}\V\p{z^{k},\bar{z}^{k},w^{k}} \\
            ={} & \V\p{z^{k},\bar{z}^{k},w^{k}} - \p{1-\rho^{2}}\V\p{z^{k},\bar{z}^{k},w^{k}} \\
            \leq{} & \V\p{z^{k},\bar{z}^{k},w^{k}} - \p{1- \gamma \LF}\p{1-\rho^{2}}\gamma^{-2}\norm{w^{k}-\bar{z}^{k}}^{2} 
        \end{align*}
        \end{allowdisplaybreaks}}%
        for each iteration \(k\) when the condition in \Cref{State:Prox-FLEX:contraction} of~\ProxFLEX{} is true, where \Cref{prop:V_performance_measure:lower_bound} is used in the last inequality. This combined with~\eqref{eq:Prox-FLEX:descent} in~\ProxFLEX{} gives~\eqref{eq:ProxFLEX:has_dual_descent}.

        Second, since~\ProxFLEX{} reduced to~\FLEX{} when \(g=0\),~\eqref{eq:ProxFLEX:has_dual_descent} implies that
        \begin{align}\label{eq:FLEX:has_dual_descent}
            \norm{F\p{z^{k+1}}}^{2} \leq{}& \norm{F\p{z^{k}}}^{2} - \min(\p{1- \gamma \LF}\p{1-\rho^{2}},\sigma\p{1-\gamma^{2}L_{F}^{2}}) \norm{F\p{z^{k}}-F\p{\bar{z}^{k}}}^{2}
        \end{align}
        for each \(k\in\naturals\), for~\FLEX.
        
        \begin{itemize}

            \item[\ref{lem:(Prox)(I)FLEX:norm_F_or_V_convergent}:] Follows from~\eqref{eq:FLEX:has_dual_descent} for~\FLEX{},~\eqref{eq:I-FLEX:descent} for~\IFLEX, and~\eqref{eq:ProxFLEX:has_dual_descent} for~\ProxFLEX, combined with the monotone convergence theorem.
            
            \item[\ref{lem:(Prox)(I)FLEX:residual_summable}:] Note that \(\seq{\norm{w^{k}-\bar{z}^{k}}^{2}} = \seq{ \gamma^{2} \norm{F\p{z^{k}}-F\p{\bar{z}^{k}}}^{2} } \) for~\FLEX{} and~\IFLEX. The statement follows from~\eqref{eq:FLEX:has_dual_descent} for~\FLEX{},~\eqref{eq:I-FLEX:descent} for~\IFLEX, and~\eqref{eq:ProxFLEX:has_dual_descent} for~\ProxFLEX, combined with a telescoping summation argument. 

            \item[\ref{lem:(Prox)(I)FLEX:uniformly}:] Suppose that \(F\) is uniformly monotone, i.e., there exists an increasing function \(\phi : \reals_{+} \to \reals_{+}\), that vanishes only at \(0\),  such that 
            \begin{align*}
                \phi\left(\norm{x-y}\right) \leq \langle x - y , F(x) - F(y) \rangle
            \end{align*}
            for each \(x,y\in \calH\). Note that
            {\begin{allowdisplaybreaks}
            \begin{align*}
                \phi\left(\gamma\norm{F\p{z^{k}}}\right) 
                    &= \phi\left(\norm{z^{k} - \bar{z}^{k}}\right) \\
                    &\leq \langle z^{k} - \bar{z}^{k} , F\p{z^{k}} - F\p{\bar{z}^{k}} \rangle \\
                    &\leq \norm{ z^{k} - \bar{z}^{k}} \norm{ F\p{z^{k}} - F\p{\bar{z}^{k}} } \\
                    &= \gamma \norm{ F\p{z^{k}} } \norm{ F\p{z^{k}} - F\p{\bar{z}^{k}} } \xrightarrow[k\rightarrow \infty]{} 0, 
            \end{align*}
            \end{allowdisplaybreaks}}%
            where \(\bar{z}^{k} = z^{k} - \gamma F(z^{k})\) is used in the first equality, the Cauchy--Schwarz inequality is used in the second inequality, \(\bar{z}^{k} = z^{k} - \gamma F(z^{k})\) is used in the last equality, and the convergence to zero in the last line follows from \Cref{lem:(Prox)(I)FLEX:norm_F_or_V_convergent} and \Cref{lem:(Prox)(I)FLEX:residual_summable}. This proves the claim. \qedhere
        \end{itemize}
    \end{proof}
\end{lemma}

\subsection{Proofs regarding~\FLEX}

\begin{proof}[Proof of \Cref{thm:FLEX:dk_summable}.]
    This follows from \Cref{thm:ProxFLEX}, since~\refProxFLEX{} reduces to~\refFLEX{} when \(g=0\). 
\end{proof}

\begin{proof}[Proof of \Cref{thm:FLEX:F_unifrom}.]
    See \Cref{lem:(Prox)(I)FLEX:uniformly}. 
\end{proof}

\begin{proof}[Proof of \Cref{thm:FLEX:F_strongly}.]
    Note that~\eqref{eq:FLEX:descent} gives that
    {\begin{allowdisplaybreaks}
    \begin{align}\notag
        \norm{F\p{z^{k+1}}}^2 
        \leq &  \norm{F\p{z^{k}}}^2 -  \sigma\p{1-\gamma^{2}L_{F}^{2}} \norm{F\p{z^{k}} - F\p{\bar{z}^{k}}}^2 \\\notag
        \leq &  \norm{F\p{z^{k}}}^2  - \sigma\mu_{F}^{2}\p{1-\gamma^{2}L_{F}^{2}} \norm{z^{k} - \bar{z}^{k}}^{2}  \\\label{eq:FLEX:dual_descent_plus_strong_give_contraction}
        = & \p{1 - \sigma\gamma^{2}\mu_{F}^{2}\p{1-\gamma^{2}L_{F}^{2}}  } \norm{F\p{z^{k}}}^2
    \end{align}
    \end{allowdisplaybreaks}}%
    for each iteration \(k\) such that the condition in \Cref{State:FLEX:contraction} in~\FLEX{} is false, where \Cref{State:FLEX:EG} in~\FLEX{} is used in the last equality. Combining~\eqref{eq:FLEX:dual_descent_plus_strong_give_contraction} and \Cref{State:FLEX:contraction} in~\FLEX{} gives~\eqref{eq:FLEX:F_strongly:linear}. Moreover, \(\seq{\norm{F(z^k)}}\) converges to zero since \( \max(\rho^2, 1- \sigma\gamma^{2}\mu_{F}^{2}\p{1-\gamma^{2}L_{F}^{2}} ) \in (0,1)\). Since
    \begin{align*}
        \norm{z^{k} - z^{\star}} \leq \mu_{F}^{-1} \norm{F(z^{k}) - F(z^{\star})} = \mu_{F}^{-1} \norm{F(z^{k})}
    \end{align*}
    for each \(k\in\naturals\), \(\seq{z^k}\) converges strongly to \(z^{\star}\in \zer(F)\). 
\end{proof}

\subsection{Proofs regarding~\IFLEX}

\begin{proof}[Proof of \Cref{thm:I-FLEX:injective}.]
    Suppose that \(\seq{z^{k}}[k\in K]\rightharpoonup z^\infty\). Weak continuity of \(F\) and \(\bar{z}^{k} = z^{k} - \gamma F\p{z^{k}}\) give that \(\seq{\bar{z}^{k}}[k\in K] \rightharpoonup \bar{z}^\infty = z^\infty - \gamma F(z^\infty)\). On the other hand, it follows from  \Cref{lem:(Prox)(I)FLEX:residual_summable} and weak continuity of \(F\) that \(F(z^\infty) = F(\bar z^\infty)\), which in view of the injectivity assumption of \(F\), implies \(z^\infty = \bar z^\infty = z^\infty - \gamma F(z^\infty)\). Therefore, \(z^\infty \in \zer\p{F}\), as claimed. 
\end{proof}

\begin{proof}[Proof of \Cref{thm:I-FLEX:dk_bounded}.]
    Note that \Cref{thm:I-FLEX:injective} gives that each weak sequential cluster point of  \(\seq{z^k}\) is in \(\zer \p{F}\). Moreover,~\cite[Lemma 5.31]{bauschke2017convexanalysismonotone} and~\eqref{eq:zk_quasi_fejer} in \Cref{lem:zk_quasi_fejer} give that \(\seq{\norm{z^k-z^\star}}\) converges. Thus,~\cite[Lemma 2.47]{bauschke2017convexanalysismonotone} gives that \(\seq{z^k}\) converges weakly to some point in \(\zer \p{F}\), as claimed. 
\end{proof}

\begin{proof}[Proof of \Cref{thm:I-FLEX:F_unifrom}.]
    See \Cref{lem:(Prox)(I)FLEX:uniformly}. 
\end{proof}

\begin{proof}[Proof of \Cref{thm:I-FLEX:F_strongly}.]
    Note that~\eqref{eq:I-FLEX:descent} gives that
    {\begin{allowdisplaybreaks}
    \begin{align*}
        \norm{F\p{z^{k+1}}}^2 
        \leq &  \norm{F\p{z^{k}}}^2 -  \sigma\p{1-\gamma^{2}L_{F}^{2}} \norm{F\p{z^{k}} - F\p{\bar{z}^{k}}}^2 \\
        \leq &  \norm{F\p{z^{k}}}^2  - \sigma\mu_{F}^{2}\p{1-\gamma^{2}L_{F}^{2}} \norm{z^{k} - \bar{z}^{k}}^{2}  \\
        = & \p{1 - \sigma\gamma^{2}\mu_{F}^{2}\p{1-\gamma^{2}L_{F}^{2}}  } \norm{F\p{z^{k}}}^2
    \end{align*}
    \end{allowdisplaybreaks}}%
    for each \(k\in\naturals\), where \Cref{State:I-FLEX:EG} in~\IFLEX{} is used in the last equality. Therefore, \(\seq{\norm{F(z^k)}}\) converges to zero since \( 1 - \sigma\gamma^{2}\mu_{F}^{2}\p{1-\gamma^{2}L_{F}^{2}}  \in (0,1)\). Since
    \begin{align*}
        \norm{z^{k} - z^{\star}} \leq \mu_{F}^{-1} \norm{F(z^{k}) - F(z^{\star})} = \mu_{F}^{-1} \norm{F(z^{k})}
    \end{align*}
    for each \(k\in\naturals\), \(\seq{z^k}\) converges strongly to \(z^{\star}\in \zer(F)\). 
\end{proof}

\subsection{Proofs regarding~\ProxFLEX}

\begin{proof}[Proof of \Cref{thm:ProxFLEX}.]
    Set \(\tau_{k} = 1\) for the iterations when the condition in \Cref{State:Prox-FLEX:contraction} in~\ProxFLEX{} is true and let \(z^{\star}\in\zer \p{F + \partial g}\). Then~\eqref{eq:zk_quasi_fejer} in \Cref{lem:zk_quasi_fejer} and {\cite[Lemma 5.31]{bauschke2017convexanalysismonotone}} imply that \(\seq{\norm{z^k-z^\star}}\) converges. Thus, the proof is complete if we can show that weak sequential cluster points of \(\seq{z^{k}}\) belong to \(\zer\p{F+\partial g}\), due to {\cite[Lemma 2.47]{bauschke2017convexanalysismonotone}}. 
    
    For this, it suffices to show that \(\seq{\norm{\bar{z}^{k}-z^{k}}}\) converges to zero. Indeed, suppose that \(\p{z^{k}}_{k\in K}\rightharpoonup z^{\infty}\) for some \(z^{\infty} \in \calH\) and \(\seq{\norm{\bar{z}^{k}-z^{k}}}\) converges to zero. Then \(\p{\bar{z}^{k}}_{k\in K}\rightharpoonup z^{\infty}\). Moreover, the proximal evaluation in \Cref{State:Prox-FLEX:EG:1} in~\ProxFLEX{} can equivalently be written as 
    \begin{align}\label{Prox-FLEX:EG:1:rewritten}
        \gamma^{-1}\p{z^{k}- \bar{z}^{k} } - F\p{z^{k}} + F\p{\bar{z}^{k}} \in \p{F + \partial g}\p{\bar{z}^{k}}.
    \end{align}
    The left-hand side of~\eqref{Prox-FLEX:EG:1:rewritten} converges strongly to zero since \(F\) is continuous and \((\norm{z^{k}-\bar{z}^{k}})_{k\in\naturals}\) converges to zero. Moreover, the operator \(F + \partial g\) is maximally monotone, since \(F\) is maximally monotone (by continuity and monotonicity~\cite[Corollary 20.28]{bauschke2017convexanalysismonotone}), \(\partial g\) is maximally monotone~\cite[Theorem 20.48]{bauschke2017convexanalysismonotone}, and \(F\) has full domain~\cite[Corollary 25.5]{bauschke2017convexanalysismonotone}. Thus,~\cite[Proposition 20.38]{bauschke2017convexanalysismonotone} gives that \(z^{\infty} \in \zer\p{F + \partial g}\), and by~\cite[Lemma 2.47]{bauschke2017convexanalysismonotone} we conclude that \(\p{z^{k}}_{k\in\naturals}\) converges weakly to a point in \(\zer\p{F + \partial g}\).

    It remains to show that \(\seq{\norm{\bar{z}^{k}-z^{k}}}\) converges to zero, which we do by showing that \(\seq{\mathcal{V}_{k}}\) converges to zero and applying \Cref{prop:V_performance_measure:lower_bound}. Let  \(K_{<1} = \set{k\in \naturals \mid \tau_k < 1} \). Suppose that \(\abs{K_{<1}} < \infty\). Then \(\V_{k+1} \leq \rho^{2}\V_{k}\) for each \(k \in \naturals\) such that \(k > \max K_{<1}\), and \(\seq{\V_{k}}\) converges to zero since \(\rho \in (0,1)\). On the contrary, suppose that \(\abs{K_{<1}} = \infty\). Let \(\Gamma : K_{<1} \to K_{<1}\) such that \(\Gamma(k) = \min\set{i \in K_{<1} \mid k < i } \) for each \(k\in K_{<1}\). Let \(k\in K_{<1}\), and notice that \(\tau_k \leq \bar{\tau}\) for any such index, where \(\bar{\tau} = \max \set{ \beta^{i} \mid i \in \llbracket 1 , M \rrbracket}\cup\set{0}<1\). Inductively summing~\eqref{eq:zk_quasi_fejer} in \Cref{lem:zk_quasi_fejer} from \(k\) to \(\Gamma(k)-1\) gives
    \begin{align}\label{eq:Prox-FLEX:zk_quasi_fejer:sum}
        \norm{z^{\Gamma(k)} - z^\star}^2 
        \leq& 
        \norm{z^k-z^\star}^2
        - (1- \bar{\tau})\alpha\p{\gamma, L_F} \V_k +  \sum_{i=k}^{\Gamma(k)-1}\varepsilon_i, 
    \end{align}
    where we used the fact that \(\tau_i = 1\) for any \(i \in K_{1}\). Inductively summing over all \(k \in K_{<1}\) in~\eqref{eq:Prox-FLEX:zk_quasi_fejer:sum}, rearranging, and dividing by \((1- \bar{\tau})\alpha\p{\gamma, L_F}>0\) gives 
    {\begin{allowdisplaybreaks}
    \begin{align} \notag
        \sum_{k\in K_{<1}}  \V_{k} \leq{}&
        \frac{\sum_{k\in K_{<1}} \p{ \norm{z^k-z^\star}^2 - \norm{z^{\Gamma(k)} - z^\star}^2 +  \sum_{i=k}^{\Gamma(k)-1}\varepsilon_i  } }{(1- \bar{\tau})\alpha\p{\gamma, L_F}} \\ \label{eq:Prox-FLEX:partial_F_sum_bounded}
        \leq{}&
        \frac{\norm{z^{\min(K_{<1}) }-z^\star}^2 + \sum_{k=0}^{\infty} \varepsilon_k  }{(1- \bar{\tau})\alpha\p{\gamma, L_F}} < \infty,
    \end{align}
    \end{allowdisplaybreaks}}%
    where summability of \(\seq{\varepsilon_k}\) is used in the last inequality. Note that 
    {\begin{allowdisplaybreaks}
    \begin{align*}
        \sum_{k=0}^{\infty} \V_{k} 
        &= \sum_{k\in K_{<1}} \sum_{i=k}^{\Gamma(k)-1} \V_{i} \\
        & \leq\sum_{k\in K_{<1}} \sum_{i=k}^{\Gamma(k)-1} \rho^{2(i - k)}\V_{k} \\
        & \leq \frac1{1-\rho^{2}}\sum_{k\in K_{<1}}  \V_{k}
        < \infty,
    \end{align*}
    \end{allowdisplaybreaks}}%
    where \Cref{State:Prox-FLEX:contraction} in~\ProxFLEX{} is used in the first inequality, the expression for the geometric series is used in the second inequality, and~\eqref{eq:Prox-FLEX:partial_F_sum_bounded} is used in the last inequality. This completes the proof. 
\end{proof}

\section{Superlinear convergence}\label{sec:superlinear}
The convergence analyses presented so far have been blind to the choice of directions \(\seq{d^k}\); nevertheless, attaining a fast convergence rate relies on their precise choice. This section presents a minimal set of assumptions on the directions that ensure superlinear convergence. Our main focus will be on quasi-Newton-type directions that are computed as 
\begin{equation}\label{eq:direction:Rgam}
    d^k = - H_{k} R_\gamma(z^k),\quad \text{where} \quad R_\gamma  = \tfrac1{\gamma} \p{\Id - \prox_{\gamma g} \circ \p{\Id - \gamma F} },
\end{equation}
\(\gamma\in\reals_{++}\), \(H_{k}:\calH\to\calH\) is a linear operator encapsulating information of the geometry of the residual mapping \(R_\gamma\) at \(z^k\), and \(F\) and \(g\) satisfy \Cref{ass:main}. The specific way \(H_k\) is computed determines the underlying quasi-Newton method (see \Cref{sec:numerical} for details). Notice that the zeros of \(R_\gamma\) coincide with the set of solutions of~\eqref{eq:the_inclusion_problem}. Moreover, when \(g = 0\), \(R_\gamma\) reduces to \(F\), and the directions are given by \(d^k = - H_{k} F(z^k)\). The following assumption on the directions \(\seq{d^k}\) can be seen as a boundedness assumption on the linear operators \(\seq{H_k}\). However, note that the assumption applies to directions beyond the ones given in~\eqref{eq:direction:Rgam}. 

\begin{assumption}\label{ass:dk:bound}%
    The sequence of directions \(\seq{d^k}\) used in~\FLEX{},~\IFLEX{}, or~\ProxFLEX{} satisfies \( \norm{d^k} \leq D \norm{R_\gamma(z^k)} \) for each \(k\in\naturals\) such that \(k \geq K\), for some constants \(D \geq 0\) and \(K \in \naturals\), where \(R_\gamma\) denotes the residual operator as defined in~\eqref{eq:direction:Rgam} (the function \(g\) is set to zero in the particular cases of~\FLEX{} and~\IFLEX).
\end{assumption}

\Cref{ass:dk:bound} is a natural assumption for directions defined in~\eqref{eq:direction:Rgam}. For example, under suitable regularity conditions for regularized Newton directions—specifically when \(g=0\)—we demonstrate this in \Cref{prop:dk_Newton}. Note that in \Cref{prop:dk_Newton}, we assume that \(F\) is continuously Fréchet differentiable; however, this assumption is made solely for illustrative purposes and is not required elsewhere in the paper. \Cref{ass:dk:bound} has also been utilized in the context of minimization and in finding zeros of nonexpansive maps, as seen in~\cite[Theorem 5.7.A3]{ahookhosh2021bregmanforwardbackward} and~\cite[Assumption 2]{themelis2019supermannsuperlinearlyconvergent}, respectively.

\begin{proposition}\label{prop:dk_Newton}  
    Let \(F:\calH\to \calH\) be monotone and continuously Fr\'echet differentiable, and suppose that the Fr\'echet derivative \(\mathsf{D} F\) at \(z^\star \in \zer\p{F}\) is left invertible. 
    Suppose that \(\seq{z^k},\seq{d^k}\in\calH^{\naturals}\) are such that
    \begin{equation}\label{eq:reg_Newton_dir}
        \p{r_k\Id + \mathsf{D} F(z^k)}d^k = - F(z^k)
    \end{equation}
    for some sequence \(\seq{r_k} \in \reals_{++}^{\naturals}\), and that \(\seq{z^k}\) converges strongly to \(z^\star\). Then, \Cref{ass:dk:bound} is satisfied with \(g =0\). 
    \begin{proof}
        Let \(t>0\) and note that
        \begin{align}\label{eq:DF:monotone}
            0 \leq \frac{\inner{ F(z + t v) - F(z)}{ z + t v - z } }{t^2} \xrightarrow[t	\downarrow 0]{} \inner{ \mathsf{D} F(z)v}{ v }
        \end{align}
        for any \(v\in \calH\) and for any \(z \in \calH\), by monotonicity of \(F\). This implies that the bounded linear operator  \(r_k\Id + \mathsf{D} F(z^k)\) is \(r_k\)-strongly monotone, and therefore invertible for each \(k \in \naturals\). This, in turn, ensures that the regularized Newton update~\eqref{eq:reg_Newton_dir} is well-defined, i.e., \(d^{k}\) is uniquely defined at each iteration.
        
        Moreover, since \(\mathsf{D} F(z^\star)\) is left invertible, there exists \(c_1>0\) such that \(\norm{\textstyle\mathsf{D} F(z^\star) v} \geq c_1\norm{v}\) for any \(v\in \calH\)~\cite[Proposition 10.29]{axler2020measureintegration}. This observation combined with \(z^k \to z^\star\) and continuity of \(\mathsf{D} F(\cdot)\) implies that there exists \(c_2 >0\) and \(K\in\naturals\) such that \(\norm{\mathsf{D} F(z^k) v} \geq c_2\norm{v}\) for any \(v\in \calH\) and for any \(k\geq K\). Therefore, 
        {\begin{allowdisplaybreaks}
        \begin{align*}
           \norm{F(z^k)}^{2} {}={}&  \norm{r_k\Id + \mathsf{D} F(z^k)d^k}^2 \\
           {}={}& \norm{r_k d^k}^2 + 2r_k \inner{ \mathsf{D} F(z^k)d^k}{ d^k} + \norm{\mathsf{D} F(z^k)d^k}^2 \\
            {}\geq{}& \norm{\mathsf{D} F(z^k)d^k}^2 \\
            {}\geq{}& c_2^2\norm{d^k}^2
        \end{align*}
        \end{allowdisplaybreaks}}%
        for each \(k\geq K\), where the first inequality follows from~\eqref{eq:DF:monotone} and \(r_k>0\). This establishes that \Cref{ass:dk:bound} is satisfied with \(D = 1/c_2\), when \(g=0\). 
    \end{proof}        
\end{proposition}

\begin{remark}
    In the case of~\FLEX, when the operator is strongly monotone, the sequence \(\seq{\norm{F(z^k)}}\) converges \(Q\)-linearly to zero, as established in \Cref{thm:FLEX:F_strongly}. This observation, combined with \Cref{ass:dk:bound} is sufficient to conclude that \(\seq{d^k} \in \ell^1(\naturals;\,\calH)\), thereby yielding global convergence as demonstrated in \Cref{thm:FLEX:dk_summable}. An analogous argument extends to~\ProxFLEX{} after incorporating the strengthening discussed in \Cref{rem:strength:proxflex}.
\end{remark}

We proceed to quantify the quality of the directions used in the algorithms that guarantee fast convergence. The classical condition of~\cite[Chapter 7.5]{facchinei2004finitedimensionalvariationalII} for Newton-type methods identifies a sequence of directions \(\seq{d^k}\) relative to a sequence \(\seq{z^k}\) converging to \(z^\star\) as superlinear if 
\begin{equation}\label{eq:Suplin:Facchinei}
        \lim_{k\to \infty} \frac{\norm{z^k + d^k - z^\star}}{\norm{z^k - z^\star}} = 0.
\end{equation}
This notion a priori assumes the convergence of the sequence \(\seq{z^k}\). Here, we use a slightly refined notion and define superlinear directions similar to~\cite[Definition VI.2]{themelis2019supermannsuperlinearlyconvergent}. 
\begin{definition}\label{def:superlin}%
    Suppose that \(\gamma\in\reals_{++}\), \((z^k)_{k\in\N}, (d^k)_{k\in\N} \in \calH^{\naturals}\), \Cref{ass:main} holds, \(T_1^{\gamma}\) and \(T_2^{\gamma}\) are the algorithmic operators defined in~\eqref{eq:T1T2}, and \(\V\) is the Lyapunov function given in~\eqref{eq:V}. Then we say that the sequence of directions \((d^k)_{k\in\N}\) is \emph{superlinear relative to} \((z^k)_{k\in\N}\) if 
    \begin{equation}\label{eq:dir:superlin}
        \lim_{k\to\infty}\frac{\V(z^k+d^k, T_1^\gamma(z^k+d^k), T_2^\gamma(z^k+d^k))}{\V(z^k, T_1^\gamma(z^k), T_2^\gamma(z^k))} = 0.
    \end{equation}
\end{definition}

In the case of solving monotone equations where \(g = 0\), addressed by~\FLEX{} and~\IFLEX{},~\eqref{eq:dir:superlin} reduces to 
\begin{equation}\label{def:eq:dir:superlin:F}
    \lim_{k\to\infty}{\frac{\norm{F(z^k + d^k)}}{\norm{F(z^k)}}} = 0.
\end{equation}
\new{This condition is closely related to the classical Dennis-Mor\'e assumption \cite{dennis1977quasinewtonmethods}. Specifically, when \(F\) is strictly differentiable at its zeros, the Dennis-Mor\'e assumption implies \eqref{def:eq:dir:superlin:F}, as shown in \cite[Theorem VI.7]{themelis2019supermannsuperlinearlyconvergent}. In particular, the condition is satisfied by Broyden’s method under mild regularity assumptions at the limit points \cite[Theorem VI.8]{themelis2019supermannsuperlinearlyconvergent}. 
}

 \begin{remark}\label{rem:superlin:DM}
     The superlinear convergence results presented in \Cref{thm:suplin} also hold under~\eqref{eq:Suplin:Facchinei} of~\cite{facchinei2004finitedimensionalvariationalII} since it implies the notion in \Cref{def:superlin}. Indeed, by \Cref{ass:dk:bound}
     {\begin{allowdisplaybreaks}
    \begin{align*}
        \norm{d^k}^2 
        \leq{}&
        D^2 \left(\norm{z^k - T_2^\gamma(z^k)} + \norm{T_2^\gamma(z^k) - T_1^\gamma(z^k)} \right)^2
        \\
        \leq{}&
        2D^2 \left(\norm{z^k - T_2^\gamma(z^k)}^2 + \norm{T_2^\gamma(z^k) - T_1^\gamma(z^k)}^2 \right) \\
        \leq{}& \frac{2\gamma^2D^2}{1-\gamma L_F} \V(z^k, T_1^\gamma(z^k), T_2^\gamma(z^k))
    \end{align*}
    \end{allowdisplaybreaks}}%
    for each \(k\in\naturals\) such that \(k \geq K\), where the triangle inequality is used in the first inequality and \Cref{prop:V_performance_measure:lower_bound} is used in the last inequality. Hence
    \begin{align}\label{eq:superlinear:ineq}
        \frac{\V(z^k+d^k, T_1^\gamma(z^k+d^k), T_2^\gamma(z^k+d^k))}{\V(z^k, T_1^\gamma(z^k), T_2^\gamma(z^k))}
        {}\leq{}  
        \frac{2\gamma^2D^2}{(1-\gamma L_F)\alpha\p{\gamma, \LF}}
        \frac{\norm{z^k+d^k - z^\star}^2}{\norm{d^k}},
    \end{align}
    for each \(k\in\naturals\) such that \(k \geq K\), where 
    \(\V(z^k+d^k, T_1^\gamma(z^k+d^k), T_2^\gamma(z^k+d^k)) \leq \norm{z^k+d^k - z^\star}^2/ \alpha\p{\gamma, L_F} \) is used (see \Cref{thm:EG:primal_descent_for_V}). Combining~\eqref{eq:superlinear:ineq} with~\eqref{eq:Suplin:Facchinei} and the fact that \(\lim_{k \to \infty} \norm{z^k - z^\star}/\norm{d^k} = 1\) (see~\cite[Lemmma 7.5.7]{facchinei2004finitedimensionalvariationalII}) shows that the ratio on the left-hand-side of~\eqref{eq:superlinear:ineq} vanishes. Therefore,~\eqref{eq:dir:superlin} is a weaker condition than~\eqref{eq:Suplin:Facchinei} under \Cref{ass:dk:bound}. We also refer the reader to~\cite{facchinei2004finitedimensionalvariationalII} for further details.
 \end{remark}

\new{As shown below in \Cref{thm:suplin:summable}, \Cref{def:superlin} in conjunction with \Cref{ass:dk:bound} is sufficient to conclude \(\seq{d^k}\in\ell^{1}(\naturals;\,\calH)\), establishing global weak sequential convergence by \Cref{thm:ProxFLEX}. See also \Cref{thm:FLEX:dk_summable,thm:I-FLEX:dk_bounded} for~\FLEX{} and~\IFLEX, respectively.}

\begin{theorem}\label{thm:suplin}
Suppose that \Cref{ass:main} and \Cref{ass:dk:bound} hold, \(\zer\p{F + \partial g }\neq \emptyset\),  \(T_1^{\gamma}\) and \(T_2^{\gamma}\) are the algorithmic operators defined in~\eqref{eq:T1T2}, \(\V\) is the Lyapunov function given in~\eqref{eq:V}, and \(\seq{d^k}\) is superlinear relative to the sequence \(\seq{z^k}\) generated by either~\refFLEX{},~\refIFLEX{}, or~\refProxFLEX{}. 
Then, the following hold. 
\begin{thmenumerate}
      \item \label{thm:suplin:full_accepted}  \(z^{k+1} = z^k + d^k\) for all \(k\in\naturals\) sufficiently large. 
      \item \label{thm:suplin:Q_superlinear} \(\seq{\V(z^k, T_1^\gamma(z^k), T_2^\gamma(z^k))}\) converges to zero at least \(Q\)-superlinearly.  
      \item \label{thm:suplin:summable} \(\seq{d^k}\in\ell^{1}(\naturals;\,\calH)\) with \(\seq{d^k}\) converging to zero at least \(R\)-superlinearly.  
      \item \label{thm:suplin:R_superlinear} If \(\dim \calH < \infty\), then \(\seq{z^k}\) converges to some point \(z^\star \in \zer\p{F + \partial g }\) at least \(R\)-superlinearly.
  \end{thmenumerate}  
\end{theorem}

\begin{proof}
    The proof is presented for~\ProxFLEX{}, with the necessary adjustments for~\FLEX{} and~\IFLEX{} outlined at the end of the proof. 
    \begin{itemize}
        \item[\ref{thm:suplin:full_accepted}] Follows from~\eqref{eq:dir:superlin} since \Cref{State:Prox-FLEX:contraction} in~\ProxFLEX{} is true for all \(k\in\naturals\) sufficiently large.
        \item[\ref{thm:suplin:Q_superlinear}] Follows from~\ref{thm:suplin:full_accepted} and~\eqref{eq:dir:superlin}.
        \item[\ref{thm:suplin:summable}] Note that \Cref{ass:dk:bound} and \Cref{prop:V_performance_measure:lower_bound} give that 
        \begin{align*}
            \norm{d^k} \leq \frac{D}{\gamma}\norm{z^{k} - T_1^\gamma(z^k)} \leq \frac{2D}{\sqrt{1-\gamma L_F}}  \sqrt{\V(z^{k}, T_1^\gamma(z^k), T_2^\gamma(z^k))}
        \end{align*}
        for each \(k\in\naturals\) such that \(k\geq K\). The claim now follows from~\ref{thm:suplin:Q_superlinear}.
        \item[\ref{thm:suplin:R_superlinear}] \Cref{thm:ProxFLEX} and~\ref{thm:suplin:summable} imply that the sequence \(\seq{z^k}\) converges to some point \(z^\star\in \zer \p{F + \partial g }\). Since \(z^{k+1}-z^{k} = d^{k}\) for all \(k\in\naturals\) sufficiently large,~\ref{thm:suplin:summable} implies that \(\seq{z^{k+1} - z^k}\) converges to zero at least \(R\)-(super)linearly. In particular, there exists \(\kappa \in \reals_{++}\) and \(\seq{c_k} \in \reals_{++}^{\naturals}\) such that \(\lim_{k\to\infty}c_{k} = 0\) and 
        \begin{align}\label{eq:superlinear:diff_R_linear}
            \norm{z^{k+1} - z^k} \leq \kappa \prod_{i=1}^{k} c_i
        \end{align}
        for each \(k\in\naturals\). Let \(k,j\in\naturals\) such that \(j>k\). The triangle inequality and~\eqref{eq:superlinear:diff_R_linear} give that
        \begin{align*}
            \norm{z^k - z^\star} \xleftarrow[j \rightarrow \infty]{} \norm{z^k - z^j} \leq \sum_{\ell=k}^{j-1} \norm{z^{\ell}- z^{\ell+1}} \leq \kappa \sum_{\ell=k}^{j-1} \prod_{i=1}^{\ell} c_i \xrightarrow[j \rightarrow \infty]{} \kappa \sum_{\ell=k}^{\infty} \prod_{i=1}^{\ell} c_i = \mu_k.
        \end{align*}
        The sequence \(\seq{\mu_k}\in\reals_{++}^{\naturals}\) converges to zero at least \(Q\)-superlinearly since 
        \begin{align*}
            \frac{\mu_k}{\mu_{k-1}}
            =
            \frac{\sum_{\ell=k}^{\infty} \prod_{i=1}^{\ell} c_i}{\sum_{\ell=k-1}^{\infty} \prod_{i=1}^{\ell} c_i}
            =
            \frac{%
                \big(%
                    \prod_{i=1}^{k-1} c_i
                \big)
                \big(%
                    \sum_{\ell=k}^{\infty} \prod_{i=k}^{\ell} c_i
                \big)
            }{%
                \big(%
                    \prod_{i=1}^{k-1} c_i
                \big)%
                \big(%
                    1 + \sum_{\ell=k}^{\infty} \prod_{i=k}^{\ell} c_i
                \big)%
                }    
             =
             \frac{%
                \big(%
                    \sum_{\ell=k}^{\infty} \prod_{i=k}^{\ell} c_i
                \big)
            }{%
                \big(%
                    1 + \sum_{\ell=k}^{\infty} \prod_{i=k}^{\ell} c_i
                \big)%
                }
                \to 0    
        \end{align*}
        and \(\lim_{k\to\infty}\sum_{\ell=k}^{\infty} \prod_{i=k}^{\ell} c_i = 0\). Thus, \(\seq{z^k}\) converges to \(z^{\star}\) at least \(R\)-superlinearly, as claimed.
    \end{itemize}
    The assertions for~\FLEX{} follow directly, as setting 
    \(g=0\) reduces~\ProxFLEX{} and its underlying assumptions to those of~\FLEX{}. For~\IFLEX{}, the only distinction is that, for sufficiently large \(k\), \(\tau_k = 1\) is always accepted in \Cref{State:I-FLEX:LS} of~\IFLEX{}, due to~\eqref{def:eq:dir:superlin:F}. All other arguments remain unchanged. \qedhere
\end{proof}


\section{Numerical experiments}\label{sec:numerical}
In this section, we assess the performance of the proposed algorithms in \Cref{sec:algs} through a series of simulations on standard problems using both synthetic and real-world datasets. Code to replicate the experiments is made available online.\footnote{\url{https://github.com/manuupadhyaya/flex}} \Cref{tab:algorithms} contains a description of the algorithms used. 

\begin{table}[h]
    \centering
    \caption{Algorithms used in the numerical simulations (when applicable).}
    \label{tab:algorithms}
    \begin{tabular}{@{}l p{0.7\textwidth}@{}}
        \toprule
        \textbf{Method} & \textbf{Description} \\
        \midrule
        {\algnamefont EG}      
            & Extragradient method~\eqref{alg:EG} with \(\gamma = 0.9/L_{F}\). \\
        {\algnamefont EAG-C}   
            & Extra anchored gradient with constant step size \(\alpha = 1/(8L_{F})\)~\cite[Section 2.1]{yoon2021accelerated}. \\
        {\algnamefont GRAAL}   
            & Golden ratio algorithm with \(\phi = 2\) and \(\alpha = 0.999/L_{F}\)~\cite[Algorithm 2]{alacaoglu2023beyond},~\cite{malitsky2020goldenratioalgorithms}. \\
        {\algnamefont aGRAAL}  
            & Adaptive golden ratio algorithm with \(\phi = (2 + \sqrt{5})/2 \), \(\gamma = 1 / \phi + 1 /\phi^2\) and \(\alpha_0 = 0.1\)~\cite[Algorithm 1]{alacaoglu2023beyond},~\cite{malitsky2020goldenratioalgorithms}. \\
        {\algnamefont EG-AA}   
            & An extragradient-type method with type-II Anderson acceleration with memory \(m = 1\)~\cite[Algorithm 1]{qu2024extragradientanderson} using the parameter values described in~\cite[Section 4]{qu2024extragradientanderson}. \\
        {\algnamefont FISTA}   
            & Fast iterative shrinkage-thresholding algorithm with constant step size~\cite[Section 4]{Beck_2009}. \\
        \FLEX
            &\Cref{alg:FLEX} with \(\gamma = 0.9/L_F\), \(\beta = 0.3\), \(\sigma = 0.1\), \(\rho = 0.99\), and \(M=2\). \\
        \IFLEX
            &\Cref{alg:I-FLEX} with \(\beta = 0.01\) and \(\sigma = 0.1\). \\
        \ProxFLEX
            &\Cref{alg:Prox-FLEX} with \(\gamma = 0.9/L_F\), \(\beta = 0.3\), \(\sigma = 0.1\), \(\rho = 0.99\), and \(M=2\). \\
        \bottomrule
    \end{tabular}
\end{table}

In the numerical experiments for~\FLEX{},~\IFLEX{}, and~\ProxFLEX{}, we use directions \(\seq{d^{k}}\) based on 
quasi-Newton directions. 

\paragraph{Anderson acceleration.} 

The first set of quasi-Newton directions we use are the standard limited-memory type-I and type-II Anderson acceleration methods~\cite{Anderson1965, fang2009multisecant}. These directions are computed via~\eqref{eq:direction:Rgam}, i.e., \(d^k = - H_{k} R_\gamma(z^k)\), where \(H_{k}\) differs between the type-I and type-II variants. Both methods employ a memory parameter \(m \in \mathbb{N}\) and define \(m_k = \min\{m, k\}\). They also maintain two buffer matrices:
\[
Y_k = \begin{bmatrix}
y^{k - m_k}  & \cdots & y^{k-1}
\end{bmatrix}
\quad\text{and}\quad
S_k = \begin{bmatrix}
s^{k - m_k}  & \cdots & s^{k-1}
\end{bmatrix},
\]
where \(y^i = R_\gamma(z^{i+1}) - R_\gamma(z^i)\) and \(s^i = z^{i+1} - z^i\). For type-I Anderson acceleration (denoted {\algnamefont AA-I}), we have
\[
H_k = I + (S_k - Y_k)\,(S_k^{\top}Y_k)^{-1}\,S_k^{\top},
\]
whereas for type-II Anderson acceleration (denoted {\algnamefont AA-II}), we have
\[
H_k = I + (S_k - Y_k)\,(Y_k^{\top}Y_k)^{-1}\,Y_k^{\top}.
\]
Additional discussion can be found in~\cite{zhang2020globallyconvergenttype}.

\paragraph{J-symmetric directions.} 

We also incorporate directions derived from the J-symmetric quasi-Newton approach proposed in~\cite{asl2024jsymmetricquasi}, which is developed for unconstrained minimax problems. This method exploits the so-called J-symmetric structure of the Hessian in such problems, allowing a rank-2 update of the (inverse) Hessian estimate that naturally generalizes the classic Powell's symmetric Broyden method from standard minimization to minimax optimization. The formula for updating \(H_k\) in~\eqref{eq:direction:Rgam} can be found in~\cite[Proposition 2.2]{asl2024jsymmetricquasi}. We refer to this method as {\algnamefont J-sym}.

\subsection{Quadratic minimax problem}\label{sec:quadratic_minimax}
Consider the quadratic convex-concave minimax problem 
\begin{align}\label{eq:quadratic_minimax}
    \minimize_{x \in \reals^n}\; \maximize_{y \in \reals^n}\; \mathcal{L}\p{x,y}
\end{align}
for the saddle function \(\mathcal{L}:\reals^{n}\times\reals^{n}\to\reals\) such that 
\begin{align*}
    \mathcal{L}\p{x,y} = \tfrac{1}{2}\p{x - x^{\star}}^{\top} A \p{x - x^{\star}} + \p{x - x^{\star}}^{\top} C \p{y - y^{\star}} - \tfrac{1}{2}\p{y - y^{\star}}^{\top} B \p{y - y^{\star}}
\end{align*}
for each \(\p{x,y} \in \reals^{n}\times\reals^{n}\), where \(x^{\star},\,y^{\star} \in \reals^{n}\), \(A,\, B \in \sym_{+}^{n}\), and \(C\in\reals^{n\times n}\). A solution to the minimax problem~\eqref{eq:quadratic_minimax} can be obtained by solving an associated saddle point problem, which in turn can equivalently be written as~\eqref{eq:the_inclusion_problem} by letting \(\calH=\reals^{2n}\) with the inner product set to the dot product, \(g=0\), and \(F: \reals^{2n}\to\reals^{2n}\) as the monotone and \(L_F\)-Lipschitz continuous operator given by
\begin{align}\label{eq:quadratic_minimax:F}
    F(z) = 
    \begin{bmatrix}
        \nabla_{x} \mathcal{L}\p{x,y} \\
        -\nabla_{y} \mathcal{L}\p{x,y} 
    \end{bmatrix}
    =
    \begin{bmatrix}
        A \p{x - x^{\star}} + C\p{y - y^{\star}} \\
        B \p{y - y^{\star}} - C^{\top}\p{x - x^{\star}}
    \end{bmatrix}
\end{align}
for each \(z = (x,y)\in \reals^{n} \times \reals^{n}\), where\footnote{The matrix norm is taken as the spectral norm.}
\begin{align*}
    L_{F} = 
    \left\lVert 
    \begin{bmatrix}
        A &  C \\
        - C^{\top} & B
    \end{bmatrix}
    \right\rVert.
\end{align*}
We generate problem data as in \cite[Section 5.1]{asl2024jsymmetricquasi}, which is outlined below. The results of the numerical experiments are presented in \Cref{fig:qmp}. We see that~\FLEX{} and~\IFLEX{} do very well for small problems, while larger ones are more challenging. Nevertheless, the use of \texttt{AA-II} directions in our algorithms systematically performs at the top.

\begin{customalgorithm}{Datagen}[H]
    \begin{algorithmic}[1]
        \INPUT \(\QuadMinMaxParam \in \reals_{+}\) and \(n \in \mathbb{N}\)
        \OUTPUT \(x^{\star},\,y^{\star} \in \reals^{n}\), \(A,\, B \in \sym^{n}_{+}\) and \(C \in \reals^{n \times n}\)
        \State Let \(x^{\star},\,y^{\star} \in \reals^{n}\) such that \(x_{i}^{\star},\,y_{i}^{\star}\sim \mathcal{N}(0, 1)\) for each \(i\in\llbracket1,n\rrbracket\)
        \State Let \(S \in \mathbb{R}^{n \times n}\) such that \([S]_{i,j} \sim \mathcal{N}(0, 1/\sqrt{n})\) for each \(i,j\in\llbracket1,n\rrbracket\)
        \State \(S \leftarrow \p{S + S^\top}/2\)
        \State \(S \leftarrow S + \p{\abs{\lambda_{\min}(S)} + 1} I\)
        \State \(A \leftarrow \QuadMinMaxParam S\)
        \State Repeat steps 2-4 with a different random seed and let \(B \leftarrow \QuadMinMaxParam S\)
        \State Repeat step 2 with a different random seed and let \(C \leftarrow S\)
    \end{algorithmic}
\end{customalgorithm}

\begin{figure}[t]
    \centering
    
    \begin{subfigure}[b]{0.49\textwidth}
        \centering
        \caption{\(n=500\) and \(\QuadMinMaxParam = 0\)}
        \begin{tikzpicture}
            \begin{semilogyaxis}[
                width=\textwidth,
                height=0.8\textwidth, 
                ylabel={\(\norm{F\p{z^{k}}}\)},
                legend to name=commonLegend,
                legend style={
                    legend columns=6,        
                    cells={anchor=west},
                    font=\tiny,
                },
                grid=both,
                grid style={dashed, gray!30},
                ticklabel style={font=\tiny},
                label style={font=\tiny},
                xmin=0,
                xmax=200000,
                ymin=0.0005,
                ymax=10^(1)
            ]
            \addplot+[
                color=color1,
                mark=none,
                line width=1.5pt,
            ] table[
                x=numFevals,
                y=performance,
                col sep=comma,
            ] {plot_data/quadratic_minimax_problem_n500_alpha0/eg.tex};
            \addlegendentry{\algnamefont EG}

            \addplot+[
                color=color5,
                mark=none,
                line width=1.5pt,
            ] table[
                x=numFevals,
                y=performance,
                col sep=comma,
            ] {plot_data/quadratic_minimax_problem_n500_alpha0/eag-c.tex};
            \addlegendentry{\algnamefont EAG-C}

            \addplot+[
                color=color2,
                mark=none,
                line width=1.5pt,
            ] table[
                x=numFevals,
                y=performance,
                col sep=comma,
            ] {plot_data/quadratic_minimax_problem_n500_alpha0/graal.tex};
            \addlegendentry{\algnamefont GRAAL}

            \addplot+[
                color=color7,
                mark=none,
                line width=1.5pt,
            ] table[
                x=numFevals,
                y=performance,
                col sep=comma,
            ] {plot_data/quadratic_minimax_problem_n500_alpha0/agraal.tex};
            \addlegendentry{\algnamefont aGRAAL}

            \addplot+[
                color=color3,
                mark=none,
                line width=1.5pt,
                dashed,
            ] table[
                x=numFevals,
                y=performance,
                col sep=comma,
            ] {plot_data/quadratic_minimax_problem_n500_alpha0/eg-aa.tex};
            \addlegendentry{\algnamefont EG-AA}

            \addplot+[
                color=color8,
                mark=none,
                line width=1.5pt,
                dashed,
            ] table[
                x=numFevals,
                y=performance,
                col sep=comma,
            ] {plot_data/quadratic_minimax_problem_n500_alpha0/flex-aai.tex};
            \addlegendentry{\algnamefont \FLEX: AA-I}

            \addplot+[
                color=color10,
                mark=none,
                line width=1.5pt,
                dashed,
            ] table[
                x=numFevals,
                y=performance,
                col sep=comma,
            ] {plot_data/quadratic_minimax_problem_n500_alpha0/flex-aaii.tex};
            \addlegendentry{\algnamefont \FLEX: AA-II}

            \addplot+[
                color=color3,
                mark=none,
                line width=1.5pt,
                dashed,
            ] table[
                x=numFevals,
                y=performance,
                col sep=comma,
            ] {plot_data/quadratic_minimax_problem_n500_alpha0/flex-jsym.tex};
            \addlegendentry{\algnamefont \FLEX: J-sym}

            \end{semilogyaxis}
        \end{tikzpicture}
    \end{subfigure}
    \hfill
    \begin{subfigure}[b]{0.49\textwidth}
        \centering
        \caption{\(n=500\) and \(\QuadMinMaxParam = 0.0001\)}
        \begin{tikzpicture}
            \begin{semilogyaxis}[
                width=\textwidth,
                height=0.8\textwidth,
                legend to name=commonLegend,
                legend style={
                    legend columns=6,
                    cells={anchor=west},
                    font=\tiny,
                },
                grid=both,
                grid style={dashed, gray!30},
                ticklabel style={font=\tiny},
                label style={font=\tiny},
                xmin=0,
                xmax=2*10^5,
                ymin=10^(-7),
                ymax=10^2,
            ]
            \addplot+[
                color=color1,
                mark=none,
                line width=1.5pt,
            ] table[
                x=numFevals,
                y=performance,
                col sep=comma,
            ] {plot_data/quadratic_minimax_problem_n500_alpha00001/eg.tex};
            \addlegendentry{\algnamefont EG}

            \addplot+[
                color=color5,
                mark=none,
                line width=1.5pt,
            ] table[
                x=numFevals,
                y=performance,
                col sep=comma,
            ] {plot_data/quadratic_minimax_problem_n500_alpha00001/eag-c.tex};
            \addlegendentry{\algnamefont EAG-C}

            \addplot+[
                color=color2,
                mark=none,
                line width=1.5pt,
            ] table[
                x=numFevals,
                y=performance,
                col sep=comma,
            ] {plot_data/quadratic_minimax_problem_n500_alpha00001/graal.tex};
            \addlegendentry{\algnamefont GRAAL}

            \addplot+[
                color=color7,
                mark=none,
                line width=1.5pt,
            ] table[
                x=numFevals,
                y=performance,
                col sep=comma,
            ] {plot_data/quadratic_minimax_problem_n500_alpha00001/agraal.tex};
            \addlegendentry{\algnamefont aGRAAL}

            \addplot+[
                color=color3,
                mark=none,
                line width=1.5pt,
                dashed,
            ] table[
                x=numFevals,
                y=performance,
                col sep=comma,
            ] {plot_data/quadratic_minimax_problem_n500_alpha00001/eg-aa.tex};
            \addlegendentry{\algnamefont EG-AA}

            \addplot+[
                color=color8,
                mark=none,
                line width=1.5pt,
                dashed,
            ] table[
                x=numFevals,
                y=performance,
                col sep=comma,
            ] {plot_data/quadratic_minimax_problem_n500_alpha00001/flex-aai.tex};
            \addlegendentry{\algnamefont \FLEX: AA-I}

            \addplot+[
                color=color10,
                mark=none,
                line width=1.5pt,
                dashed,
            ] table[
                x=numFevals,
                y=performance,
                col sep=comma,
            ] {plot_data/quadratic_minimax_problem_n500_alpha00001/flex-aaii.tex};
            \addlegendentry{\algnamefont \FLEX: AA-II}

            \addplot+[
                color=color3,
                mark=none,
                line width=1.5pt,
                dashed,
            ] table[
                x=numFevals,
                y=performance,
                col sep=comma,
            ] {plot_data/quadratic_minimax_problem_n500_alpha00001/flex-jsym.tex};
            \addlegendentry{\algnamefont \FLEX: J-sym}

            \addplot+[
                color=color5,
                mark=none,
                line width=1.5pt,
                dashed,
            ] table[
                x=numFevals,
                y=performance,
                col sep=comma,
            ] {plot_data/quadratic_minimax_problem_n500_alpha00001/iflex-aai.tex};
            \addlegendentry{\algnamefont \IFLEX: AA-I}

            \addplot+[
                color=color2,
                mark=none,
                line width=1.5pt,
                dashed,
            ] table[
                x=numFevals,
                y=performance,
                col sep=comma,
            ] {plot_data/quadratic_minimax_problem_n500_alpha00001/iflex-aaii.tex};
            \addlegendentry{\algnamefont \IFLEX: AA-II}

            \addplot+[
                color=color7,
                mark=none,
                line width=1.5pt,
                dashed,
            ] table[
                x=numFevals,
                y=performance,
                col sep=comma,
            ] {plot_data/quadratic_minimax_problem_n500_alpha00001/iflex-jsym.tex};
            \addlegendentry{\algnamefont \IFLEX: J-sym}
            
            \end{semilogyaxis}
        \end{tikzpicture}
    \end{subfigure}

    \begin{subfigure}[b]{0.49\textwidth}
        \centering
        \caption{\(n=20\) and \(\QuadMinMaxParam = 0\)}
        \begin{tikzpicture}
            \begin{semilogyaxis}[
                width=\textwidth,
                height=0.8\textwidth, 
                xlabel={Number of \(F\) evaluations},
                ylabel={\(\norm{F\p{z^{k}}}\)},
                legend to name=commonLegend,
                legend style={
                    legend columns=6,        
                    cells={anchor=west},
                    font=\tiny,
                },
                grid=both,
                grid style={dashed, gray!30},
                ticklabel style={font=\tiny},
                label style={font=\tiny},
                xmin=0,
                xmax=50000,
                ymin=10^(-8),
                ymax=10,
            ]
            \addplot+[
                color=color1,
                mark=none,
                line width=1.5pt,
            ] table[
                x=numFevals,
                y=performance,
                col sep=comma,
            ] {plot_data/quadratic_minimax_problem_n20_alpha0/eg.tex};
            \addlegendentry{\algnamefont EG}

            \addplot+[
                color=color5,
                mark=none,
                line width=1.5pt,
            ] table[
                x=numFevals,
                y=performance,
                col sep=comma,
            ] {plot_data/quadratic_minimax_problem_n20_alpha0/eag-c.tex};
            \addlegendentry{\algnamefont EAG-C}

            \addplot+[
                color=color2,
                mark=none,
                line width=1.5pt,
            ] table[
                x=numFevals,
                y=performance,
                col sep=comma,
            ] {plot_data/quadratic_minimax_problem_n20_alpha0/graal.tex};
            \addlegendentry{\algnamefont GRAAL}

            \addplot+[
                color=color7,
                mark=none,
                line width=1.5pt,
            ] table[
                x=numFevals,
                y=performance,
                col sep=comma,
            ] {plot_data/quadratic_minimax_problem_n20_alpha0/agraal.tex};
            \addlegendentry{\algnamefont aGRAAL}

            \addplot+[
                color=color3,
                mark=none,
                line width=1.5pt,
                dashed,
            ] table[
                x=numFevals,
                y=performance,
                col sep=comma,
            ] {plot_data/quadratic_minimax_problem_n20_alpha0/eg-aa.tex};
            \addlegendentry{\algnamefont EG-AA}

            \addplot+[
                color=color8,
                mark=none,
                line width=1.5pt,
                dashed,
            ] table[
                x=numFevals,
                y=performance,
                col sep=comma,
            ] {plot_data/quadratic_minimax_problem_n20_alpha0/flex-aai.tex};
            \addlegendentry{\algnamefont \FLEX: AA-I}

            \addplot+[
                color=color10,
                mark=none,
                line width=1.5pt,
                dashed,
            ] table[
                x=numFevals,
                y=performance,
                col sep=comma,
            ] {plot_data/quadratic_minimax_problem_n20_alpha0/flex-aaii.tex};
            \addlegendentry{\algnamefont \FLEX: AA-II}

            \addplot+[
                color=color3,
                mark=none,
                line width=1.5pt,
                dashed,
            ] table[
                x=numFevals,
                y=performance,
                col sep=comma,
            ] {plot_data/quadratic_minimax_problem_n20_alpha0/flex-jsym.tex};
            \addlegendentry{\algnamefont \FLEX: J-sym}

            \end{semilogyaxis}
        \end{tikzpicture}
    \end{subfigure}
    \hfill
    \begin{subfigure}[b]{0.49\textwidth}
        \centering
        \caption{\(n=20\) and \(\QuadMinMaxParam = 0.0001\)}
        \begin{tikzpicture}
            \begin{semilogyaxis}[
                width=\textwidth,
                height=0.8\textwidth,
                xlabel={Number of \(F\) evaluations},
                legend to name=commonLegend,
                legend style={
                    legend columns=6,
                    cells={anchor=west},
                    font=\tiny,
                },
                grid=both,
                grid style={dashed, gray!30},
                ticklabel style={font=\tiny},
                label style={font=\tiny},
                xmin=0,
                xmax=50000,
                ymin=10^(-8),
                ymax=10,
            ]
            \addplot+[
                color=color1,
                mark=none,
                line width=1.5pt,
            ] table[
                x=numFevals,
                y=performance,
                col sep=comma,
            ] {plot_data/quadratic_minimax_problem_n20_alpha00001/eg.tex};
            \addlegendentry{\algnamefont EG}

            \addplot+[
                color=color5,
                mark=none,
                line width=1.5pt,
            ] table[
                x=numFevals,
                y=performance,
                col sep=comma,
            ] {plot_data/quadratic_minimax_problem_n20_alpha00001/eag-c.tex};
            \addlegendentry{\algnamefont EAG-C}

            \addplot+[
                color=color2,
                mark=none,
                line width=1.5pt,
            ] table[
                x=numFevals,
                y=performance,
                col sep=comma,
            ] {plot_data/quadratic_minimax_problem_n20_alpha00001/graal.tex};
            \addlegendentry{\algnamefont GRAAL}

            \addplot+[
                color=color7,
                mark=none,
                line width=1.5pt,
            ] table[
                x=numFevals,
                y=performance,
                col sep=comma,
            ] {plot_data/quadratic_minimax_problem_n20_alpha00001/agraal.tex};
            \addlegendentry{\algnamefont aGRAAL}

            \addplot+[
                color=color3,
                mark=none,
                line width=1.5pt,
                dashed,
            ] table[
                x=numFevals,
                y=performance,
                col sep=comma,
            ] {plot_data/quadratic_minimax_problem_n20_alpha00001/eg-aa.tex};
            \addlegendentry{\algnamefont EG-AA}

            \addplot+[
                color=color8,
                mark=none,
                line width=1.5pt,
                dashed,
            ] table[
                x=numFevals,
                y=performance,
                col sep=comma,
            ] {plot_data/quadratic_minimax_problem_n20_alpha00001/flex-aai.tex};
            \addlegendentry{\algnamefont \FLEX: AA-I}

            \addplot+[
                color=color10,
                mark=none,
                line width=1.5pt,
                dashed,
            ] table[
                x=numFevals,
                y=performance,
                col sep=comma,
            ] {plot_data/quadratic_minimax_problem_n20_alpha00001/flex-aaii.tex};
            \addlegendentry{\algnamefont \FLEX: AA-II}

            \addplot+[
                color=color3,
                mark=none,
                line width=1.5pt,
                dashed,
            ] table[
                x=numFevals,
                y=performance,
                col sep=comma,
            ] {plot_data/quadratic_minimax_problem_n20_alpha00001/flex-jsym.tex};
            \addlegendentry{\algnamefont \FLEX: J-sym}

            \addplot+[
                color=color5,
                mark=none,
                line width=1.5pt,
                dashed,
            ] table[
                x=numFevals,
                y=performance,
                col sep=comma,
            ] {plot_data/quadratic_minimax_problem_n20_alpha00001/iflex-aai.tex};
            \addlegendentry{\algnamefont \IFLEX: AA-I}

            \addplot+[
                color=color2,
                mark=none,
                line width=1.5pt,
                dashed,
            ] table[
                x=numFevals,
                y=performance,
                col sep=comma,
            ] {plot_data/quadratic_minimax_problem_n20_alpha00001/iflex-aaii.tex};
            \addlegendentry{\algnamefont \IFLEX: AA-II}

            \addplot+[
                color=color7,
                mark=none,
                line width=1.5pt,
                dashed,
            ] table[
                x=numFevals,
                y=performance,
                col sep=comma,
            ] {plot_data/quadratic_minimax_problem_n20_alpha00001/iflex-jsym.tex};
            \addlegendentry{\algnamefont \IFLEX: J-sym}
            
            \end{semilogyaxis}
        \end{tikzpicture}
    \end{subfigure}

    \vspace{1em}
    \centering
    \pgfplotslegendfromname{commonLegend}
    \vspace{0.5em}

    \caption{Convergence of algorithms on the quadratic minimax problem~\eqref{eq:quadratic_minimax}. Both {\algnamefont AA-I} and {\algnamefont AA-II} use memory parameter \(m=20\). \new{When \(\QuadMinMaxParam = 0\), the operator \(F\) in \eqref{eq:quadratic_minimax:F} is monotone; for \(\QuadMinMaxParam > 0\), it becomes strongly monotone.}}

    \label{fig:qmp}

\end{figure}
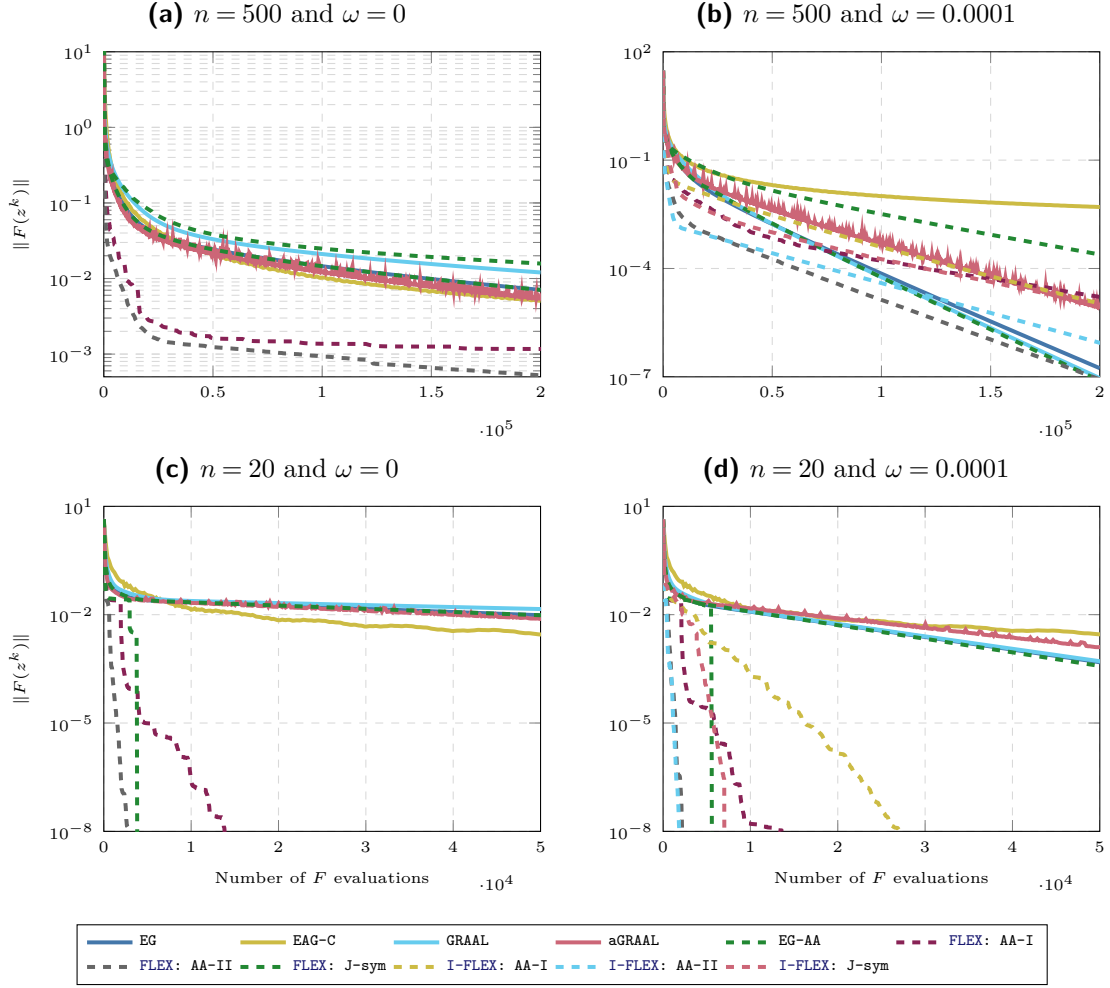

\subsection{Bilinear zero-sum game with simplex constraints}\label{sec:bilinear}
Consider the bilinear zero-sum game with simplex constraints given by
\begin{align}\label{eq:bilinear}
    \minimize_{x \in \Delta^n}\; \maximize_{y \in \Delta^n}\; x^{\top} A y
\end{align}
where \(A\in\reals^{n\times n}\) is the payoff matrix and \(\Delta^{n}=\set{w\in\reals_{+}^n\mid w^{\top}\mathbf{1} = 1 }\) is the probability simplex in \(\reals^n\), which is equivalent to finding a saddle point \(\p{x^{\star},y^{\star}}\in\Delta^{n}\times\Delta^{n}\) (which is guaranteed to exist), i.e.,
\begin{align*}
     \p{x^{\star}}^{\top} A y \leq \p{x^{\star}}^{\top} A y^{\star} \leq x^{\top} A y^{\star} 
\end{align*}
for each \(\p{x,y} \in \Delta^{n}\times\Delta^{n}\). This, in turn, is equivalent to solving~\eqref{eq:the_inclusion_problem} by letting \(\calH=\reals^{2n}\) with the inner product set to the dot product, \(g=\delta_{\Delta^{n}\times\Delta^{n}}\), and \(F: \reals^{2n}\to\reals^{2n}\) as the monotone and \(L_F\)-Lipschitz continuous operator given by
\begin{align*}
    F(z) = 
    \begin{bmatrix}
        Ay \\
        - A^{\top}x
    \end{bmatrix}
\end{align*}
for each \(z = (x,y)\in \reals^{n} \times \reals^{n}\), where
\begin{align*}
    L_{F} = 
    \left\lVert 
    \begin{bmatrix}
        0 &  A \\
        - A^{\top} & 0
    \end{bmatrix}
    \right\rVert.
\end{align*}

We generate \(A = S - S^{\top}\) for a random matrix \(S\in\reals^{n\times n}\) such that \([S]_{i,j} \sim \mathcal{N}(0, 1)\) for each \(i,j \in \llbracket 1, n \rrbracket\), resulting in a skew-symmetric matrix \(A\). The results of the numerical experiments are presented in \Cref{fig:bmg}. We see that using \texttt{AA-II} directions in~\ProxFLEX{} gives good performance.

\begin{figure}[t]
    \centering
    \begin{subfigure}[c]{0.49\textwidth}
        \centering
        \caption{\(n=250\)}
        \label{fig:bmg:250}
        \begin{tikzpicture}
            \begin{semilogyaxis}[
                width=\textwidth,
                height=0.8\textwidth,
                xlabel={Number of operator evaluations},
                ylabel={\(\norm{r^{k}}\)},
                y label style={rotate=-90},
                legend to name=commonLegend,
                legend style={
                    legend columns=6,
                    cells={anchor=west},
                    font=\tiny,
                },
                grid=both,
                grid style={dashed, gray!30},
                ticklabel style={font=\tiny},
                label style={font=\tiny},
                xmin=0,
                xmax=4*10^5,
                ymin=10^(-10),
            ]
            \addplot+[
                color=color1,
                mark=none,
                line width=1.5pt,
            ] table[
                x=numoperatorevals,
                y=performance,
                col sep=comma,
            ] {plot_data/bilinear_matrix_game_n250/eg.tex};
            \addlegendentry{\algnamefont EG}
    
            \addplot+[
                color=color2,
                mark=none,
                line width=1.5pt,
            ] table[
                x=numoperatorevals,
                y=performance,
                col sep=comma,
            ] {plot_data/bilinear_matrix_game_n250/graal.tex};
            \addlegendentry{\algnamefont GRAAL}
            
            \addplot+[
                color=color7,
                mark=none,
                line width=1.5pt,
            ] table[
                x=numoperatorevals,
                y=performance,
                col sep=comma,
            ] {plot_data/bilinear_matrix_game_n250/agraal.tex};
            \addlegendentry{\algnamefont aGRAAL}
    
            \addplot+[
                color=color3,
                mark=none,
                line width=1.5pt,
                dashed,%
            ] table[
                x=numoperatorevals,
                y=performance,
                col sep=comma,
            ] {plot_data/bilinear_matrix_game_n250/eg-aa.tex};
            \addlegendentry{\algnamefont EG-AA}
            
            \addplot+[
                color=color8,
                mark=none,
                line width=1.5pt,
                dashed,
            ] table[
                x=numoperatorevals,
                y=performance,
                col sep=comma,
            ] {plot_data/bilinear_matrix_game_n250/proxflex-aai.tex};
            \addlegendentry{\algnamefont \ProxFLEX: AA-I}
    
            \addplot+[
                color=color10,
                mark=none,
                line width=1.5pt,
                dashed,
            ] table[
                x=numoperatorevals,
                y=performance,
                col sep=comma,
            ] {plot_data/bilinear_matrix_game_n250/proxflex-aaii.tex};
            \addlegendentry{\algnamefont \ProxFLEX: AA-II}
            
            \end{semilogyaxis}
        \end{tikzpicture}
    \end{subfigure}
    \hfill
    \begin{subfigure}[c]{0.49\textwidth}
        \centering
        \caption{\(n=500\)}
        \label{fig:bmg:500}
        \begin{tikzpicture}
            \begin{semilogyaxis}[
                width=\textwidth,
                height=0.8\textwidth,
                xlabel={Number of operator evaluations},
                y label style={rotate=-90},
                legend to name=commonLegend,
                legend style={
                    legend columns=6,
                    cells={anchor=west},
                    font=\tiny,
                },
                grid=both,
                grid style={dashed, gray!30},
                ticklabel style={font=\tiny},
                label style={font=\tiny},
                xmin=0,
                xmax=4*10^5,
                ymin=10^(-8),
            ]
            \addplot+[
                color=color1,
                mark=none,
                line width=1.5pt,
            ] table[
                x=numoperatorevals,
                y=performance,
                col sep=comma,
            ] {plot_data/bilinear_matrix_game_n500/eg.tex};
            \addlegendentry{\algnamefont EG}
    
            \addplot+[
                color=color2,
                mark=none,
                line width=1.5pt,
            ] table[
                x=numoperatorevals,
                y=performance,
                col sep=comma,
            ] {plot_data/bilinear_matrix_game_n500/graal.tex};
            \addlegendentry{\algnamefont GRAAL}
            
            \addplot+[
                color=color7,
                mark=none,
                line width=1.5pt,
            ] table[
                x=numoperatorevals,
                y=performance,
                col sep=comma,
            ] {plot_data/bilinear_matrix_game_n500/agraal.tex};
            \addlegendentry{\algnamefont aGRAAL}
    
            \addplot+[
                color=color3,
                mark=none,
                line width=1.5pt,
                dashed,%
            ] table[
                x=numoperatorevals,
                y=performance,
                col sep=comma,
            ] {plot_data/bilinear_matrix_game_n500/eg-aa.tex};
            \addlegendentry{\algnamefont EG-AA}
            
            \addplot+[
                color=color8,
                mark=none,
                line width=1.5pt,
                dashed,
            ] table[
                x=numoperatorevals,
                y=performance,
                col sep=comma,
            ] {plot_data/bilinear_matrix_game_n500/proxflex-aai.tex};
            \addlegendentry{\algnamefont \ProxFLEX: AA-I}
    
            \addplot+[
                color=color10,
                mark=none,
                line width=1.5pt,
                dashed,
            ] table[
                x=numoperatorevals,
                y=performance,
                col sep=comma,
            ] {plot_data/bilinear_matrix_game_n500/proxflex-aaii.tex};
            \addlegendentry{\algnamefont \ProxFLEX: AA-II}
            
            \end{semilogyaxis}
        \end{tikzpicture}
    \end{subfigure}

    \vspace{1em}
    \centering
    \pgfplotslegendfromname{commonLegend}
    \vspace{0.5em}
    
    \caption{Convergence of algorithms on the bilinear zero-sum game with simplex constraints~\eqref{eq:bilinear} where \(r^{k} = R_{1/2L_F}(z^k)\) and \(R\) is the residual mapping in~\eqref{eq:direction:Rgam}. Both {\algnamefont AA-I} and {\algnamefont AA-II} use memory parameter \(m=10\) for \Cref{fig:bmg:250} and \(m=20\) for \Cref{fig:bmg:500}. The number of operator evaluations equals the number of \(F\) and \(\prox_{\gamma g}\) evaluations.}
    \label{fig:bmg}
\end{figure}
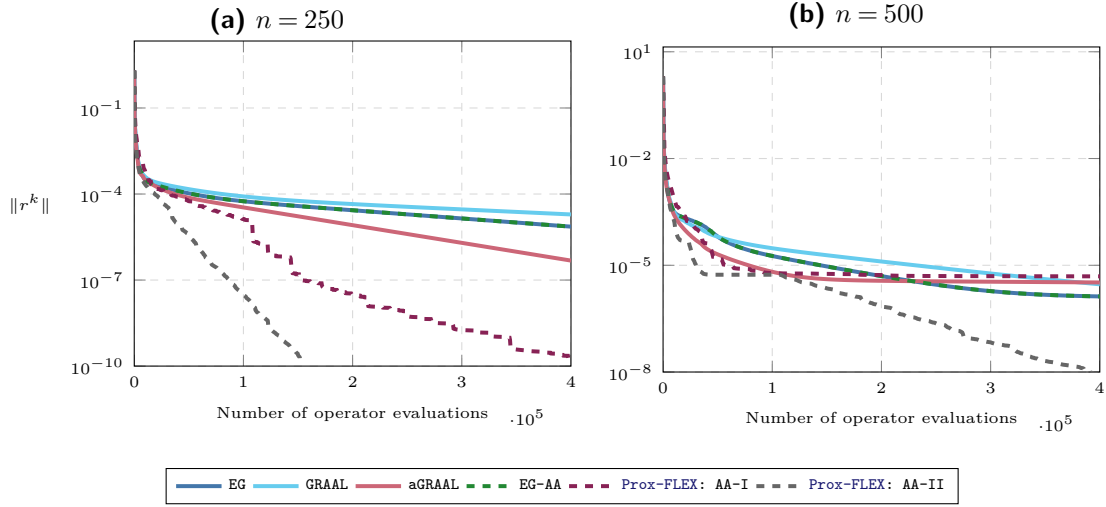

\subsection{Cournot--Nash equilibrium problem}\label{sec:equilibrium}
Consider a noncooperative game with \(n\in\mathbb{N}\) players, in which each player \(i \in \llbracket1,n\rrbracket\) has to pick a strategy \(z_{i}\) that lies in \(\mathcal{Z}_{i}\), a subset of a real Hilbert space \(\calH_i\), and has an associated loss function \(\varphi_{i}:\calH \to \reals\), where \(\calH=\prod_{j=1}^{n}\calH_j\).  In this case, a pure strategy Nash equilibrium is a strategy profile \(z = \p{z_{1},\ldots, z_{n}}\in\calH\) that solves the problem
\begin{align}\label{eq:Nash}
    \text{find}\  z\in\calH\ \text{ such that }\ z_{i} \in \Argmin_{x \in \mathcal{Z}_{i}} \varphi_{i}\p{x;z_{\smallsetminus i}} \ \text{ for each }\ i \in \llbracket1,n\rrbracket,
\end{align}
where we have used the notation \(\p{x;z_{\smallsetminus i}} = \p{z_{1},\ldots, z_{i-1}, x, z_{i+1},\ldots,z_{n}}\) for each \(x\in\calH_{i}\) and \(i \in \llbracket1,n\rrbracket\). In particular, assume that, for each \(i \in \llbracket1,n\rrbracket\), the function \(\varphi_{i}(\,\cdot\,;z_{\smallsetminus i}):\calH_{i} \to \reals\) is convex for each \(z\in\calH\), the gradient \(\nabla_{z_{i}}\varphi_{i}:\calH \to \calH_{i}\) exists and is Lipschitz continuous, and the set \(\mathcal{Z}_{i}\subseteq\calH_{i}\) is nonempty, closed and convex. Then~\eqref{eq:Nash} can equivalently be written as~\eqref{eq:the_inclusion_problem} by letting \(F:\calH \to \calH : z \mapsto \p{\nabla_{z_{i}}\varphi_{i}\p{z}}_{i=1}^{n}\) and \(g=\delta_{\mathcal{Z}}\), where \(\mathcal{Z} = \prod_{i=1}^{n}\mathcal{Z}_{i}\), and it is straightforward to verify that \Cref{ass:main} holds.

Let us further specialize the model to the Cournot--Nash equilibrium problem for oligopolistic markets with concave-quadratic cost functions and a differentiated commodity, as presented in~\cite{Bigi_Passacantando_2017}. Such models are useful for policymakers and economists in analyzing market outcomes, assessing welfare effects, and evaluating the impact of various market interventions~\cite{Bonanno_1990,fudenberg1986dynamic,von_Mouche_Quartieri_2016,Muu2008,Okuguchi_Szidarovszky_1999,Vives_1989}. In particular, in the model of~\cite{Bigi_Passacantando_2017}, each producer \(i\in\llbracket1,n\rrbracket\) chooses to produce and supply a quantity \(z_{i} \in [0,T_{i}]\) of a differentiated commodity at a cost \(c_{i}:\reals\to\reals\) such that 
\begin{align*}
    c_{i}\p{z_{i}} = a_{i} z_{i}^{2} + b_{i} z_{i},
\end{align*}
for each \(z_{i}\in \reals\), where \(T_{i}>0\) denotes the maximum capacity of production, and \(a_{i}<0\) and \(b_{i}>0\) are numbers such that \(b_{i} \geq -2T_{i}a_{i}\), ensuring that \(c_{i}\) is increasing on \([0,T_{i}]\). Moreover, each producer \(i\in\llbracket1,n\rrbracket\) has a price per produced unit of the differentiated commodity\footnote{Also known as the inverse demand function.}, denoted by \(p_{i}:\reals^{n}\to\reals\), that also depends on the other producers' supply, and is modeled by 
\begin{align*}
    p_{i}\p{z} = m_{i} - d_{i} \sum_{j=1}^{n}z_{j}
\end{align*}
for each \(z = \p{z_{1},\ldots,z_{n}}\in\reals^{n}\), for some \(m_{i} > b_{i}\) and \(d_{i}> - a_{i}\), where the last two assumptions guarantee a positive profit in a monopolistic setting, i.e., when \(n=1\). Thus, given that the goal of each producer is to maximize profit, or equivalently minimize losses, an equilibrium state where no producer has any incentive to deviate unidirectionally from its production plan can be modeled by~\eqref{eq:Nash}, with  \(\mathcal{Z}_{i} = [0, T_{i}]\), \(\calH_{i} = \reals\), and
\begin{align*}
    \varphi_{i}\p{z} &= c_{i}\p{z_{i}} - z_{i} p\p{z} 
\end{align*}
for each \(z = \p{z_{1},\ldots,z_{n}}\in\reals^{n}\) and \(i\in\llbracket1,n\rrbracket\), which fulfill the assumptions in the first paragraph of this section. We identify \(F\) as 
\begin{align*}
    F\p{z} = 
    \underbracket{
    \begin{bmatrix}
        2\p{a_{1} + d_{1}}  & d_{1}             & d_{1}     & \cdots & d_{1} \\ 
        d_{2}               & 2\p{a_{2} + d_{2}}& d_{2} & \cdots & d_{2} \\
        d_{3} & d_{3} & 2\p{a_{3} + d_{3}} & \cdots & d_{3} \\
        \vdots & \vdots & \vdots & \ddots & \vdots \\
        d_{n} & d_{n} & d_{n} & \cdots & 2\p{a_{n} + d_{n}}
    \end{bmatrix} 
    }_{= A}
    z 
    +
    \begin{bmatrix}
        b_{1} - m_{1} \\
        \vdots \\
        b_{n} - m_{n} 
    \end{bmatrix} 
\end{align*}
with Lipschitz constant \(L_{F} = \norm{A}\). We also note that~\cite{Rosen_1965} provides the existence of a solution in this case. We generate the data similar to the approach in~\cite[Section 4.1]{Bigi_Passacantando_2017}, as outlined below. The results of the numerical experiments are presented in \Cref{fig:nash}. Although \(n=100\) in \Cref{fig:nash:100} is not representative of a real oligopolistic market, we include this larger problem size to evaluate the performance and scalability of the algorithms. We observe that~\ProxFLEX{} has a superlinear drop-off in both cases and that \texttt{EG-AA} and \texttt{aGRAAL} scale well for this particular problem.  

\begin{customalgorithm}{Datagen}[H]
    \begin{algorithmic}[1] 
        \INPUT \(n \in \mathbb{N}\) 
        \OUTPUT \(\p{\p{T_{i}, a_{i}, b_{i}, m_{i}, d_{i}}}_{i=1}^{n}\) 
        \Repeat 
            \State For each \(i \in \llbracket1,n\rrbracket\), sample \(m_{i}\) uniformly from \([150, 250]\) 
            \State For each \(i \in \llbracket1,n\rrbracket\), sample \(b_{i}\) uniformly from \([30, 50]\) 
            \State For each \(i \in \llbracket1,n\rrbracket\), sample \(T_{i}\) uniformly from \([3, 7]\) 
            \State For each \(i \in \llbracket1,n\rrbracket\), sample \(d_{i}\) uniformly from \([5, 20]\) 
            \State Sort \(\p{d_{i}}_{i=1}^{n}\) in increasing order 
            \State For each \(i \in \llbracket1,n\rrbracket\), sample \(u_{i}\) uniformly from \([-10, -5]\) 
            \State For each \(i \in \llbracket1,n\rrbracket\), compute \(a_{i} = d_{i} / u_{i}\) 
            \State Sort \(\p{a_{i}}_{i=1}^{n}\) in decreasing order 
            \State \(\text{valid} \leftarrow \texttt{True}\) 
            \For{\(i \in \llbracket1,n\rrbracket\)} 
                \If{\(b_{i} < -2 a_{i} T_{i}\) \textbf{or} \(m_{i} \leq b_{i}\) \textbf{or} \(d_{i} \leq - a_{i}\) } 
                    \State \(\text{valid} \leftarrow \texttt{False}\) 
                    \State \textbf{break} 
                \EndIf 
            \EndFor 
        \Until{\(\text{valid}\) is \texttt{True}} 
    \end{algorithmic}
\end{customalgorithm}

\begin{figure}[t]
    \centering
    
    \begin{subfigure}[c]{0.49\textwidth}
        \centering
        \caption{\(n=10\) }
        \label{fig:nash:10}
        \begin{tikzpicture}
            \begin{semilogyaxis}[
                width=\textwidth,
                height=0.8\textwidth, 
                xlabel={Number of operator evaluations},
                ylabel={\(\norm{r^{k}}\)},
                y label style={rotate=-90},
                legend to name=commonLegend,
                legend style={
                    legend columns=6,
                    cells={anchor=west},
                    font=\tiny,
                },
                grid=both,
                grid style={dashed, gray!30},
                ticklabel style={font=\tiny},
                label style={font=\tiny},
                xmin=0,
                xmax=1400,
                ymin=10^(-10),
            ]
            \addplot+[
                color=color1,
                mark=none,
                line width=1.5pt,
            ] table[
                x=numoperatorevals,
                y=performance,
                col sep=comma,
            ] {plot_data/cournot-nash_equilibrium_problem_n10/eg.tex};
            \addlegendentry{\algnamefont EG}

            \addplot+[
                color=color2,
                mark=none,
                line width=1.5pt,
            ] table[
                x=numoperatorevals,
                y=performance,
                col sep=comma,
            ] {plot_data/cournot-nash_equilibrium_problem_n10/graal.tex};
            \addlegendentry{\algnamefont GRAAL}
            
            \addplot+[
                color=color7,
                mark=none,
                line width=1.5pt,
            ] table[
                x=numoperatorevals,
                y=performance,
                col sep=comma,
            ] {plot_data/cournot-nash_equilibrium_problem_n10/agraal.tex};
            \addlegendentry{\algnamefont aGRAAL}

            \addplot+[
                color=color3,
                mark=none,
                line width=1.5pt,
            ] table[
                x=numoperatorevals,
                y=performance,
                col sep=comma,
            ] {plot_data/cournot-nash_equilibrium_problem_n10/eg-aa.tex};
            \addlegendentry{\algnamefont EG-AA}
            
            \addplot+[
                color=color8,
                mark=none,
                line width=1.5pt,
                dashed,
            ] table[
                x=numoperatorevals,
                y=performance,
                col sep=comma,
            ] {plot_data/cournot-nash_equilibrium_problem_n10/proxflex-aai.tex};
            \addlegendentry{\algnamefont \ProxFLEX: AA-I}

            \addplot+[
                color=color10,
                mark=none,
                line width=1.5pt,
                dashed,
            ] table[
                x=numoperatorevals,
                y=performance,
                col sep=comma,
            ] {plot_data/cournot-nash_equilibrium_problem_n10/proxflex-aaii.tex};
            \addlegendentry{\algnamefont \ProxFLEX: AA-II}
            
            \end{semilogyaxis}
        \end{tikzpicture}
    \end{subfigure}
    \hfill
    \begin{subfigure}[c]{0.49\textwidth}
        \centering
        \caption{\(n=100\)}
        \label{fig:nash:100}
        \begin{tikzpicture}
            \begin{semilogyaxis}[
                width=\textwidth,
                height=0.8\textwidth,
                xlabel={Number of operator evaluations},
                legend to name=commonLegend,
                legend style={
                    legend columns=6,
                    cells={anchor=west},
                    font=\tiny,
                },
                grid=both,
                grid style={dashed, gray!30},
                ticklabel style={font=\tiny},
                label style={font=\tiny},
                xmin=0,
                xmax=4000,
                ymin=10^(-10),
            ]
            \addplot+[
                color=color1,
                mark=none,
                line width=1.5pt,
            ] table[
                x=numoperatorevals,
                y=performance,
                col sep=comma,
            ] {plot_data/cournot-nash_equilibrium_problem_n100/eg.tex};
            \addlegendentry{\algnamefont EG}

            \addplot+[
                color=color2,
                mark=none,
                line width=1.5pt,
            ] table[
                x=numoperatorevals,
                y=performance,
                col sep=comma,
            ] {plot_data/cournot-nash_equilibrium_problem_n100/graal.tex};
            \addlegendentry{\algnamefont GRAAL}
            
            \addplot+[
                color=color7,
                mark=none,
                line width=1.5pt,
            ] table[
                x=numoperatorevals,
                y=performance,
                col sep=comma,
            ] {plot_data/cournot-nash_equilibrium_problem_n100/agraal.tex};
            \addlegendentry{\algnamefont aGRAAL}

            \addplot+[
                color=color3,
                mark=none,
                line width=1.5pt,
            ] table[
                x=numoperatorevals,
                y=performance,
                col sep=comma,
            ] {plot_data/cournot-nash_equilibrium_problem_n100/eg-aa.tex};
            \addlegendentry{\algnamefont EG-AA}
            
            \addplot+[
                color=color8,
                mark=none,
                line width=1.5pt,
                dashed,
            ] table[
                x=numoperatorevals,
                y=performance,
                col sep=comma,
            ] {plot_data/cournot-nash_equilibrium_problem_n100/proxflex-aai.tex};
            \addlegendentry{\algnamefont \ProxFLEX: AA-I}

            \addplot+[
                color=color10,
                mark=none,
                line width=1.5pt,
                dashed,
            ] table[
                x=numoperatorevals,
                y=performance,
                col sep=comma,
            ] {plot_data/cournot-nash_equilibrium_problem_n100/proxflex-aaii.tex};
            \addlegendentry{\algnamefont \ProxFLEX: AA-II}

            \end{semilogyaxis}
        \end{tikzpicture}
    \end{subfigure}

    \vspace{1em}
    \centering
    \pgfplotslegendfromname{commonLegend}
    \vspace{0.5em}
    
    \caption{Convergence of algorithms on the Cournot–Nash equilibrium problem where \(r^{k} = R_{1/2L_F}(z^k)\) and \(R\) is the residual mapping in~\eqref{eq:direction:Rgam}. Both {\algnamefont AA-I} and {\algnamefont AA-II} use memory parameter \(m=3\). The number of operator evaluations equals the number of \(F\) and \(\prox_{\gamma g}\) evaluations.}
    \label{fig:nash}
\end{figure}
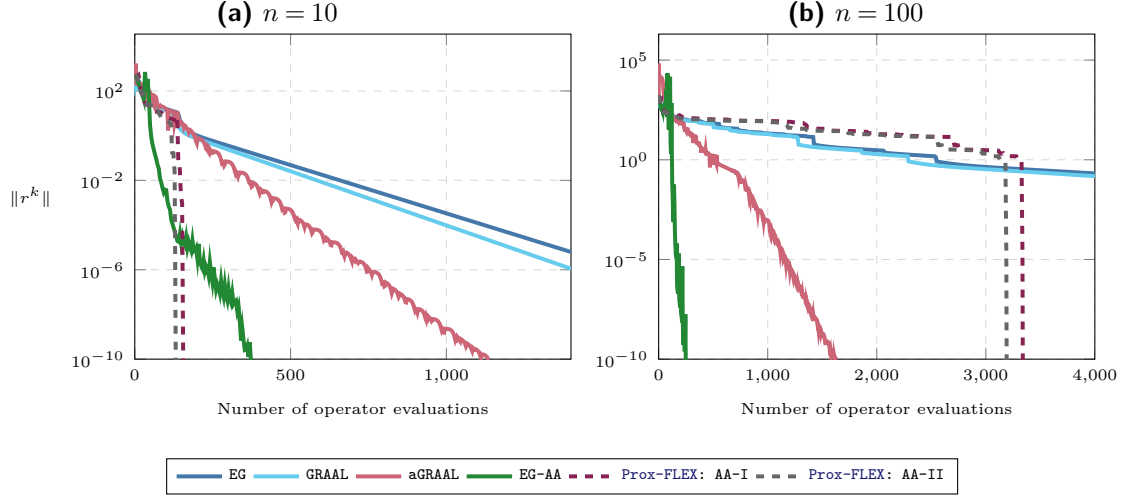

\subsection{Sparse logistic regression}\label{sec:logistic}
Consider the sparse logistic regression problem
\begin{align}\label{eq:logistic_01_label}
    \minimize_{x\in\mathbb{R}^{n}}\; \sum_{i=1}^m \log\left(1 + \exp\left({-b_i a_i^{\top}x}\right)\right) + \lambda \norm{x}_{1}
\end{align}
where \((a_i,b_i) \in \mathbb{R}^n\times \{\pm 1\}\) for each \(i = 1,\dots,m\). The minimization problem~\eqref{eq:logistic_01_label} can equivalently be written as the inclusion problem \eqref{eq:the_inclusion_problem} by letting \(\calH=\reals^{n}\) with the inner product set to the dot product, \(F:\mathbb{R}^{n}\to\mathbb{R}^{n}\) such that \(F(x)=K^{\top} \sigma(Kx)\) for each \(x\in\reals^{n}\)
where
\begin{align*}
    \sigma:\mathbb{R}^{m}\to\mathbb{R}^{m}:
    \begin{bmatrix} u_1 \\ \vdots \\ u_m \end{bmatrix} 
    \mapsto 
    \begin{bmatrix} \frac{\exp(u_1)}{1+\exp(u_1)} \\ \vdots \\ \frac{\exp(u_m)}{1+\exp(u_m)} \end{bmatrix},  \quad K = 
    \begin{bmatrix}
        - b_1 a_1^{\top} \\
        \vdots  \\
        - b_m a_m^{\top} 
    \end{bmatrix}
    \in \mathbb{R}^{m\times n},
\end{align*}
and \(g = \lambda \norm{\cdot}_{1} \). Moreover, note that \Cref{ass:main} holds with \(L_F=(1/4)\norm{K}^2\). The results of the numerical experiments are presented in \Cref{fig:logistic}. Although not designed specifically for minimization problems, we observe that~\ProxFLEX{} with \texttt{AA-II} directions performs at the top in all but one problem.

\input{plots/logistic.tex}


\section{Conclusions}\label{sec:conclusions}
This paper investigated algorithms for solving inclusion problems involving the sum of a monotone and Lipschitz continuous operator and the subdifferential of a proper, convex, and lower semicontinuous function. We proposed a new Lyapunov function for Korpelevich’s extragradient method and established a last-iterate convergence result. Departing from the standard Fej\'er-type analysis, this Lyapunov-based optimality measure did not rely on a known solution to the inclusion problem. It underpinned three novel algorithms that extend the extragradient method. These algorithms balanced user-specified directions and standard extragradient steps, guided by carefully designed line search steps based on the new Lyapunov analysis. In addition to providing global convergence results under various assumptions, we showed that when the directions are superlinear, no backtracking is triggered, leading to superlinear convergence.

\new{Future research directions include developing solution-independent Lyapunov functions for other related methods. In particular, the forward-reflected-backward method of Malitsky and Tam~\cite{malitsky2020forwardbackwardsplitting} is a relevant case. Relatedly, a solution-independent Lyapunov function is already available for Malitsky’s projected reflected gradient method~\cite{malitsky2015projectedreflectedgradient} in \cite[Section~D]{cai2023accelerated}}. Another promising direction is to broaden the scope of the analysis beyond the monotone setting to include cohypomonotone operators~\cite{pennanen2002localconvergenceproximal}, and the more general class of problems characterized by the weak Minty condition~\cite{diakonikolas2021efficient,lee2021fast,pethick2022escaping}. Additionally, further exploration is warranted to adapt the approach to the mirror prox framework~\cite{nemirovski2004proxmethodrate}.


\begin{appendix}

\section{Background on Korpelevich's extragradient method}\label{app:background}
In the original paper~\cite{korpelevich1976extragradientmethodfinding}, the extragradient method~\eqref{alg:EG} was analyzed under the assumption that \(g\) is the indicator function of a nonempty, closed, and convex set, making the proximal operator reduce to the projection onto that set. However, as noted in~\cite{monteiro2011complexityvariantstsengs}, the extragradient method extends to the more general setting~\eqref{eq:the_inclusion_problem}. The remainder of this section presents results in this more general context, with proofs included for completeness.

\begin{definition}
    Suppose that \Cref{ass:main} holds and let \(\gamma\in\reals_{++}\). A point \(z\in\calH\) is said to be a fixed point of the extragradient method~\eqref{alg:EG} if
    \begin{subequations}\label{eq:EG:fixed_point}
        \begin{align}\label{eq:EG:fixed_point:1}
            \bar{z} &= \prox_{\gamma g}\p{z - \gamma F\p{z}}, \\ \label{eq:EG:fixed_point:2}
            z &=\prox_{\gamma g}\p{z - \gamma F\p{\bar{z} }}.
        \end{align}
    \end{subequations}
\end{definition}

\begin{proposition}
\label{prop:EG:fixed_point}
    Suppose that \Cref{ass:main} holds and let \(\gamma\in\reals_{++}\). Then, the following hold:
    \begin{propenumerate}
        \item If \(z\in \zer\p{F + \partial g}\), then \(z\) is a fixed point of the extragradient method, i.e.,~\eqref{eq:EG:fixed_point} holds, and \(z = \bar{z}\). \label{prop:EG:fixed_point:solutions_are_fixed_points}
        \item If \(\gamma \in (0,1/L_{F})\), \(z\) is a fixed point of the extragradient method, and \(\bar{z}\) is defined as in~\eqref{eq:EG:fixed_point:1}, then \(z=\bar{z} \in \zer\p{F + \partial g}\). \label{prop:EG:fixed_point:fixed_points_are_solutions}
    \end{propenumerate}
    \begin{proof}
        The proximal evaluations in~\eqref{eq:EG:fixed_point:1} and~\eqref{eq:EG:fixed_point:2} can equivalently be written via their subgradient characterization as
        \begin{subequations}\label{eq:EG:fixed_point:implicit_prox}
            \begin{align}\label{eq:EG:fixed_point:implicit_prox:1}
                \gamma^{-1}\p{z  - \bar{z}} -  F\p{z} &\in \partial g\p{\bar{z}}, \\\label{eq:EG:fixed_point:implicit_prox:2}
                -F\p{\bar{z}} &\in \partial g\p{z},
            \end{align}
        \end{subequations}
        respectively.
        \begin{itemize}
            \item[\ref{prop:EG:fixed_point:solutions_are_fixed_points}:] Note that \(z\in \zer\p{F + \partial g}\) and~\eqref{eq:EG:fixed_point:1} is equivalent to \(-F\p{z}\in \partial g \p{z}\) and~\eqref{eq:EG:fixed_point:implicit_prox:1}, respectively. Using monotonicity of \(\partial g\)~\cite[Theorem 20.48]{bauschke2017convexanalysismonotone},  we get that 
            \begin{align*}
                0   &\leq \inner{\gamma^{-1}\p{z  - \bar{z}} -  F\p{z} + F\p{z}}{\bar{z} - z } \\
                    &= - \gamma^{-1} \norm{z-\bar{z}}^{2} \leq 0,
            \end{align*}
            since \(\gamma\in\reals_{++}\). We conclude that \(z = \bar{z}\) and that~\eqref{eq:EG:fixed_point} holds.
            \item[\ref{prop:EG:fixed_point:fixed_points_are_solutions}:] By using monotonicity of \(\partial g\) at the points \(\bar{z}\) and  \(z\), and the corresponding subgradients in~\eqref{eq:EG:fixed_point:implicit_prox}, we get that 
            {\begin{allowdisplaybreaks}
            \begin{align}
                0   &\leq \inner{\gamma^{-1}\p{z  - \bar{z}} -  F\p{z} + F\p{\bar{z}}}{\bar{z} - z } \notag\\
                    &= - \gamma^{-1} \norm{z-\bar{z}}^{2} + \inner{F\p{\bar{z}} - F\p{z}}{\bar{z} - z} \notag\\
                    &\leq - \gamma^{-1} \norm{z-\bar{z}}^{2} + \norm{F\p{\bar{z}} - F\p{z}}\norm{\bar{z} - z} \notag\\
                    &\leq \p{L_{F}- \gamma^{-1}} \norm{z-\bar{z}}^{2}, \label{eq:EG:fixed_point:ineq}
            \end{align}
            \end{allowdisplaybreaks}}%
            where the Cauchy–Schwarz inequality is used in the second inequality, and Lipschitz continuity of \(F\) in the third inequality. Since \(L_{F}- \gamma^{-1}<0\), we conclude from~\eqref{eq:EG:fixed_point:ineq} that \(z = \bar{z}\). That \(z=\bar{z} \in \zer\p{F + \partial g}\) now follows from~\eqref{eq:EG:fixed_point:implicit_prox:1} or~\eqref{eq:EG:fixed_point:implicit_prox:2}. \qedhere
        \end{itemize}
    \end{proof}
\end{proposition}

\begin{proposition}
    \label{prop:EG:primal_descent}
    Suppose that \Cref{ass:main} holds, the sequence \(\p{\p{z^{k},\bar{z}^{k}}}_{k\in\naturals}\) is generated by~\eqref{alg:EG} with initial point \(z^0\in \calH\) and step-size parameter \(\gamma\in\reals_{++}    \), and \(z^{\star}\in\zer\p{F + \partial g}\). Then
    \begin{equation}\label{eq:EG:Lyapunov_with_solution}
        \norm{z^{k+1}-z^{\star}}^{2} \leq \norm{z^{k}-z^{\star}}^{2} 
        - \p{1-\gamma^{2} L_{F}^{2}}\norm{\bar{z}^{k}-z^{k}}^{2} 
    \end{equation}
    for each \(k\in\naturals\). Moreover, if \(\gamma\in\p{0,1/L_{F}}\), then \(\p{z^{k}}_{k\in\naturals}\) converges weakly to a point in \(\zer\p{F + \partial g}\).
    \begin{proof}
        Note that the first and second proximal evaluations in~\eqref{alg:EG} are equivalent to 
        \begin{align}\label{alg:EG:1:rewritten}
            0\leq g\p{z} - g\p{\bar{z}^{k}} - \inner{\gamma^{-1}\p{z^{k} - \bar{z}^{k}} - F\p{z^{k}}}{ z - \bar{z}^{k} } \text{ for each } z \in \calH,
        \end{align}
        and
        \begin{align}\label{alg:EG:2:rewritten}
            0\leq g\p{z} - g\p{z^{k+1}} - \inner{\gamma^{-1}\p{z^{k} - z^{k+1}} - F\p{\bar{z}^{k}}}{z - z^{k+1} } \text{ for each } z \in \calH,
        \end{align}
        respectively, and the assumption \(z^{\star}\in\zer\p{F + \partial g}\) is equivalent to
        \begin{align}\label{ass:Korpelevic:main:solution_exists:rewritten}
            0 \leq g\p{z} - g(z^{\star}) - \inner{ -F\p{z^{\star}}}{z - z^{\star} } \text{ for each } z \in \calH.
        \end{align}
        Picking \(z = z^{k+1}\) in~\eqref{alg:EG:1:rewritten}, \(z = z^{\star}\) in~\eqref{alg:EG:2:rewritten}, \(z = \bar{z}^{k}\) in~\eqref{ass:Korpelevic:main:solution_exists:rewritten}, summing the resulting inequalities, and multiplying by \(2\gamma\) gives 
        {\begin{allowdisplaybreaks}
        \begin{align*}
            0 &\leq
            \begin{aligned}[t]
                & 2\gamma g\p{z^{k+1}} - 2\gamma g\p{\bar{z}^{k}} - 2\inner{ z^{k} - \bar{z}^{k} - \gamma F\p{z^{k}}}{z^{k+1} - \bar{z}^{k} } \\
                & + 2\gamma g(z^{\star}) - 2\gamma g\p{z^{k+1}} - 2\inner{ z^{k} - z^{k+1} - \gamma F\p{\bar{z}^{k}}}{z^{\star} - z^{k+1} } \\
                & + 2\gamma g\p{\bar{z}^{k}} - 2\gamma g(z^{\star}) - 2\inner{ -\gamma F\p{z^{\star}}}{\bar{z}^{k} - z^{\star} } 
            \end{aligned}\\
            &= A_k + B_k,
        \end{align*} 
        \end{allowdisplaybreaks}}%
        where
        {\begin{allowdisplaybreaks}
        \begin{align*}
            A_k {}={}& - 2\inner{ z^{k} - \bar{z}^{k}}{z^{k+1} - \bar{z}^{k}}  - 2\inner{ z^{k} - z^{k+1}}{z^{\star} - z^{k+1} } \\
            {}={}& \norm{z^{k} -z^{k+1}}^{2} - \norm{z^{k} - \bar{z}^{k}}^{2} - \norm{z^{k+1} - \bar{z}^{k}}^{2} + \norm{z^{k} - z^{\star}}^{2} - \norm{z^{k} - z^{k+1}}^{2} \\
            {}{}&- \norm{z^{\star} - z^{k+1}}^{2} \\
            {}={}&  \norm{z^{k} - z^{\star}}^{2} - \norm{z^{\star} - z^{k+1}}^{2} - \norm{z^{k} - \bar{z}^{k}}^{2} - \norm{z^{k+1} - \bar{z}^{k}}^{2},
        \end{align*} 
        \end{allowdisplaybreaks}}%
        and 
        {\begin{allowdisplaybreaks}
        \begin{align*}
            B_k &= 2\gamma\inner{  F\p{z^{k}}}{z^{k+1} - \bar{z}^{k} } +  2\gamma\inner{ F\p{\bar{z}^{k}}}{z^{\star} - z^{k+1} }  - 2\gamma\inner{ - F\p{z^{\star}}}{\bar{z}^{k} - z^{\star} } \\
            &= 
            \begin{aligned}[t]
                & 2\gamma\inner{  F\p{z^{k}}}{z^{k+1} - \bar{z}^{k} } +  2\gamma\inner{ F\p{\bar{z}^{k}}}{z^{\star} - z^{k+1} } + 2\gamma\inner{F\p{\bar{z}^{k}}}{\bar{z}^{k} - z^{\star}} \\
                & - 2\gamma\inner{F\p{\bar{z}^{k}} - F\p{z^{\star}}}{\bar{z}^{k} - z^{\star} }
            \end{aligned}\\
            & \leq 2\gamma\inner{  F\p{z^{k}}}{z^{k+1} - \bar{z}^{k} } +  2\gamma\inner{ F\p{\bar{z}^{k}}}{z^{\star} - z^{k+1} } + 2\gamma\inner{F\p{\bar{z}^{k}}}{\bar{z}^{k} - z^{\star}} \\
            & = 2\gamma\inner{  F\p{z^{k}} - F\p{\bar{z}^{k}}}{z^{k+1} - \bar{z}^{k} } \\
            &\leq \gamma^{2}\norm{F\p{z^{k}} - F\p{\bar{z}^{k}}}^{2} + \norm{z^{k+1} - \bar{z}^{k}}^{2} \\
            &\leq \gamma^2L_{F}^{2}\norm{z^{k} - \bar{z}^{k}}^{2}  + \norm{z^{k+1} - \bar{z}^{k}}^{2},
        \end{align*} 
        \end{allowdisplaybreaks}}%
        where monotonicity of \(F\) is used in the first inequality, Young's inequality is used in the second inequality, and Lipschitz continuity of \(F\) in the third inequality. We conclude that 
        \begin{align*}
            0   &\leq A_k + B_k \leq  \norm{z^{k} - z^{\star}}^{2} - \norm{z^{\star} - z^{k+1}}^{2} - \p{1 - \gamma^{2}L_{F}^{2} } \norm{z^{k} - \bar{z}^{k}}^{2},
        \end{align*}
        which proves~\eqref{eq:EG:Lyapunov_with_solution}.
    
        Next, note that~\eqref{eq:EG:Lyapunov_with_solution} gives that \(\p{\norm{z^{k} - z^{\star}}}_{k\in\naturals}\) converges. Thus, \(\p{z^{k}}_{k\in\naturals}\) is bounded and there exists a subsequence \(\p{z^{k}}_{k\in K}\rightharpoonup z^{\infty}\) for some \(z^{\infty} \in \calH\)~\cite[Lemma 2.45]{bauschke2017convexanalysismonotone}. Moreover,~\eqref{eq:EG:Lyapunov_with_solution} and the requirement \(\gamma\in\p{0, 1/L_{F}}\) give that \(\p{\norm{\bar{z}^{k}-z^{k}}^{2}}_{k\in\naturals}\) is summable, and therefore, \(\p{\bar{z}^{k}}_{k\in K}\rightharpoonup z^{\infty}\). The first proximal evaluation in~\eqref{alg:EG} can equivalently be written as 
        \begin{align}\label{alg:EG:1:rewritten:seq}
            \gamma^{-1}\p{z^{k}- \bar{z}^{k} } - F\p{z^{k}} + F\p{\bar{z}^{k}} \in \p{F + \partial g}\p{\bar{z}^{k}}.
        \end{align}
        The left-hand side of~\eqref{alg:EG:1:rewritten:seq} converges strongly to zero since \(F\) is continuous and \((\norm{z^{k}-\bar{z}^{k}})_{k\in\naturals}\) converges to zero. Moreover, the operator \(F + \partial g\) is maximally monotone, since \(F\) is maximally monotone (by continuity and monotonicity~\cite[Corollary 20.28]{bauschke2017convexanalysismonotone}), \(\partial g\) is maximally monotone~\cite[Theorem 20.48]{bauschke2017convexanalysismonotone}, and \(F\) has full domain~\cite[Corollary 25.5]{bauschke2017convexanalysismonotone}. Thus,~\cite[Proposition 20.38]{bauschke2017convexanalysismonotone} gives that \(z^{\infty} \in \zer\p{F + \partial g}\), and by~\cite[Lemma 2.47]{bauschke2017convexanalysismonotone} we conclude that \(\p{z^{k}}_{k\in\naturals}\) converges weakly to a point in \(\zer\p{F + \partial g}\), as claimed. 
    \end{proof}
\end{proposition}

\new{
\begin{remark}
    Similar to \Cref{rmk:cyclically_monotone}, the results in this section remain valid when \(\partial g\) in \eqref{eq:the_inclusion_problem} is replaced with a maximally monotone and 3-cyclically monotone operator \(T:\calH\to\calH\) and the proximal operators \(\prox_{\gamma g}\) in \eqref{alg:EG} with the resolvent \(\p{\Id + \gamma T}^{-1}\).
\end{remark}
}

\section{Counterexamples}\label{app:counterexamples}
\begin{example}\label{ex:counterexample:tseng}
    Let \(\V_k = \V\p{z^{k},\bar{z}^{k},z^{k+1}}\) for the Lyapunov function \(\V\) defined in~\eqref{eq:V} and iterates \(\seq{\p{z^{k},\bar{z}^{k}}}\) generated by Tseng's method~\eqref{alg:Tseng}. This example contains a particular instance of the inclusion problem~\eqref{eq:the_inclusion_problem}, initial point \(z^{0}\in\calH\), and step size \(\gamma\in\p{0,1/L_F}\) for which \(\V_k\) increases between the first two consecutive iterations, thereby establishing that \(\V_k\) has no (one-step) descent inequality in this case. In particular, consider \(\calH=\reals^4\), \(F:\reals^4 \to \reals^4\), and \(g:\reals^4\to\reals\cup \{+\infty\}\) such that 
    \begin{align*}
        F(z) {}={} 
        \begin{bmatrix}
            A x \\
            - A^{\top} y
        \end{bmatrix}
        \quad 
        \text{ and }
        \quad
        g(z) {}={}
        \begin{cases}
            0 & \text{if } x \in [-7,6]^2 \text{ and } y \in [1,8]^2,\\
            +\infty & \text{otherwise}
        \end{cases}
    \end{align*}
    for each \(z = \p{x,y} \in \reals^2 \times \reals^2\), respectively, where   
    \begin{align*}
        A {}={}
        \begin{bmatrix}
            7 & 6 \\
            1 & 0
        \end{bmatrix}.
    \end{align*}
    It is straightforward to verify that \Cref{ass:main} holds with\footnote{The matrix norm is taken as the spectral norm.} 
    \begin{align*}
        L_{F} {}={} 
        \left\lVert 
        \begin{bmatrix}
            0 &  A \\
            - A^{\top} & 0
        \end{bmatrix}
        \right\rVert \approx 9.25091.
    \end{align*}
    By letting \(z^{0} = \p{-1, -7, -1, 7} \) and \(\gamma = {1}/{10}\), Tseng's method gives that 
    {\begin{allowdisplaybreaks}
    \begin{align*}
        \bar{z}^{0} &{}={} \left(-\frac{9}{2}, - \frac{69}{10}, 1, \frac{32}{5}\right), \\ 
        z^{1} &{}={} \left(-\frac{277}{50}, -\frac{71}{10}, -\frac{36}{25}, \frac{43}{10}\right), \\ 
        \bar{z}^{1} &{}={} \left(-7,-\frac{1739}{250}, 1, 1\right), 
    \end{align*}
    \end{allowdisplaybreaks}}%
    and therefore
    \begin{align*}
        \V_0 {}={} 1662 \quad \text{ and }\quad
        \V_1 {}={} \frac{1187246}{625} {}={} 1899.5936,
    \end{align*}
    establishing the claim.\footnote{Code to reproduce this example can be found at \url{https://github.com/ManuUpadhyaya/flex/blob/main/Counterexample_B1.ipynb}.} \qed
\end{example}

\begin{example}\label{ex:counterexample:full_eg}
    Consider the inclusion problem
    \begin{align*}
        \text{find}\  z\in\calH\ \text{ such that }\ 0 \in F(z) + T(z)
    \end{align*}
    where \(F:\calH\to\calH\) satisfies \Cref{ass:F} and \(T:\calH\to2^\calH\) is a maximally monotone operator. Let
    \begin{align}\label{alg:full_eg}
        \begin{split}
          \bar{z}^{k}
          &= \p{\Id+\gamma T}^{-1}\p{z^{k} - \gamma F(z^{k})}, 
          \\
          z^{k+1} 
          &= \p{\Id+\gamma T}^{-1}\p{z^{k} - \gamma F(\bar{z}^{k})},
        \end{split}
    \end{align}
    and \(\V_k = \V\p{z^{k},\bar{z}^{k},z^{k+1}}\) for the Lyapunov function \(\V\) defined in~\eqref{eq:V}. This example contains a particular problem instance for which \(\seq{z^{k}}\) diverges and \(\V_k\) increases between the first two consecutive iterations. In particular, consider \(\calH = \reals^2\), \(z^0 = (10, 10)\), \(\gamma = 1/10\), and \(F,T:\reals^2 \to \reals^2\) such that
    \begin{align*}
        F(z) {}={} 
        \underbracket{
        \begin{bmatrix}
            0 & 9 \\
            -9 & 0 
        \end{bmatrix}
        }_{=A}
        \begin{bmatrix}
            x \\
            y
        \end{bmatrix}
        \quad \text{ and } \quad
        T(z) {}={}
        \underbracket{
        \begin{bmatrix}
            0 & -4 \\
            4 & 0 
        \end{bmatrix}
        }_{=B}
        \begin{bmatrix}
            x \\
            y
        \end{bmatrix}
    \end{align*}
    for each \(z = \p{x,y} \in \reals\times\reals\), where \(L_F = 9\). Note that~\eqref{alg:full_eg} reduces to
    \begin{align*}
        z^{k+1} = \underbracket{\left(I+\gamma B\right)^{-1}\left(I - \gamma A \left( \left(I+\gamma B\right)^{-1} \left( I - \gamma A\right) \right) \right)}_{=C} z^{k},
    \end{align*}
    where \(C\in\reals^{2\times 2}\)  has full rank and spectral radius \(\approx 1.132596\), which is greater than one. Therefore, we can conclude that \(\seq{z^{k}}\) diverges. Moreover,~\eqref{alg:full_eg} gives that
    {\begin{allowdisplaybreaks}
    \begin{align*}
        \bar{z}^0 &{}={} \left(\frac{215}{29}, \frac{465}{29}\right), \\
        z^1 &{}={} \left(\frac{3245}{1682}, \frac{26745}{1682}\right), \\
        \bar{z}^1 &{}={} \left(-\frac{447965}{97556}, \frac{1899785}{97556}\right), \\
        z^2 &{}={} \left(-\frac{53118995}{5658248}, \frac{87834005}{5658248}\right),
    \end{align*}
    \end{allowdisplaybreaks}}%
    and therefore
    {\begin{allowdisplaybreaks}
    \begin{align*}
        \V_0 &{}={} \frac{5875000}{841} {}\approx{} 6985.73127229489, \\
        \V_1 &{}={} \frac{12676046875}{1414562} {}\approx{} 8961.11084208398,
    \end{align*}
    \end{allowdisplaybreaks}}%
    establishing the second claim.\footnote{Code to reproduce this example can be found at \url{https://github.com/ManuUpadhyaya/flex/blob/main/Counterexample_B2.ipynb}.} \qed
\end{example}


\new{%
\section{Comparison to standard optimality measures}\label{sec:comparison}
This section presents some standard optimality measures for \eqref{eq:the_inclusion_problem} and compares them to \(\V_k = \V\p{z^{k},\bar{z}^{k},z^{k+1}}\) for the Lyapunov function \(\V\) defined in~\eqref{eq:V} and iterates \(\seq{\p{z^{k},\bar{z}^{k}}}\) generated by the extragradient method~\eqref{alg:EG}. It also includes a comparison to recent last-iterate convergence results for the extragradient method.

\begin{definition}\label{def:performance_measures}
    Suppose that \Cref{ass:main} holds.
    \begin{enumerate}[label=(\roman*)]
        \item The \emph{natural residual} is defined as
        \begin{align*}
            \norm{z - \prox_{g}(z - F(z))}
        \end{align*}
        for each \(z\in \calH\).
        \item The \emph{tangent residual} is defined as
        \begin{align*}
            \inf_{\xi \in \partial g(z)} \norm{F(z)+ \xi}
        \end{align*}
        for each \(z\in \calH\).
        \item Suppose that \(z^{0}\in\calH\), \(z^{\star} \in \mathrm{zer}(F + \partial g)\) and \(\delta = \|z^0 - z^{\star}\|>0\). Then the restricted gap function \(\textup{\texttt{Gap}}:\calH\to\reals\cup\set{+\infty}\) is defined as
        \begin{align}\label{eq:gap_functions}
            \textup{\texttt{Gap}}(z) = \sup_{w \in \mathrm{dom}\,\partial g \cap \mathbb{B}(z^{\star};\delta)} \left(\inner{F(w)}{z-w} + g(z) - g(w)\right)
        \end{align}
        for each \(z\in \calH\), where \(\mathbb{B}(z^{\star};\delta) = \{ z \in\mathcal{H} : \|z - z^{\star}\| \leq \delta \} \).
    \end{enumerate}
\end{definition}

\begin{remark}
    In this remark, we establish that the measures in \Cref{def:performance_measures} are indeed optimality measures for \eqref{eq:the_inclusion_problem}.
    \begin{enumerate}[label=(\roman*)]
        \item The natural residual is nonnegative and zero if and only if \(z\in\zer\p{F + \partial g}\), since 
        \begin{align*}
            \norm{z - \prox_{g}(z - F(z))} = 0 \; \iff \; z = \prox_{g}(z - F(z)) \; \iff \; -F(z) \in \partial g(z),
        \end{align*}
        where the last equivalence follows from the subgradient characterization of the proximal operator. 
        \item The tangent residual is nonnegative and zero if and only if \(z\in\zer\p{F + \partial g}\), since the tangent residual upper bounds the natural residual by \Cref{prop:natres_upper_tanres}.
        \item  It is well-known that (e.g., see \cite[Lemma 1]{nesterov2006dualextrapolationits} and \cite[Lemma 3]{malitsky2020goldenratioalgorithms})
        \begin{itemize}
            \item \(\textup{\texttt{Gap}}(z) \geq 0\) for each \(z\in\mathcal{H}\), 
            \item if \(z\in\mathrm{zer}(F + \partial g)\cap \mathbb{B}(z^{\star};\delta)\), then \(\textup{\texttt{Gap}}(z) = 0\), and
            \item if $z \in \mathrm{dom}\,\partial g \cap \mathbb{B}(z^{\star};\tilde{\delta})$ for some \(\tilde{\delta} \in (0,\delta)\), then \(\textup{\texttt{Gap}}(z) = 0\) implies that \(z\in\mathrm{zer}(F + \partial g)\).
        \end{itemize}
    \end{enumerate}
\end{remark}

Note that the second proximal step in~\eqref{alg:EG} can equivalently be written via its subgradient characterization as
\begin{equation}\label{eq:comparison:subgradient}
     \underbracket{\gamma^{-1}\p{z^{k}-z^{k+1}}-F(\bar{z}^{k})}_{=\xi^{k+1}} \in \partial g(z^{k+1})
\end{equation}
for each \(k\in\naturals\). This notation allows us to introduce another standard optimality measure, namely \(\p{\norm{F(z^{k}) + \xi^{k}}}_{k\in\mathbb{N}}\). It is clear that this measure upper bounds the tangent residual, i.e.,
\begin{align}\label{eq:tanres_upper_tseng_res}
    \inf_{\xi \in \partial g(z^{k})} \|F(z^{k})+ \xi\| \leq \norm{F(z^{k}) + \xi^{k}}
\end{align}
for each \(k\in\mathbb{N}\).

Next, we present a result that, when combined with \eqref{eq:tanres_upper_tseng_res}, shows that \(\V_k\) upper bounds all the squared optimality measures considered above, except the squared restricted gap function, which is upper bounded up to a positive constant. 

\begin{proposition}\label{prop:comp:bounds}
    Suppose that \Cref{ass:main} holds, the sequence \(\p{\p{z^{k},\bar{z}^{k}}}_{k\in\naturals}\) is generated by~\eqref{alg:EG} with initial point \(z^0\in \calH\) and step-size parameter \(\gamma\in(0,1/L_{F})\), the sequence \(\p{\mathcal{V}_{k}}_{k\in\naturals}\) is given by \(\V_k =\mathcal{V}(z^k, \bar z^k, z^{k+1})\) for each \(k\in \naturals\) and the Lyapunov function \(\V\) defined in~\eqref{eq:V}, and the sequence \(\p{\xi^{k}}_{k\in\mathbb{N}}\) is given by \eqref{eq:comparison:subgradient}. Then,
    \begin{propenumerate}
        \item \label{prop:tsengres_upper_V}\(\norm{F(z^{k+1}) + \xi^{k+1}}^2 \leq \V_{k}\) for each \(k \in \naturals\), 
        \item \label{prop:natres_upper_tanres} \(\norm{z - \prox_{g}(z - F(z))} \leq \inf_{\xi \in \partial g(z)} \norm{F(z)+ \xi}\) for each \(z \in \calH\), 
        \item \label{prop:gap_upper_tanres} \(\textup{\texttt{Gap}}(z)\leq\p{\delta + \norm{z-z^{\star}}}\inf_{\xi \in \partial g(z)} \norm{F(z)+ \xi}\) for each \(z \in \calH\), and in particular, \(\textup{\texttt{Gap}}(z^{k})\leq2\delta \inf_{\xi \in \partial g(z^{k})} \norm{F(z^{k})+ \xi}\) for each \(k\in\mathbb{N}\), where \(\textup{\texttt{Gap}}\) is defined in \eqref{eq:gap_functions}.
    \end{propenumerate}
\end{proposition}
\begin{proof}
    The proofs of \Cref{prop:natres_upper_tanres,prop:gap_upper_tanres} are simple generalizations of corresponding results found in \cite{sedlmayer2023fastoptimistic}, which we include for completeness. 
    \begin{itemize}
        \item[\ref{prop:tsengres_upper_V}]
        Note that 
        \begin{align*}
             & \|F(z^{k+1})+ \xi^{k+1}\|^2  \\
             & = \|\tfrac{1}{\gamma}\big(z^{k}-z^{k+1}\big)+F(z^{k+1})-F(\bar{z}^{k})\|^{2}
             \\ 
             & =
             \|F(z^{k+1})-F(\bar{z}^{k})\|^{2}+\tfrac{1}{\gamma^{2}}\|z^{k}-z^{k+1}\|^{2}+\tfrac{2}{\gamma}\langle F(z^{k+1})-F(\bar{z}^{k}),z^{k}-z^{k+1}\rangle
             \\ 
             & \leq
             L_{F}^{2}\|z^{k+1}-\bar{z}^{k}\|^{2}+\tfrac{1}{\gamma^{2}}\|z^{k}-z^{k+1}\|^{2}+\tfrac{2}{\gamma}\langle F(z^{k+1})-F(\bar{z}^{k}),z^{k}-z^{k+1}\rangle
             \\ 
             & \leq
             \tfrac1{\gamma^{2}}\|z^{k+1}-\bar{z}^{k}\|^{2}+\tfrac{1}{\gamma^{2}}\|z^{k}-z^{k+1}\|^{2}+\tfrac{2}{\gamma}\langle F(z^{k})-F(\bar{z}^{k}),z^{k}-z^{k+1}\rangle
             \\ 
             & = \mathcal V_k
        \end{align*} 
        where Lipschitz continuity of \(F\) was used in the first inequality,  and \(\gamma L_{F} \leq 1\) and monotonicity of \(F\) was used in the second inequality.
        \item[\ref{prop:natres_upper_tanres}]
        We claim that the natural residual is upper bounded by the tangent residual, i.e., 
        \begin{align}\label{eq:natural_leq_tangent}
            \|z - \mathrm{prox}_{g}(z - F(z))\| \leq \inf_{\xi \in \partial g(z)} \|F(z)+ \xi\|
        \end{align}
        for each \(z\in\calH\).
        When \(z \notin \mathrm{dom}\,\partial g\), then \eqref{eq:natural_leq_tangent} holds trivially. Thus, suppose that \(z \in \mathrm{dom}\,\partial g\) and let \(\xi \in \partial g (z)\). The latter inclusion holds if and only if \(z = \mathrm{prox}_{g}(z + \xi)\). Therefore, 
        \begin{align*}
            \|z - \mathrm{prox}_{g}(z - F(z))\| 
            = \| \mathrm{prox}_{g}(z + \xi) - \mathrm{prox}_{g}(z - F(z))\| 
            \leq \|F(z) + \xi \|
        \end{align*}
        where the inequality follows from the nonexpansivity of the proximal operator \cite[Proposition 12.28]{bauschke2017convexanalysismonotone}. 
        However, since \(\xi \in \partial g (z)\) is arbitrary, \eqref{eq:natural_leq_tangent} follows. 
        \item[\ref{prop:gap_upper_tanres}]
        We claim that the restricted gap function is upper bounded by the tangent residual up to a positive quantity, i.e., 
        \begin{align}\label{eq:restricted_gap_function_leq_tangent_residual}
            \texttt{Gap}(z) \leq \p{\delta + \norm{z-z^{\star}}} \inf_{\xi \in \partial g(z)} \|F(z)+ \xi\|
        \end{align}
        for each \(z\in\calH\), and therefore,
        \begin{align}\label{eq:restricted_gap_function_leq_tangent_residual:k}
            \textup{\texttt{Gap}}(z^{k}) \leq 2\delta \inf_{\xi \in \partial g(z^{k})} \|F(z^{k})+ \xi\|
        \end{align}
        for each \(k\in \mathbb{N}\), by \Cref{prop:EG:primal_descent}.
        Let us prove \eqref{eq:restricted_gap_function_leq_tangent_residual}. If $z \notin \mathrm{dom}\,\partial g$, then \eqref{eq:restricted_gap_function_leq_tangent_residual} holds trivially. Thus, assume that $z \in \mathrm{dom}\,\partial g$, and let \(\xi \in \partial g(z)\) and \(w\in \mathrm{dom}\,\partial g \cap \mathbb{B}(z^{\star};\delta) \). Then we have that 
        \begin{align}
            &\inner{F(w)}{z-w} + g(z) - g(w) \nonumber\\
            & = \underbracket{\inner{F(w) - F(z)}{z-w}}_{\leq 0} + \inner{F(z)}{z-w} + \underbracket{g(z) - g(w)}_{\leq \inner{\xi}{z - w}} \nonumber\\
            &\leq \inner{F(z)+\xi}{z-z^{\star}} + \inner{F(z)+\xi}{z^{\star}-w} \nonumber\\
            &\leq \norm{F(z)+\xi}\norm{{z-z^{\star}}} + \norm{F(z)+\xi}\underbracket{\norm{z^{\star}-w}}_{\leq \delta} \nonumber\\
            &\leq \p{\delta + \norm{z-z^{\star}}} \|F(z)+ \xi\|
            \label{eq:restricted_gap_function_leq_tangent_residual:helper},
        \end{align}
        Maximizing over \(w\in \mathrm{dom}\,\partial g \cap \mathbb{B}(z^{\star};\delta) \) and minimizing over \(\xi \in \partial g(z)\) in \eqref{eq:restricted_gap_function_leq_tangent_residual:helper} gives \eqref{eq:restricted_gap_function_leq_tangent_residual}, as claimed. \qedhere
    \end{itemize}
\end{proof}

\begin{corollary}\label{cor:rates_standard_measures}
    Suppose that \Cref{ass:main} and \(\zer\p{F + \partial g}\neq \emptyset\) hold, the sequence \(\p{\p{z^{k},\bar{z}^{k}}}_{k\in\naturals}\) is generated by~\eqref{alg:EG} with initial point \(z^0\in \calH\) and step-size parameter \(\gamma\in(0,1/L_{F})\), and the sequence \(\p{\xi^{k}}_{k\in\mathbb{N}}\) is given by \eqref{eq:comparison:subgradient}. Then,
    \begin{align*}
        \norm{F(z^{k}) + \xi^{k}} & \in o\p{1/\sqrt{k}} \text{ as } k \to \infty,\\
        \inf_{\xi \in \partial g(z^{k})} \|F(z^{k})+ \xi\| & \in o\p{1/\sqrt{k}} \text{ as } k \to \infty, \\
        \norm{z^k - \prox_{g}(z - F(z^k))} & \in o\p{1/\sqrt{k}} \text{ as } k \to \infty, \\
        \textup{\texttt{Gap}}(z^{k}) & \in o\p{1/\sqrt{k}} \text{ as } k \to \infty,
    \end{align*}
    and for any \(k\in\mathbb{N}\) and \(z^{\star}\in\zer\p{F + \partial g}\) it holds that  
    \begin{gather*}
        \norm{z^k - \prox_{g}(z - F(z^k))} 
        \leq 
        \inf_{\xi \in \partial g(z^{k})} \|F(z^{k})+ \xi\| 
        \leq 
        \norm{F(z^{k}) + \xi^{k}} 
        \leq 
        \frac{\norm{z^{0} - z^{\star} } }{ \sqrt{\alpha\p{ \gamma, L_{F}} k} }, \\
        \textup{\texttt{Gap}}(z^{k})
        \leq 
        \frac{2\delta\norm{z^{0} - z^{\star} } }{ \sqrt{\alpha\p{ \gamma, L_{F}} k} },
    \end{gather*}
    where \(\alpha\p{ \gamma, L_{F}}=\tfrac{\gamma^{2}}{2}\p{\sqrt{5-4\gamma^2 L_{F}^2 } - 1}>0\) and \(\textup{\texttt{Gap}}\) is defined in \eqref{eq:gap_functions}.
    \begin{proof}
        Follows immediately from \eqref{eq:tanres_upper_tseng_res}, \Cref{prop:comp:bounds}, and \Cref{cor:lastiter:T}.
    \end{proof}
\end{corollary}

\begin{remark}\label{rmk:last-iterate_rates}
    In the items listed below, we compare the recent last-iterate rates in \cite[Theorem 3]{cai2022tightlastiterateconvergenceextragradient} and \cite[Corollary 4.1(b)]{trandinh2024revisitingextragradienttype} with \Cref{cor:rates_standard_measures}. In the following, assume that \Cref{ass:main} holds, the sequence \(\p{\p{z^{k},\bar{z}^{k}}}_{k\in\naturals}\) is generated by~\eqref{alg:EG} with initial point \(z^0\in \calH\) and step-size parameter \(\gamma\in(0,1/L_{F})\), and the sequence \(\p{\xi^{k}}_{k\in\mathbb{N}}\) is given by \eqref{eq:comparison:subgradient}.
    \begin{enumerate}
        \item \cite[Theorem 3]{cai2022tightlastiterateconvergenceextragradient} proves that 
        \begin{align}\label{eq:rate_cai}
            \inf_{\xi \in \partial g(z^{k})} \|F(z^{k})+ \xi\| \leq \frac{3 \norm{z^{0} - z^{\star} } }{ \gamma \sqrt{\p{1-\gamma^{} L_{F}^{2}} k }  }
        \end{align}
        for each \(k\in\mathbb{N}\) and \(z^{\star}\in\zer\p{F + \partial g}\), in the particular case when \(g\) is the indicator function of a closed, convex, and nonempty set. Simple computation shows that the corresponding rate in \Cref{cor:rates_standard_measures} sharpens \eqref{eq:rate_cai}, i.e.,
        \begin{align}\label{eq:rate_cai:comparison}
            \frac{1 }{ \sqrt{\alpha\p{ \gamma, L_{F}}} } < \frac{3}{ \gamma \sqrt{1-\gamma^{2} L_{F}^{2} }  }
        \end{align}
        for any \(\gamma \in (0,1/L_{F})\). For convenience, we plot the quotient 
        \begin{align}\label{eq:rate_cai:comparison:quotient}
            \left.\frac{1 }{ \sqrt{\alpha\p{ \gamma, L_{F}}} } \middle/  \frac{3}{ \gamma \sqrt{1-\gamma^{2} L_{F}^{2}}} \right.
        \end{align}
        in \Cref{fig:rate_quotient:a}, proving \eqref{eq:rate_cai:comparison} graphically.

        \item \cite[Corollary 4.1(b)]{trandinh2024revisitingextragradienttype} proves that 
        \begin{align}\label{eq:rate_trandinh}
            \|F(z^{k})+ \xi^{k}\| \leq \frac{\sqrt{3+2\gamma^2L_{F}^{2}} \norm{z^{0} - z^{\star} } }{ \gamma \sqrt{\p{1-\gamma L_{F}} k }  }
        \end{align}
        for each \(k\in\mathbb{N}\) and \(z^{\star}\in\zer\p{F + \partial g}\). Simple computation shows that the corresponding rate in \Cref{cor:rates_standard_measures} sharpens \eqref{eq:rate_trandinh}, i.e.,
        \begin{align}\label{eq:rate_trandinh:comparison}
            \frac{1 }{ \sqrt{\alpha\p{ \gamma, L_{F}}} } < \frac{\sqrt{3+2\gamma^2L_{F}^{2}} }{ \gamma \sqrt{1-\gamma L_{F}}}
        \end{align}
        for any \(\gamma \in (0,1/L_{F})\). For convenience, we plot the quotient 
        \begin{align}\label{eq:rate_trandinh:comparison:quotient}
            \left.\frac{1 }{ \sqrt{\alpha\p{ \gamma, L_{F}}} } \middle/  \frac{\sqrt{3+2\gamma^2L_{F}^{2}} }{ \gamma \sqrt{1-\gamma L_{F}  }   } \right.
        \end{align}
        in \Cref{fig:rate_quotient:b}, proving \eqref{eq:rate_trandinh:comparison} graphically.
    \end{enumerate}

\end{remark}

\begin{figure}[!t]
  \centering

  \begin{subfigure}{0.99\linewidth}
    \centering
    \begin{tikzpicture}
      \pgfmathdeclarefunction{q}{1}{%
        \pgfmathparse{(1.0/3.0)*sqrt( 2*(1 - #1*#1) / (sqrt(5 - 4*#1*#1) - 1) )}%
      }

      \begin{axis}[
        width=\linewidth,
        height=0.5\linewidth,
        domain=0.001:0.999,            
        xmin=0, xmax=1,                
        ymin=0.3, ymax=0.5, 
        samples=1200,
        ymode=log,
        log ticks with fixed point,
        axis lines=left,
        xlabel={$\gamma L_F$},
        ylabel={},
        tick align=outside,
        tick style={black},
        xtick={0,0.1,...,1},
        minor x tick num=9,
        grid=both,
        grid style={dashed,black!35},
        minor grid style={dotted,black!25},
        legend style={
          draw=none,
          fill=white, fill opacity=0.9, text opacity=1,
          font=\small,
          at={(0.98,0.98)}, anchor=north east,
          row sep=8pt, inner sep=8pt,
          cells={anchor=west}
        },
        legend image post style={xscale=1.8},
        every axis plot/.append style={color=black, mark=none, line cap=round, line join=round},
        clip=false
      ]

        \addplot[ultra thick, restrict x to domain=0.001:0.999] ({x}, {q(x)});
        \addlegendentry{$\displaystyle
          \left.\frac{1}{\sqrt{\alpha(\gamma,L_F)}} \middle/ \frac{3}{\gamma\sqrt{1-\gamma^{2}L_F^{2}}}\right.
          \;=\; \frac{1}{3}\sqrt{\frac{2\bigl(1-\gamma^{2} L_F^{2}\bigr)}{\sqrt{5-4\gamma^{2} L_F^{2}}-1}}$}

      \end{axis}
    \end{tikzpicture}
    \caption{
    The ratio between the coefficients of the last-iterate rate in \Cref{cor:rates_standard_measures} and that of \cite[Theorem 3]{cai2022tightlastiterateconvergenceextragradient}.
    }
    \label{fig:rate_quotient:a}
  \end{subfigure}

  \vspace{0.6em}

  \begin{subfigure}{0.99\linewidth}
    \centering
    \begin{tikzpicture}
      \pgfmathdeclarefunction{qb}{1}{%
        \pgfmathparse{ sqrt( 2*(1-#1) / ((sqrt(5-4*#1*#1)-1)*(3+2*#1*#1)) ) }%
      }

      \begin{axis}[
        width=\linewidth,
        height=0.5\linewidth,
        domain=0.001:0.999,            
        xmin=0, xmax=1,                
        ymin=0.2, ymax=1.3,
        samples=1200,
        ymode=log,
        log ticks with fixed point,
        axis lines=left,
        xlabel={$\gamma L_F$},
        ylabel={},
        tick align=outside,
        tick style={black},
        xtick={0,0.1,...,1},
        minor x tick num=9,
        grid=both,
        grid style={dashed,black!35},
        minor grid style={dotted,black!25},
        legend style={
          draw=none,
          fill=white, fill opacity=0.9, text opacity=1,
          font=\small,
          at={(0.98,0.98)}, anchor=north east,
          row sep=8pt, inner sep=8pt,
          cells={anchor=west}
        },
        legend image post style={xscale=1.8},
        every axis plot/.append style={color=black, mark=none, line cap=round, line join=round},
        clip=false
      ]

        \addplot[ultra thick, restrict x to domain=0.001:0.999] ({x}, {qb(x)});
        \addlegendentry{$\displaystyle
          \left.\frac{1}{\sqrt{\alpha(\gamma,L_F)}} \middle/ \frac{\sqrt{3+2\gamma^{2}L_F^{2}}}{\gamma\sqrt{1-\gamma L_F}}\right.
          \;=\; \sqrt{\frac{2\bigl(1-\gamma L_F\bigr)}{\bigl(\sqrt{5-4\gamma^{2} L_F^{2}}-1\bigr)\,\bigl(3+2\gamma^{2} L_F^{2}\bigr)}}$}

      \end{axis}
    \end{tikzpicture}
    \caption{
    The ratio between the coefficients of the last-iterate rate in \Cref{cor:rates_standard_measures} and that of \cite[Corollary 4.1(b)]{trandinh2024revisitingextragradienttype}.
    }
    \label{fig:rate_quotient:b}
  \end{subfigure}

  \caption{}
  \label{fig:rate_quotients}
\end{figure}

\begin{remark}
    Similar to \Cref{rmk:cyclically_monotone}, the results in this section remain valid (except the ones involving the restricted gap function \(\textup{\texttt{Gap}}\)) when \(\partial g\) in \eqref{eq:the_inclusion_problem} is replaced with a maximally monotone and 3-cyclically monotone operator \(T:\calH\to\calH\) and the proximal operators \(\prox_{\gamma g}\) in \eqref{alg:EG} with the resolvent \(\p{\Id + \gamma T}^{-1}\).
\end{remark}
}%

\end{appendix}

\section*{Acknowledgments}
M. Upadhyaya and P. Giselsson acknowledge support from the ELLIIT Strategic Research Area and the Wallenberg AI, Autonomous Systems, and Software Program (WASP), funded by the Knut and Alice Wallenberg Foundation. Additionally, P. Giselsson acknowledges support from the Swedish Research Council. P. Latafat is a member of the Gruppo Nazionale per l'Analisi Matematica, la Probabilit\`a e le loro Applicazioni (GNAMPA - National Group for Mathematical Analysis, Probability and their Applications) of the Istituto Nazionale di Alta Matematica (INdAM - National Institute of Higher Mathematics). We thank Yurong Chen for highlighting some details in~\Cref{sec:equilibrium}. \new{We also thank Max Nilsson, Sebastian Banert, and the anonymous reviewers for providing valuable comments and suggestions for a previous version of this work.}


    \phantomsection
    \addcontentsline{toc}{section}{References}
        \bibliographystyle{abbrvurl}
    \bibliography{TeX/mybib.bib}

\appendix

\end{document}